\documentclass[11pt, reqno]{amsart}
\usepackage{amssymb, amsthm, amsmath, amsfonts}
\usepackage{graphics}
\usepackage{hyperref}
\usepackage{bbm}
\usepackage[all]{xy}
\usepackage{enumerate}
\usepackage[mathscr]{eucal}
\usepackage{enumerate}
\usepackage{thmtools}
\usepackage{xcolor}
\usepackage{ytableau}
\usepackage{tikz}

\theoremstyle{plain}
\newtheorem{theorem}{Theorem}[section]
\newtheorem{lemma}[theorem]{Lemma}
\newtheorem{proposition}[theorem]{Proposition}

\newtheorem{corollary}[theorem]{Corollary}
\numberwithin{equation}{section}

\theoremstyle{definition}

\newtheorem{definition}[theorem]{Definition}
\newtheorem{example}[theorem]{Example}
\newtheorem{remark}[theorem]{Remark}

\newcommand{\Mod}{{\textrm{-}\mathrm{Mod}}}

\newcommand{\uc}{\mathrm{uc}}
\newcommand{\con}{\mathrm{con}}

\DeclareMathOperator{\Ob}{Ob}
\DeclareMathOperator{\Aut}{Aut}
\DeclareMathOperator{\Sym}{Sym}

\DeclareMathOperator{\Hom}{Hom}
\DeclareMathOperator{\End}{End}

\DeclareMathOperator{\Ext}{Ext}

\DeclareMathOperator{\Res}{Res}

\DeclareMathOperator{\id}{id}

\DeclareMathOperator{\sgn}{sgn}
\DeclareMathOperator{\Irr}{Irr}

\newcommand{\FA}{{\mathrm{FA}}}
\newcommand{\CA}{{\mathrm{CA}}}
\newcommand{\SA}{{\mathrm{SA}}}
\newcommand{\BA}{{\mathrm{BA}}}
\newcommand{\OA}{{\mathrm{OA}}}
\newcommand{\VA}{{\mathrm{VA}}}
\newcommand{\FI}{{\mathrm{FI}}}
\newcommand{\VI}{{\mathrm{VI}}}
\newcommand{\C}{{\mathscr{C}}}
\newcommand{\D}{{\mathscr{D}}}

\newcommand{\op}{{\mathrm{op}}}
\newcommand{\tor}{{\mathrm{tor}}}
\newcommand{\Sh}{{\mathrm{Sh}}}

\title[Representations of relational structures and endomorphism monoids]{Representations of categories of finite relational structures and associated endomorphism monoids}

\author{Liping Li}
\address{School of Mathematics and Statistics, Hunan Normal University, Changsha 410081, China.}
\email{lipingli@hunnu.edu.cn}
\thanks{L. Li was partly supported by NSFC Grant No. 12571037.}
\keywords{Representations of categories, Highly homogeneous structures, Dold–Kan correspondence, Generalized Reedy categories, Infinite endomorphism monoids, Continuous representations.}

\begin{document}

\begin{abstract}
We develop a unified representation theory for the categories of finite subsets and relation-preserving maps of highly homogeneous relational structures classified by Cameron. For any commutative coefficient ring $k$, we extend the classical Dold–Kan correspondence to this setting—with the sole exception of the category $\FA$—and prove that finitely generated representations are noetherian (resp., artinian) when $k$ is noetherian (resp., artinian).

When $k$ is a field, we obtain a precise structural description of these representation categories. We classify irreducible representations, showing that canonical quotients of indecomposable standard modules are either irreducible (the regular case) or has length 2 (the singular case). In the case that $k$ has characteristic 0, we establish a (direct sum or triangular) decomposition into a singular component governed by classical Dold–Kan theory and a regular component exhibiting semisimplicity or representation stability.

Finally, we establish a monoidal generalization of Artin’s reconstruction theorem for topological groups, proving an equivalence between uniformly continuous representations of infinite topological transformation monoids and sheaves on the associated categories of finite subsets. In the special case where the transformation monoid is a permutation group, our result recovers Artin’s theorem.
\end{abstract}

\maketitle

\section{Introduction}

\subsection{Motivation}

Representation categories of combinatorial and diagrammatic origin play a central role in modern algebra, topology, and geometry. Prominent examples include simplicial modules arising from the simplex category and representations of the category $\FI$ of finite sets and injections, which are foundational in representation stability theory. A recurring phenomenon in these settings is that remarkably rigid and structured representation-theoretic behavior—manifesting as homological stability, finiteness properties, or categorical decompositions—emerges from very elementary combinatorial input.

A classical instance of this phenomenon is the Dold--Kan correspondence, which identifies simplicial modules with nonnegatively graded chain complexes \cite{Dold, WeibelHA}. From a categorical perspective, this correspondence reflects a decomposition of the linearized simplex category into a \emph{singular} part, governing homological behavior (corresponding to normalized chains), and a complementary \emph{regular} part, which is homologically trivial. Extensions of this philosophy to other structures, most notably the cyclic category $\Lambda$, have been explored by Connes, Dwyer, Kan, Pirashvili, and many others, demonstrating that analogous normalization procedures recover mixed complexes and cyclic homology \cite{Connes94, DK85, Kassel, KaygunKaya, Pirashvili99}. Several systematic categorical frameworks have also been developed, including crossed simplicial groups introduced by Fiedorowicz--Loday \cite{FL91} and generalized Reedy categories by Berger--Moerdijk \cite{BergerMoerdijk2011}. Yet all of these theories are fundamentally homological: they are primarily designed to support homotopy theory, model structures, and derived constructions \cite{CisinskiMoerdijk, GoerssJardine, Helmstutler, Hirschhorn, Kassel, Loday92, LQ, Marker, Quillen}. A representation-theoretic counterpart—namely, a systematic theory of module categories and a structural classification of special representations—remains largely absent.

In a different direction, the representation theory of categories of finite sets with structured injective morphisms has developed rapidly in recent years, particularly through the theory of $\FI$-modules introduced by Church--Ellenberg--Farb \cite{CEF} and further developed by many authors; see for instances \cite{CEFN, Gadish, GL2015, GL2019, GLX, GS2023, LR, MPW, NagelRoemer, PutmanSam, RandalWilliamsWahl, SamSnowden2017} and the surveys \cite{RollandWilson, Sam2020}. These works establish deep structural results, including noetherianity, local self-injectivity, and connections to representation stability, polynomial functors and functor homology \cite{DV2010, DV2015}. By contrast, categories in which injective and surjective morphisms coexist symmetrically remain largely outside the scope of current structural theories. Notable exceptions include the category $\FA$ of finite sets and maps and the category $\VA_q$ of finite dimension vector spaces over a finite field $\mathbb{F}_q$ and linear maps. Their representation theory has been investigated by Kuhn, Powell, Vespa, and Wiltshire-Gordon \cite{Kuhn, PowellVespa2023, Powell, WGordon}. Moreover, representation theory of $\FA$ has been applied to study holomorphic forms on moduli spaces of stable curves by Canning--Larson--Payne--Willwacher \cite{CLPW}.

Our aim is to develop a representation theory for combinatorial categories in which injective and surjective morphisms are treated on equal footing. A natural source of such categories arises from categories of finite subsets and relation-preserving maps associated with oligomorphic group actions \cite{Ca90}. On one hand, these categories include many important examples previously studied in \cite{BergerMoerdijk2011, FL91}; on the other hand, this framework provides a structural bridge between representations of small categories and continuous representations of their associated automorphism groups or endomorphism monoids, in the sense of categorical reconstruction \cite{DLLX, LiPermutation}.

As a first step, we focus on the class of highly homogeneous relational structures, classified by Cameron \cite{Ca76, Ca90}. Cameron showed that there exist exactly five highly homogeneous relations on a countable set. Each such relation yields a skeletal category of finite subsets and relation-preserving maps, giving rise to the five categories
\[
\OA,\ \CA,\ \BA,\ \SA,\ \FA,
\]
where $\OA$ is the simplex category, $\CA$ is the cyclic category, $\BA$ is the category of finite subsets with maps preserving linear betweenness relations, and $\SA$ is the category of finite sets with maps preserving separation relations; for details, see \cite[Section 13]{BMMN}.

The purpose of this paper is to initiate a systematic study of representations of these five categories, addressing several conceptual gaps. First, existing results are largely fragmentary and rely on case-by-case combinatorial arguments, with no uniform framework treating all five categories simultaneously. Second, the appearance of simplicial phenomena—such as the Dold--Kan correspondence—has not been explained intrinsically from the ambient categorical structure. Third, the relationship between representations of these combinatorial categories and representations of associated infinite transformation monoids has not been clarified at a structural level.

\subsection{Notations}

Before stating the main results, we introduce the basic notation used throughout the paper. Unless otherwise specified, let $\C$ be the skeleton of one of the five categories described above, with objects $[n] = \{1, \ldots, n\}$ for $n \in \mathbb{Z}_+$. Denote by $\C^+$ (resp., $\C^-$) the wide subcategory of injective (resp., surjective) morphisms. For each $n \in \mathbb{Z}_+$, let $G_n$ be the automorphism group of the object $[n]$.

We verify that $\C$ is a generalized Reedy category in the sense of \cite{BergerMoerdijk2011}: every morphism $f: [m] \to [n]$ in $\C$ admits a decomposition
\[
f = f^+ \circ f^-,
\]
where $f^-: [m] \twoheadrightarrow [l]$ is a surjective morphism in $\C^-$ and $f^+: [l] \hookrightarrow [n]$ is an injective morphism in $\C^+$. This decomposition is unique up to the action of automorphisms in $G_l$. Moreover, $G_l$ acts freely on the set of surjective morphisms $\C^-(m, l)$ from the left and on the set of injective morphisms $\C^+(l, n)$ from the right.

Fix a commutative ring $k$. A \emph{representation} of $\C$ (or a \emph{$\C$-module}) $V$ is a covariant functor from $\C$ to $k \Mod$, the category of $k$-modules. Denote by $P_n$ the representable functor $k\C(n, -)$. We define the \emph{standard module} $\Delta_n$ as the quotient of $P_n$ by the submodule spanned by all non-injective morphisms with source $[n]$, an analogue of the classical notion in algebraic representation theory: Dalezois-{\v{S}}t'ov{\'\i}{\v{c}}ek \cite{DS2025} define $k$-linear Reedy categories with $k$ a field, which can be viewed infinite quasi-hereditary algebras; inspired by their work, Di-Li-Liang \cite{DLL} define generalized $k$-linear Reedy categories for arbitrary commutative rings $k$, which can be viewed as infinite standardly stratified algebras. The generalized Reedy structure ensures that $P_n$ admits a finite filtration with standard modules as its factors. In the case where $k$ is a field and $\lambda$ is an irreducible $kG_n$-module, we define the \emph{indecomposable standard module}
\[
\Delta_{n, \lambda} = \Delta_n \otimes_{kG_n} kG_n e_{\lambda},
\]
where $e_{\lambda} \in kG_n$ is the primitive idempotent such that $kG_n e_{\lambda}$ is the projective cover of $\lambda$. Correspondingly, we define
\[
\overline{\Delta}_{n, \lambda} = \Delta_n \otimes_{kG_n} \lambda,
\]
which is a quotient of $\Delta_{n, \lambda}$, and is isomorphic to $\Delta_{n, \lambda}$ when $k$ has characteristic 0.

Finally, note that the simplex category $\OA$ is a wide subcategory of $\C$. It possesses a normalization operator defined by the ordered product
\[
\Psi_n = (\id_n - d_{n-1}s_{n-1}) (\id_n - d_{n-2}s_{n-2}) \cdots (\id_n - d_1 s_1),
\]
which is an idempotent in the endomorphism algebra $k\OA(n, n)$, where $d_i$ and $s_i$ denote the $i$-th face and degeneracy maps, respectively; see \cite{DK85, Pirashvili99}. Since $k\OA(n, n)$ is a subalgebra of $k\C(n, n)$, $\Psi_n$ is also an idempotent in the category algebra $k\C$.

\subsection{Main results}

Our first main result establishes an extended Dold--Kan correspondence for the category $\C$ (excluding $\FA$), realized via a canonical normalization procedure. Specifically, we show that the set
\[
\{ k\C \Psi_n \mid n \in \mathbb{Z}_+ \}
\]
of normalized projective modules forms a system of projective generators of $\C\Mod$. Each normalized projective either coincides with the standard module $\Delta_n$ (for $\OA$ and $\BA$) or admits a filtration by two standard modules (for $\CA$ and $\SA$); see Proposition \ref{filtration of projectives}. Consequently, the representation category of $\C$ can be described in terms of a normalized differential-type structure, generalizing the classical Dold--Kan normalization from simplicial modules to representations of the four Cameron categories except $\FA$.

\begin{theorem}[Theorem \ref{generalized D-K correspondence}] \label{first main result}
Let $k$ be a commutative ring and let $\C \in \{\OA,\CA,\BA,\SA\}$. Then $\C \Mod$ is equivalent to the representation category of the $k$-linear category $\mathscr{K}$ generated by:
\begin{itemize}
\item automorphisms $\sigma \in G_n = \Aut_{\C}([n])$,
\item normalized cyclic face morphisms $\mathbf d_{n+1}: n \to n+1$,
\item and, in the cases $\C=\CA,\SA$, normalized degeneracy morphisms
$\mathbf s_{n+1}: n+1 \to n$.
\end{itemize}
The automorphism groups $G_n$ form a crossed simplicial group in the sense of Fiedorowicz--Loday \cite{FL91}: elements of $G_n$ permute the face and degeneracy maps; in $\mathscr K$ where only the top face and degeneracy survive, their images are stable under this induced action.

\medskip

For $\OA$ and $\BA$, $\mathscr{K}$ is the $k$-linear category associated to the quiver
\[
\xymatrix{
1 \ar@(ul,ur)[]^{G_1} \ar[r]^{\mathbf{d}_2} &
2 \ar@(ul,ur)[]^{G_2} \ar[r]^{\mathbf{d}_3} &
3 \ar@(ul,ur)[]^{G_3} \ar[r]^{\mathbf{d}_4} &
\cdots
}
\]
modulo the relations $\mathbf d_{n+1} \circ \mathbf d_n = 0$ for all $n \geqslant 1$, where $G_n$ is either trivial or cyclic of order $2$.

\medskip

For $\CA$ and $\SA$, $\mathscr{K}$ is the $k$-linear category associated to the quiver
\[
\xymatrix{
1 \ar@(ul,ur)[]^{G_1} \ar[r]^{\mathbf{d}_2} &
2 \ar@(ul,ur)[]^{G_2} \ar@/^1pc/[r]^{\mathbf{d}_3} &
3 \ar@(ul,ur)[]^{G_3} \ar@/^1pc/[r]^{\mathbf{d}_4} \ar@/^1pc/[l]^{\mathbf{s}_3} &
\cdots \ar@/^1pc/[l]^{\mathbf{s}_4}
}
\]
modulo the relations:
\begin{itemize}
\item \emph{Complex relations:} $\mathbf d_{n+1}\circ \mathbf d_n = 0$ and $\mathbf s_n \circ \mathbf s_{n+1} = 0$;
\item \emph{Normalization:} $\mathbf s_{n+1}\circ \mathbf d_{n+1} = \Psi_n$;
\item \emph{Cyclic idempotents:} $\mathbf d_n \circ \mathbf s_n = \delta_n$ and $\delta_n^2=\delta_n$;
\item \emph{Retraction laws:} $\delta_n \mathbf d_n = \mathbf d_n$ and $\mathbf s_n \delta_n = \mathbf s_n$;
\item \emph{Orthogonality:} $\mathbf d_n \delta_{n-1} = 0$ and $\delta_{n-1}\mathbf s_n=0$.
\end{itemize}
\end{theorem}

\begin{remark}
While this result is classical for the simplex category and known for the cyclic category \cite[Corollary 6.6]{DK85}, to our knowledge this unified treatment is new for the categories $\BA$ and $\SA$. In contrast to the result by Kaygun--Kaya \cite{KaygunKaya}, which applies normalization procedures to crossed simplicial groups within a model-categorical context, our work focuses on the representation theory of these categories.
\end{remark}

As a consequence of this structural equivalence, we obtain the following finiteness results, which hold for all five categories arising from Cameron’s classification, including $\FA$.

\begin{corollary}[Corollary \ref{finite length 1}]
If $k$ is noetherian, then every finitely generated $\C$-module is noetherian. If $k$ is artinian, then every finitely generated $\C$-module has finite length.
\end{corollary}

\begin{remark}
This result is well known for $\OA$ and can be deduced for $\CA$ from \cite[Corollary 6.6]{DK85}. For $\FA$, the noetherian fact is established in \cite[Corollary 7.3.5]{SamSnowden2017}, and the finite length result has been established in \cite[Corollary 4.33]{Powell} and \cite[Theorem 5.3]{WGordon} under the assumption that $k$ is a field.
\end{remark}

While the established equivalence with the normalized category $\mathscr{K}$ holds over any commutative ring and reveals the underlying structural framework of these representation categories, the intrinsic complexity of the resulting quiver and its relations can obscure finer representation-theoretic features. To overcome this limitation and access more delicate structural information, we turn to the intrinsic combinatorial architecture of the categories—specifically their generalized Reedy structure. This leads to the introduction of a new combinatorial invariant, the \emph{mutation graph}, which governs the interaction between injective and surjective morphisms.

\begin{definition}
Given an injective morphism $f: [m] \hookrightarrow [n]$ in $\C$, a \emph{mutation} of $f$ is another injective morphism $g: [m] \hookrightarrow [n]$ such that:
\begin{itemize}
\item $f(i) \neq g(i)$ for exactly one $i \in [m]$;
\item there exists a surjective morphism $h: [n] \twoheadrightarrow [n-1]$ with $h \circ f = h \circ g$.
\end{itemize}
The \emph{mutation graph} $\Gamma_{m, n}$ has vertices all injective morphisms $[m]\hookrightarrow[n]$, with an edge between $f$ and $g$ whenever $g$ is a mutation of $f$.
\end{definition}

Mutation graphs are connected for $\OA, \CA,$ and $\FA$, and have at most two connected components for $\BA$ and $\SA$. When $n = m+1$, each component is bipartite. These combinatorial results lead to an explicit computation of Hom-spaces between standard modules.

\begin{theorem}[Theorem \ref{hom spaces}]
Let $k$ be a commutative ring. For all $m,n \geqslant 1$, one has
\[
\Hom_{k\C}(\Delta_m,\Delta_n) \;\cong\;
\begin{cases}
kG_n, & m = n,\\[2mm]
k^d, & m = n+1,\\[1mm]
0, & \text{otherwise,}
\end{cases}
\]
where $d$ is the number of connected components of $\Gamma_{n, n+1}$. In particular, $d=1$ for $\FA,\, \OA,\, \CA$ and $d=2$ for $\BA,\, \SA$ when $n \geqslant 2$.
\end{theorem}

This explicit computation has several consequences. For instance, when $k$ is a field, it implies that every standard module for $\OA$ and $\BA$ is projective, yielding an alternative proof of the Dold--Kan correspondence that does not rely on normalization techniques; see Corollary \ref{OA and BA}.

When $k$ is a field, the structure becomes more refined. In this case, each standard module admits a canonical decomposition
\[
\Delta_n \cong \bigoplus_{\lambda \in \Irr(G_n)} (\Delta_{n,\lambda})^{\oplus a_{\lambda}},
\]
where $\Irr(G_n)$ is a complete set of non-isomorphic irreducible $kG_n$-modules. We call $\lambda$ \emph{singular} if it appears in the classification of Proposition~\ref{classify singular}, that is, there exists a nonzero homomorphism from $\Delta_{n+1}$ to $\Delta_{n, \lambda}$. Otherwise, $\lambda$ is called \emph{regular}. More explicitly, $\lambda \in \Irr(G_n)$ is singular if and only if it is isomorphic to the one-dimensional space spanned by the alternating sum of the vertices in a connected component of the mutation graph $\Gamma_{n-1, n}$.

These simple combinatorial observations yield a sharp structural dichotomy.

\begin{theorem}[Propositions \ref{regular standard is simple} and \ref{length 2}]
Let $k$ be a field. Then $\overline{\Delta}_{n, \lambda}$ is irreducible if $\lambda$ is regular. If $\lambda$ is singular, then $\overline{\Delta}_{n, \lambda}$ has length $2$.
\end{theorem}

As a consequence, we obtain a transparent classification of irreducible $\C$-modules.

\begin{corollary}[Theorem \ref{classification of irreducibles}]
Let $k$ be a field. Isomorphism classes of irreducible $\C$-modules are parameterized by pairs $(n,\lambda)$ with $n \in \mathbb{Z}_+$ and $\lambda \in \mathrm{Irr}(G_n)$. If $\lambda$ is regular, the corresponding irreducible $\C$-module is $\overline{\Delta}_{n,\lambda}$; if $\lambda$ is singular, it is the top of $\overline{\Delta}_{n,\lambda}$.
\end{corollary}

\begin{remark}
For the category $\OA$, this classification of irreducible modules is classical. For $\FA$, the above classification has been established via different methods under the assumption that $k$ has characteristic 0; see \cite[Theorem 1.4]{Powell}, \cite[Theorem 4.1]{Rains}, or \cite[Theorem 5.5]{WGordon}. Our approach extends these classifications to fields of arbitrary characteristic.
\end{remark}

Using the regular--singular dichotomy, we obtain a refined form of the Dold--Kan correspondence in the case that $k$ is a field of characteristic 0. More precisely, we construct a $k$-linear category $\mathscr{K}$ whose representation category is equivalent to $\C\Mod$, but whose structure is substantially simpler than the category appearing in Theorem~\ref{first main result}. For $\BA$, the category $\mathscr{K}$ is essentially the direct sum of two copies of the $k$-linear category arising in the classical Dold--Kan correspondence for $\OA$. For the remaining categories, the structure exhibits new and genuinely different behavior. Specifically, one has:
\begin{itemize}
\item If $\C=\CA$ or $\SA$, then every regular indecomposable standard module is projective (this is true for any field; see Proposition \ref{projective cover})  and irreducible. As a consequence, we have a decomposition
\[
\mathscr{K} = \mathscr{K}^{\mathrm{reg}} \oplus \mathscr{K}^{\mathrm{sing}},
\]
reflecting the regular--singular dichotomy of indecomposable standard modules. The regular part $\mathscr{K}^{\mathrm{reg}}$ is semisimple, while the singular part $\mathscr{K}^{\mathrm{sing}}$ admits an explicit quiver description generalizing the classical Dold--Kan category. In particular, the homological structure of $\C\Mod$ is governed entirely by $\mathscr{K}^{\mathrm{sing}}$. For a precise statement, see Theorem \ref{DK for CA and SA}.

\item If $\C = \FA$, the situation is reversed: every singular indecomposable standard module is projective. In this case, the above direct sum decomposition is no longer true, but we still have a triangular decomposition: there are no nonzero morphisms from $\mathscr{K}^{\mathrm{sing}}$ to $\mathscr{K}^{\mathrm{reg}}$. The singular part $\mathscr{K}^{\mathrm{sing}}$ coincides with the classical Dold--Kan category for $\OA$, while $\mathscr{K}^{\mathrm{reg}}$ is an inverse category whose endomorphism algebras of objects are one-dimensional.  For a precise statement, see Theorem~\ref{DK for FA}.
\end{itemize}

Our final main result places the representation theory of $\C$ into a genuinely infinite-dimensional and structural framework. Explicitly, let $\mathscr{E} \leqslant \Omega^{\Omega}$ be a transformation monoid. We equip $\mathscr{E}$ with the topology of pointwise convergence. A $k\mathscr{E}$-module $M$ is said to be \emph{continuous} if the action of $\mathscr{E}$ on $M$ is continuous with respect to the discrete topology on $M$; it is \emph{uniformly continuous} if the action of $\mathscr{E}$ is uniformly continuous with respect to the natural uniform structure on $\mathscr{E}$ induced by pointwise convergence.

Let $\C$ be the category of finite subsets of $\Omega$ whose morphisms are restrictions of elements in $\mathscr{E}$. A $\C$-module $V$ is said to be \emph{torsion} if for every $A \in \Ob(\C)$ and every $v \in V(A)$, there is an inclusion $\iota_{A, B}: A \hookrightarrow B$ such that $\iota_{A, B} \cdot v = 0$. We also define the Grothendieck topology $J$ on $\C^{\op}$ as follows: for every $A \in \Ob(\C)$, a sieve is a covering sieve in $J(A)$ if it contains an inclusion morphism with source $A$.

\begin{theorem}[Theorem~\ref{equivalence 2} and Corollary \ref{equivalence 3}]
Let $k$ be a commutative ring, $\mathscr{E}$ and $\C$ be as defined above. Then one has the following equivalences of categories
\[
k\mathscr{E} \Mod^{\mathrm{uc}} \simeq \C \Mod / \C \Mod^{\tor} \simeq \Sh(\C^{\op}, \, J, \, \underline{k})
\]
where $k\mathscr{E} \Mod^{\mathrm{uc}}$ is the category of uniformly continuous $k\mathscr{E}$-modules, $\C \Mod^{\tor}$ is the category of torsion $\C$-modules, and $\Sh(\C^{\op}, \, J, \, \underline{k})$ is the sheaf category over the constant structure sheaf $\underline{k}$.

\medskip

In particular, if every inclusion morphism $\iota_{A,B}: A \hookrightarrow B$ in $\C$ admits a retraction given by a surjective morphism $p: B \twoheadrightarrow A$, then $J$ is the trivial topology, and hence
\[
\C\Mod \simeq k\mathscr{E} \Mod^{\uc}.
\]
\end{theorem}

This equivalence provides a systematic framework connecting three different perspectives: representations of combinatorial categories, uniformly continuous representations of infinite transformation monoids, and their interpretation via sheaf-theoretic localization. In this sense, it may be viewed as a monoidal analogue of Artin's classical reconstruction theorem for topological groups (see \cite[III.9]{MM}), extending the reconstruction principle from invertible symmetries to general transformation monoids. The equivalence therefore allows one to translate structural and representation-theoretic questions about infinite transformation monoids into problems about representations of categories $\C$ of finite subsets of $\Omega$, where powerful algebraic techniques are available. We discuss this relationship in more detail in Subsection~\ref{Artin's theorem}.

If the category $\C$ satisfies the additional assumption in the second statement of the theorem, the equivalence simplifies considerably. In this case it is realized directly by the functor
\[
\mathcal{R}(V)=\varinjlim_{A\in\Ob(\C)} V(A),
\]
where the colimit is taken over the poset of finite subsets ordered by inclusion. Under this assumption, no nonzero $\C$-module is torsion with respect to inclusions, so the localization step disappears and the correspondence becomes completely canonical. This condition is satisfied by the five combinatorial categories arising from highly homogeneous relations (when $\mathscr{E}$ is taken to be the monoid of maps preserving the underlying defining relations on $\Omega$). Consequently, the theorem allows one to apply the rich theory of representations of these combinatorial categories to the study of uniformly continuous representations of the corresponding infinite transformation monoids.

A striking feature of the theory is that very different monoids can have identical uniformly continuous representation theory. Explicitly, for any highly homogeneous relation on $\Omega$, the automorphism group shares the same uniformly continuous representation theory as the monoid of injective relation-preserving maps; in this case, every continuous representation of the monoid is automatically uniformly continuous. Similarly, the full transformation monoid consisting of all relation-preserving maps and the submonoid of surjective such maps still share the same uniformly continuous representations, although continuous representations of these monoids may fail to be uniformly continuous. These examples illustrate that, despite substantial combinatorial differences, their uniformly continuous representation theory can be indistinguishable. For details, see Subsection \ref{Application II}.

Beyond permutation structures, our framework also applies to infinite-dimensional algebraic settings. A natural example is the endomorphism monoid $\mathscr{E} = \End_{\mathbb{F}_q} (\Omega)$ of a countably infinite-dimensional vector space $\Omega$ over a finite field $\mathbb{F}_q$. We show that the category of uniformly continuous representations of $\mathscr{E}$ is equivalent to the well-studied category of $\VA_q$-modules, where $\VA_q$ denotes the category of finite-dimensional $\mathbb{F}_q$-vector spaces and linear maps. In particular, when $k$ is a field of characteristic 0, this equivalence as well as a result of Kuhn \cite{Kuhn} tells us that the category of uniformly continuous representations of $k\mathscr{E}$ is semisimple.

\subsection{Organization and conventions}

The paper is organized as follows. In Section 2, we study categories of finite subsets and relation-preserving maps and prove that they admit generalized Reedy structures in full generality. Section 3 introduces mutation graphs, a key combinatorial tool used to compute Hom- and Ext-spaces between standard modules. In Section 4, we establish the generalized Dold--Kan correspondence and the equivalence with the normalized category $\mathscr{K}$ for all Cameron categories except $\FA$. Section 5 is devoted to explicit computations of Hom-spaces between standard modules and a classification of irreducible modules. In Section 6, we present further representation-theoretic results for $\CA$, $\SA$, and $\FA$, including the refined Dold--Kan correspondence. Finally, in Section 7, we investigate the relationship between uniformly continuous representations of endomorphism monoids and representations of the associated categories of finite subsets, establishing a monoidal extension of Artin’s theorem.

\medskip

All functors (resp., modules) in this paper are covariant (resp., left modules) unless otherwise specified. Morphism composition is read from right to left. For $n \in \mathbb{Z}_+$, we set $[n] = \{1, \dots, n\}$, excluding $0$.

\section{Reedy Structure}

Throughout this paper, let $\Omega$ be an infinite set equipped with an \emph{underlying relation} $\mathcal{R}_{\Omega}$. Let $\C$ be a skeletal full subcategory of the category whose objects are finite subsets of $\Omega$ equipped with relations inherited from $\mathcal{R}_{\Omega}$, and whose morphisms are relation-preserving maps. Thus every finite subset of $\Omega$ is isomorphic to a unique object in $\C$. Denote by $\C^+$ the wide subcategory consisting of injective morphisms, and by $\C^-$ the wide subcategory consisting of surjective morphisms.

We will show that $\C$ is a generalized Reedy category in the sense of Berger--Moerdijk \cite{BergerMoerdijk2011}. For an object $A\in\C$, denote by $G_A$ its automorphism group; equivalently,
\[
G_A = \C^+(A, A) = \C^-(A, A).
\]

\begin{lemma}\label{free action}
For any objects $A, B \in \Ob(\C)$, one has:
\begin{enumerate}
\item the right action of $G_{B}$ on $\C^+(B,A)$ is free;
\item the left action of $G_{A}$ on $\C^-(B,A)$ is free.
\end{enumerate}
\end{lemma}

\begin{proof}
We only prove (1) since (2) is a dual statement. Let $h: B \hookrightarrow A$ be an injective morphism and let $\sigma \in G_B$ satisfy $h \circ \sigma = h$. For any $x \in B$ we have $h(\sigma(x)) = h(x)$. Since $h$ is injective on underlying sets, it follows that $\sigma(x) = x$. Hence $\sigma = \id_B$, and the right action is free.
\end{proof}

\begin{lemma}\label{morphism factorization}
For every morphism $f: B \to A$ in $\C$, there exist
\begin{itemize}
\item a unique object $C$ of $\C$,
\item a surjective morphism $f^- \in \C^-(B, C)$,
\item an injective morphism $f^+ \in \C^+(C, A)$,
\end{itemize}
such that $f = f^+ \circ f^-$.

Furthermore, this factorization is unique up to a unique automorphism of $C$: if $f = g^+ \circ g^-$ is another such factorization with $g^- \in \C^-(B, C)$ and $g^+ \in \C^+(C, A)$, then there exists a unique $\sigma \in G_C$ such that the diagram
\[
\xymatrix{
 & C \ar[d]^-{\sigma} \ar[dr]^-{f^+} \\
B \ar[r]_-{g^-} \ar[ur]^-{f^-} & C \ar[r]_-{g^+} & A
}
\]
commutes.
\end{lemma}

\begin{proof}
We first show the existence of the factorization. Define an equivalence relation $\sim$ on the underlying set of $B$ by
\[
x \sim y \quad \Longleftrightarrow \quad f(x) = f(y).
\]
Let $C = B/\sim$ be the set of equivalence classes. We equip $C$ with a relational structure as follows: for any relation symbol $R$ and any $[x_1], \dots, [x_n] \in C$, define
\[
R_C([x_1], \dots, [x_n]) \quad \Longleftrightarrow \quad R_A(f(x_1), \dots, f(x_n)).
\]
This is well-defined, since if $x_i \sim x_i'$ for all $i$, then $f(x_i) = f(x_i')$, so the right-hand side does not depend on the choice of representatives.

Without loss of generality, we may assume that $C$ is an object of $\C$. Let $f^-: B \twoheadrightarrow C$ be the canonical projection. By construction, $f^-$ is surjective. Moreover, if $R_B(x_1,\dots,x_n)$ holds in $B$, then since $f$ is relation-preserving, $R_A(f(x_1),\dots,f(x_n))$ holds in $A$, and hence $R_C([x_1],\dots,[x_n])$ holds in $C$. Thus $f^-$ is relation-preserving, so $f^- \in \C^-$.

Next, define $f^+: C \to A$ by $f^+([x]) = f(x)$. This is well-defined by definition of $\sim$, and by construction we have $f = f^+ \circ f^-$. It is clear that $f^+$ is injective. Furthermore, if $R_C([x_1],\dots,[x_n])$ holds, then by definition $R_A(f(x_1),\dots,f(x_n))$ holds, so $R_A(f^+([x_1]),\dots,f^+([x_n]))$ holds as well. Thus $f^+$ is relation-preserving, so $f^+ \in \C^+$. This gives the required factorization.

\medskip
For uniqueness, suppose $f = g^+ \circ g^-$ is another factorization. Then for $x,y \in B$,
\[
g^-(x) = g^-(y) \quad \Longrightarrow \quad f(x) = f(y),
\]
so $\ker(g^-) = \sim$. By the universal property, there exists a unique map $\sigma: C \to C$ such that
\[
g^- = \sigma \circ f^-.
\]
Since $g^- = \sigma \circ f^-$ and $f^-$ is surjective, it follows that $\sigma$ is relation-preserving and surjective; from
\[
f = f^+ \circ f^- = g^+ \circ g^- = g^+ \circ \sigma \circ f^-
\]
and the surjectivity of $f^-$, we deduce that $f^+ = g^+ \circ \sigma$ and the injectivity of $\sigma$. Thus $\sigma$ is contained in $G_C$ and makes the diagram commute.
\end{proof}

Recall from \cite[Definitions~1.1]{BergerMoerdijk2011} that a small skeletal category $\mathcal{C}$ is called a \emph{generalized Reedy category} with respect to a degree map $d: \Ob(\mathcal{C}) \to \lambda$, where $\lambda$ is an ordinal, if it satisfies the following conditions:
\begin{enumerate}
\item $\mathcal{C}$ has a wide subcategory $\mathcal{C}^+$ such that $\mathcal{C}^+(x,y) \neq \varnothing$ only if $x=y$ or $d(x)<d(y)$;

\item $\mathcal{C}$ has another wide subcategory $\mathcal{C}^-$ such that $\mathcal{C}^-(x,y) \neq \varnothing$ only if $x=y$ or $d(x)>d(y)$;

\item for every $x \in \Ob(\mathcal{C})$, one has $\mathcal{C}^+(x,x) = \mathcal{C}^-(x,x) = G_x$, the automorphism group of $x$;

\item every morphism $f$ in $\mathcal{C}$ admits a factorization $f = f^+ \circ f^-$ with $f^+$ in $\mathcal{C}^+$ and $f^-$ in $\mathcal{C}^-$, and this factorization is unique up to automorphisms;

\item for all $x,y\in \Ob(\mathcal{C})$, the automorphism group $G_y$ acts freely on $\mathcal{C}^-(x,y)$ from the left.
\end{enumerate}

\medskip
From Lemma~\ref{free action} and Lemma~\ref{morphism factorization}, we immediately obtain:

\begin{proposition}\label{Reedy}
The category $\C$ is a generalized Reedy category with respect to the degree map
\[
d: \Ob(\C) \to \mathbb{Z}_+, \qquad A \mapsto |A|.
\]
\end{proposition}

Let $k$ be a commutative ring. By definition, a \emph{$\C$-module} is a covariant functor from $\C$ to $k\Mod$. Denote by $\C \Mod$ the category of all $\C$-modules. It is a Grothendieck category and in particular has enough projective objects. Equivalently, one may form the associative algebra $k\C$, called the \emph{category algebra} of $\C$, defined in the same spirit as the path algebra of a quiver. A $\C$-module can then be identified with a left $k\C$-module graded by the set $\Ob(\C)$ of objects.

Given an object $A$ of $\C$, define the following $\C$-modules:
\begin{itemize}
\item $P_A = k\C(A,-)$, the representable functor;

\item $\mathfrak{I}_{A}$, the submodule of $P_A$ spanned by all \emph{non-injective} morphisms with source $A$;

\item $\Delta_{A} = P_{A}/\mathfrak{I}_{A}$.
\end{itemize}

We call $\Delta_{A}$ the \emph{standard module} associated to $A$. By construction, $\Delta_{A}$ admits a $k$-basis parameterized by all injective morphisms starting from $A$.

We collect some fundamental properties of these modules in the following proposition, which follow from Lemma~\ref{free action}, Proposition~\ref{Reedy}, and \cite[Lemmas~2.10 and~3.4, and Theorem~3.6]{DLL}. Set $\C^0 = \C^+ \cap \C^-$, the wide subcategory of $\C$ consisting of isomorphisms.

\begin{proposition}\label{standard modules}
With notation as above, the following hold:
\begin{enumerate}
\item $k\C$ is a left projective $\C^+$-module and a right projective $\C^-$-module;

\item there is a decomposition of $(\C^0, \C^0)$-bimodules
\[
\mathfrak{I}_{A} \cong \bigoplus_{|B|<|A|} k\C^+(B, -) \otimes_{kG_{B}} k\C^-(A, B);
\]

\item $P_{A}$ admits a finite filtration whose successive quotients are standard modules $\Delta_{B}$ with $|B| \leqslant |A|$, and in which $\Delta_{A}$ occurs with multiplicity one;

\item there is a natural isomorphism
\[
\Delta_{A} \cong k\C \otimes_{k\C^-} kG_{A},
\]
where $kG_{A}$ is viewed as a left $\C^-$-module concentrated on the object $A$;

\item $\End_{k\C}(\Delta_{A}) \cong kG_{A}$;

\item $\Hom_{k\C}(\Delta_{A}, \, \Delta_{B}) \neq 0 \Longrightarrow |A| \geqslant |B|$;

\item $\Ext_{k\C}^1(\Delta_{A}, \, \Delta_{B}) \neq 0 \Longrightarrow |A|>|B|$.
\end{enumerate}
\end{proposition}

\section{Highly homogeneous structures and mutation graphs}

Recall that the action of a permutation group $G \leqslant \Sym(\Omega)$ on $\Omega$ is \emph{highly homogeneous} if, for every $n \in \mathbb{Z}_+$, the set of $n$-element subsets of $\Omega$ forms a single $G$-orbit. By the fundamental classification theorem of Cameron, there are, up to equivalence, exactly five highly homogeneous relations $\mathcal{R}_{\Omega}$ on $\Omega$ induced by such group actions: the empty relation (pure set), the dense linear order, the dense cyclic order, the dense linear betweenness relation, and the dense separation relation; see \cite[Section 13]{BMMN}.

Accordingly, a skeletal category $\C$ of finite subsets and relation-preserving maps can be described explicitly as follows:

\begin{center}
\begin{tabular}{c|c|c|c}
relation & category & objects & morphisms\\
  \hline
empty & $\FA$ & $[n]$, $n \in \mathbb{Z}_+$ & all maps\\
linear order & $\OA$ & $[n]$, $n \in \mathbb{Z}_+$ & order-preserving maps\\
cyclic order & $\CA$ & $[n]$, $n \in \mathbb{Z}_+$ & cyclic order-preserving maps\\
linear betweenness & $\BA$ & $[n]$, $n \in \mathbb{Z}_+$ & betweenness-preserving maps\\
separation & $\SA$ & $[n]$, $n \in \mathbb{Z}_+$ & separation-preserving maps\\
\end{tabular}
\end{center}

In this section we fix $\C$ to be one of the above five categories unless otherwise specified, and introduce a key combinatorial property of the morphism set of $\C$.

\begin{definition}
Let $f,f': [m]\hookrightarrow[n]$ be two morphisms in $\C^+$. We say that $f'$ is a \emph{mutation} of $f$ if the following conditions hold:
\begin{itemize}
\item there exists a unique element $x \in [m]$ such that $f(x) \neq f'(x)$;
\item there exists a surjective morphism $g: [n] \twoheadrightarrow [n-1]$ such that $g\circ f = g\circ f'$.
\end{itemize}

For $m,n \in \mathbb{Z}_+$, the \emph{mutation graph} $\Gamma_{m, n}$ is the graph whose vertices are injective morphisms $[m] \hookrightarrow [n]$, and where two vertices are joined by an edge if one is a mutation of the other.
\end{definition}

\begin{remark} \label{mutation}
By the definition, the surjective morphism $g$ must collapse $f(x)$ and $f'(x)$. Furthermore, the composite $g \circ f = g \circ f'$ is injective.
\end{remark}

The following example illustrates the above construction.

\begin{example}
Let $\C = \FA$. Then $\Gamma_{2, 3}$ is depicted as below, where we denote an injection $f$ by the tuple $(f(1), f(2))$:
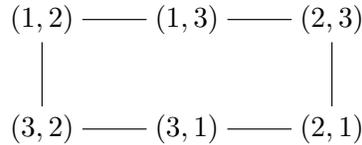
\[
\xymatrix{
(1, 2) \ar@{-}[r] \ar@{-}[d] & (1, 3) \ar@{-}[r] & (2, 3) \ar@{-}[d]\\
(3, 2) \ar@{-}[r] & (3, 1) \ar@{-}[r] & (2, 1)
}
\]

\medskip

Let $\C = \OA$, in which morphisms are order-preserving maps. Then $\Gamma_{2, 4}$ is depicted as below:
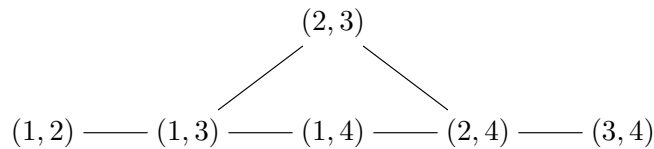
\[
\xymatrix{
 & & (2, 3) \ar@{-}[dr] \ar@{-}[dl] & &\\
(1, 2) \ar@{-}[r] & (1, 3) \ar@{-}[r] & (1, 4) \ar@{-}[r] & (2, 4) \ar@{-}[r] & (3, 4)
}
\]

\medskip

Let $\C = \BA$, in which morphisms are order-preserving or order-reversing maps. Then $\Gamma_{2, 4}$ is depicted as below:
\[
\xymatrix{
(1, 2) \ar@{-}[r] & (1, 3) \ar@{-}[r] & (2, 3) \\
(3, 2) \ar@{-}[r] & (3, 1) \ar@{-}[r] & (2, 1)
}
\]
\end{example}

The reader can see from the above example that $\Gamma_{m, n}$ is either connected or has two connected components. This is not occasional.

\begin{proposition}\label{connectedness}
Suppose that the action of $G$ on $\Omega$ is highly homogeneous and that $n-m\geqslant 1$. Then the mutation graph $\Gamma_{m, n}$ is connected for $\OA$, $\CA$, and $\FA$, and has at most two connected components for $\BA$ and $\SA$.
\end{proposition}

\begin{proof}
It suffices to show that any injective morphism $f: [m]\hookrightarrow[n]$ can be transformed by a finite sequence of mutations into the standard inclusion (for $\FA,\OA,\CA$), or, in the cases $\BA$ and $\SA$, into either the standard inclusion or its reversal $i \mapsto n+1-i$.

\medskip
\noindent\textbf{Pure set ($\FA$).} Let $f: [m] \hookrightarrow [n]$ be any injection. Since $n>m$, there exists at least one element of $[n]$ not in the image of $f$. If $1 \notin f([m])$, we mutate $f$ by sending $1 \in [m]$ to $1 \in [n]$; if $1 = f(i)$ for some $i \neq 1$, choose $j \notin f([m])$, mutate $f$ by sending $i$ to $j$, and then mutate again to send $1$ to $1$. Repeating this procedure successively for $2,3,\dots,m$, we obtain the natural inclusion.

\medskip
\noindent\textbf{Dense linear order ($\OA$).} Let $f: [m] \hookrightarrow [n]$ be order-preserving with $n>m$. If $f(1) \neq 1$, then $f(1)-1 \notin f([m])$. The order-preserving surjection collapsing $f(1)-1$ and $f(1)$ produces a valid mutation $f'$ with $f'(1) = f(1)-1$. Iterating this procedure, we move $f(1)$ to $1$. After fixing $f(1)=1$, the same argument applies inductively to $f(2),\dots,f(m)$, yielding the standard inclusion.

\medskip
\noindent\textbf{Dense cyclic order ($\CA$).} Let $f: [m] \hookrightarrow [n]$ be a cyclic order embedding. Since $n>m$, there exists at least one hole in the $n$-cycle. A cyclic-order-preserving surjection collapsing two adjacent points allows an occupied point $f(x)$ to be moved into a neighboring hole, producing a mutation. By repeating such mutations, we may arrange $f(1) = 1$. Removing the point $1$ breaks the cycle into a linear order, reducing the problem to the $\OA$ case. Hence the mutation graph is connected.

\medskip
\noindent\textbf{Betweenness ($\BA$).}
For $m\geqslant 2$, embeddings fall into two classes: order-preserving and order-reversing. Any surjective morphism preserving the betweenness relation preserves this orientation type. Consequently, mutations preserve whether an embedding is order-preserving or order-reversing. As in the linear order case, any embedding can be mutated into either the standard inclusion or its reversal. Thus the graph has two connected components for $m\geqslant 2$, and is connected for $m=1$.

\medskip
\noindent\textbf{Separation ($\SA$).} The argument is identical. The separation relation detects cyclic orientation, which is preserved by separation-preserving surjections. Hence mutations preserve orientation class. Every embedding is mutation-equivalent to either the standard embedding or its reverse, and these two are not mutation-equivalent. Therefore the graph has at most two connected components.
\end{proof}

For the special case $m=n-1$, one can check that $f': [n-1] \hookrightarrow [n]$ is a mutation of $f$ if and only if there exists a surjective morphism $g:[n] \twoheadrightarrow [n-1]$ such that $g\circ f = g\circ f' = \id_{[n-1]}$. In this case, the mutation graph has a particularly rigid structure.

\begin{lemma}
Suppose that the action of $G$ on $\Omega$ is highly homogeneous. Then the mutation graph $\Gamma_{n-1, n}$ is bipartite.
\end{lemma}

\begin{proof}
We first treat the pure set case $\FA$. Each vertex $f$ of the mutation graph is uniquely determined by the sequence $(f(1), \ldots, f(n-1))$. Since the target has $n$ elements, there is a unique element $j \in [n] \setminus f([n-1])$, called the \emph{hole} of $f$. We associate to $f$ a permutation $\sigma_f$ in the symmetric group $S_n$ by
\[
\sigma_f(i)=
\begin{cases}
f(i), & 1 \leqslant i \leqslant n-1,\\
j, & i=n.
\end{cases}
\]
This gives a bijection between injections $[n-1] \hookrightarrow [n]$ and permutations in $S_n$. By definition of mutation, $f$ and $f'$ differ at a unique $a \in [n-1]$, and injectivity forces $f'(a) = j$, the hole of $f$. Equivalently, $\sigma_{f'}$ is obtained from $\sigma_f$ by exchanging the values at positions $a$ and $n$, i.e.
\[
\sigma_{f'} = \sigma_f \circ (a \; n),
\]
a single transposition. Thus every edge in the mutation graph corresponds to a transposition in $S_n$. Since each transposition flips the parity of a permutation, every edge flips parity, and therefore the graph contains no odd cycles. Hence $\Gamma_{n-1, n}$ is bipartite.

For the other highly homogeneous relations, the mutation graph $\Gamma'_{n-1, n}$ is a subgraph of $\Gamma_{n-1, n}$. Since $\Gamma_{n-1, n}$ is a bipartite graph, so is $\Gamma'_{n-1, n}$.
\end{proof}

\medskip

By this lemma, each connected component of $\Gamma_{n-1, n}$ admits a canonical sign system compatible with mutations.

\begin{corollary}\label{sign}
On each connected component of $\Gamma_{n-1, n}$ there exists a function $\epsilon: \Gamma_{n-1, n} \to \{\pm1\}$ such that whenever $f$ and $g$ differ by a mutation, one has $\epsilon(f) = -\epsilon(g)$. In addition, such a function is unique up to multiplication by $-1$.
\end{corollary}

\begin{proof}
Fix a basepoint $f_0$ in each connected component and define
\[
\epsilon(f)=(-1)^{d(f,f_0)},
\]
where $d(f,f_0)$ is the length of any path from $f_0$ to $f$ in the mutation graph. Since the graph is bipartite, the parity of $d(f,f_0)$ is independent of the chosen path. Thus $\epsilon$ is well defined and satisfies $\epsilon(f)=-\epsilon(g)$ whenever $f$ and $g$ are adjacent. Uniqueness up to global sign on each connected component is immediate.
\end{proof}

We need to establish one more elementary fact, which plays a crucial role for us to compute Hom-spaces between standard modules.

\begin{lemma} \label{lifts}
Let $f: [m] \hookrightarrow [n]$ be an injective morphism with $m < n$, and let $h: [n] \twoheadrightarrow [n-1]$ be a surjective morphism identifying the pair $\{a, b\} \subseteq [n]$. Then the following hold:
\begin{enumerate}
\item the composite $h \circ f$ is injective if and only if $|\{a,b\} \cap f([m])| \leqslant 1$;

\item if $\{a,b\}\cap f([m]) = \varnothing$, then for an injective morphism $f': [m] \hookrightarrow [n]$, one has
\[
h \circ f' = h \circ f \quad \Longleftrightarrow f = f';
\]

\item if $|\{a,b\}\cap f([m])| = 1$, then there exists a unique injective morphism
\[
f': [m] \hookrightarrow [n], \quad f' \neq f,
\]
such that $h \circ f' = h \circ f$. Furthermore, $f'$ is a mutation of $f$.
\end{enumerate}
\end{lemma}

\begin{proof}
Statements (1) and (2) are obvious, so we only give a proof for (3). Without loss of generality assume $\{a,b\} \cap f([m]) = \{a\}$, and let $i_0 \in [m]$ be the unique index with $f(i_0) = a$. Define $f': [m] \to[n]$ by
\[
f'(i)=
\begin{cases}
f(i), & i\neq i_0,\\
b, & i=i_0 .
\end{cases}
\]
The map $f'$ is clearly an injective map of sets and satisfies $h \circ f' = h \circ f$ by construction. Besides, $f'$ is a morphism in $\C$. To see this, note that the relation on $h(f([m])) \subseteq [n-1]$ is induced by the relations on $[n]$ via the morphism $h$. Since $h \circ f' = h \circ f$ and $h$ is a morphism that identifies only the pair $\{a, b\}$, the fact that $h \circ f$ is an embedding implies that the relations on the tuple $(f'(1), \dots, f'(m))$ must coincide with those on $(f(1), \dots, f(m))$.

Let $f''$ be any injective morphism with $h \circ f'' = h \circ f$. Then for all $i \neq i_0$ the fiber of $h$ over $(h \circ f)(i)$ is a singleton, so $f''(i) = f(i)$. Injectivity of $f''$ forces $f''(i_0) \in \{a,b\}$, hence $f'' = f$ or $f'' = f'$. Thus $f'$ is the unique injective morphism distinct from $f$ with $h \circ f' = h \circ f$. Moreover, since $f$ and $f'$ differ at exactly one index and are identified by the surjection $h$, $f'$ is a mutation of $f$ by definition.
\end{proof}

\section{Normalization}

The classical Dold--Kan correspondence is established via normalization. The main goal of this section is to extend this powerful tool from the simplex category to categories equipped with an underlying linear order. Unless otherwise specified, in this section let $\C$ be one of the following categories:
\[
\OA, \quad \CA, \quad \BA, \quad \SA.
\]
Correspondingly, for a fixed object $[n]$, its automorphism group $G_n$ is one of the following: the trivial group; the cyclic group $C_n$ of order $n$; the cyclic group $C_2$ of order $2$ (for $n \geqslant 2$); and the dihedral group $D_{2n}$ of order $2n$ (for $n \geqslant 3$).

Let $k$ be a commutative ring and $\D$ be the simplex category, viewed as a wide subcategory of $\C$. For each $n$, define the normalization operator by the ordered product
\[
\Psi_n = (\id_n - d_{n-1}s_{n-1}) (\id_n - d_{n-2}s_{n-2}) \cdots (\id_n - d_1 s_1) \;\in \; k\D(n,n),
\]
where the product is taken in the category algebra $k\D$. Here $d_i: [n-1] \hookrightarrow [n]$ is the $i$-th simplicial face map defined by
\[
d_i(j)=
\begin{cases}
j, & j < i,\\
j+1, & j \geqslant i,
\end{cases}
\]
and $s_i: [n] \twoheadrightarrow [n-1]$ is the $i$-th simplicial degeneracy map defined by
\[
s_i(j)=
\begin{cases}
j, & j \leqslant i,\\
j-1, & j > i.
\end{cases}
\]
By convention, we set $\Psi_1 = \id_{[1]}$. Note that we use the shifted convention $[n] = \{1,2,\ldots,n\}$ rather than the standard simplicial indexing $\{0,1,\ldots,n\}$.

The following property about $\Psi_n$ is implicitly used in \cite{DK85, Pirashvili99}. For completeness, we provide a detailed proof.

\begin{lemma} \label{idempotent}
Notation as above. For each $n \geqslant 1$, the operator $\Psi_n$ is an idempotent in $k\D(n, n)$ whose associated projection is precisely the normalization.
\end{lemma}

\begin{proof}
The first step is to check that $s_j \Psi_n = 0 = \Psi_n d_j$ for all $1 \leqslant j \leqslant n-1$. Fix $j$ and factor $\Psi_n = F_> (\id_n - d_j s_j) F_<$, where
\[
F_> = \prod_{i=n-1}^{j+1} (\id_n - d_i s_i), \qquad F_< = \prod_{i=j-1}^{1} (\id_n - d_i s_i),
\]
with products ordered decreasing in $i$. For $i>j$, the simplicial identities
\[
s_j d_i = d_{i-1} s_j, \quad s_j s_i = s_{i-1} s_j
\]
imply the operator relation:
\[
s_j (\id_n - d_i s_i) = s_j - d_{i-1} s_j s_i = s_j - d_{i-1} s_{i-1} s_j = (\id_{n-1} - d_{i-1} s_{i-1}) s_j.
\]
As $s_j$ moves right through $F_>$, each factor $\id_n - d_is_i$ is replaced by $\id_{n-1} - d_{i-1}s_{i-1}$, while the operator $s_j$ keeps its index. Upon reaching the central factor,
\[
s_j (\id_n - d_j s_j) = s_j - s_j d_j s_j = 0
\]
since $s_j d_j = \id_{n-1}$. Multiplication on the right by $F_<$ preserves zero, so $s_j \Psi_n = 0$.

Next, consider $d_j$ acting on the right. For any $i < j$, the identities
\[
d_i d_{j-1} = d_j d_i, \quad s_i d_j = d_{j-1} s_i
\]
imply:
\[
(\id_n - d_i s_i) d_j = d_j - d_i (d_{j-1} s_i) = d_j - d_j d_i s_i = d_j (\id_{n-1} - d_i s_i).
\]
Sliding $d_j$ leftward through $F_<$, it eventually reaches the central factor,
\[
(\id_n - d_j s_j) d_j = d_j - d_j s_j d_j = 0.
\]
Multiplication on the left by $F_>$ preserves zero, hence $\Psi_n d_j = 0$.

To see that $\Psi_n$ is idempotent, expand the product:
\[
\Psi_n = \id_n + \sum_{r \geqslant 1} \sum_{i_1 < \cdots < i_r} (\pm\, d_{i_r}s_{i_r} \cdots d_{i_1}s_{i_1}).
\]
Every non-identity term ends on the right with a factor $d_i s_i$. Since $s_i \Psi_n = 0$ for all $i$, it follows that $(d_i s_i) \Psi_n = 0$. Therefore,
\[
\Psi_n^2 = (\id_n + \text{terms ending in } d_i s_i) \cdot \Psi_n = \id_n \cdot \Psi_n + 0 = \Psi_n.
\]

Finally, if $v \in \bigcap \ker s_i$, then $s_i v = 0$ for all $i$. Thus $(\id_n - d_i s_i) v = v - 0 = v$. Applying this factor by factor, $\Psi_n v = v$. Consequently, $\Psi_n$ is an idempotent operator that projects onto the intersection of the kernels of the $s_i$, which is precisely the standard normalization.
\end{proof}

\begin{remark}
Although each factor in $\Psi_n$ is an idempotent, adjacent factors do not commute. The above proof relies only on the simplicial identities and associativity, without assuming commutativity. We shall remind the reader that the ordering of the factors is crucial: the push-through procedure guarantees that each $s_j$ reaches its corresponding annihilating factor $(\id_n - d_j s_j)$ without requiring commuting adjacent factors.
\end{remark}

Since $k\D(n, n)$ is a subalgebra of $k\C(n, n)$, it follows that $\Psi_n$ is also an idempotent in $k\C$, and hence $k\C \Psi_n$ as a direct summand of $P_n = k\C(n, -)$ is projective.

Note that $\CA$ or $\SA$ have extra cyclic degeneracy maps $s_n: [n] \twoheadrightarrow [n-1]$ identifying the pair $(1, n)$. However, the next lemma tells us that we can factorize any surjective morphism using solely simplicial degeneracy maps when its source and target are not adjacent.

\begin{lemma}\label{factorization}
If $m \geqslant n + 2$, then for any $f \in \C^-(m, n)$, there is a simplicial degeneracy map $s_i: [m] \twoheadrightarrow [m-1]$ with $i < m$, and a surjective morphism $h \in \C^-(m-1, n)$ such that $f = h \circ s_i$.

Dually, if $n \geqslant m +2$, then for any $f \in \C^+(m, n)$, there is a simplicial face map $d_i: [n-1] \hookrightarrow [n]$ with $i < n$, and an injective morphism $h: [m] \hookrightarrow [n-1]$ such that $f = d_i \circ h$.
\end{lemma}

\begin{proof}
These statements are well known for the simplex category $\D = \OA$ when $m \geqslant n+2$. We extend them to the four categories considered in this section. For $f \in \C^-(m,n)$, we can write $f = \sigma \circ f_0$ with $\sigma \in G_n$ and $f_0 \in \D^-(m,n)$. But in $\D$ we have $f_0 = h_0 \circ s_i$ for some simplicial degeneracy map $s_i: [m] \twoheadrightarrow [m-1]$ with $i < m$ and a morphism $h_0 \in \D^-(m-1, n)$. Setting $h = \sigma \circ h_0$ gives the desired decomposition.

Dually, for $f\in \C^+(m, n)$ we write $f = f_0 \circ \sigma$ with $\sigma \in G_n$ and $f_0 \in \D^+(m, n)$. But in $\D$ we have $f_0 = d_i \circ h_0$ for some face map $d_i: [n-1] \hookrightarrow [n]$ with $i < n$ and $h_0 \in \D^+(m, n-1)$. Setting $h = h_0 \circ \sigma$ yields the desired factorization.
\end{proof}

The above lemma does not hold for $\FA$ since $h \circ s_i$ identify the pair of adjacent indices $(i, i+1)$, but there are surjections $f$ identifying no adjacent indices.

Let $\mathfrak{J}_n$ be the left $k\C$-ideal generated by simplicial degeneracy maps $s_i: [n] \twoheadrightarrow [n-1]$ with $i<n$. By the proof of Lemma \ref{idempotent}, $\mathfrak{J}_n$ is contained in the kernel of $\Psi_n$. The next result asserts that it is precisely the kernel.

\begin{corollary}\label{ideal description}
 Then there is a canonical isomorphism of left $\C$-modules
\[
k\C \Psi_n \longrightarrow k\C(n,-)/\mathfrak{J}_n .
\]
\end{corollary}

\begin{proof}
Define a $k$-linear map
\[
\phi: k\C(n, -) \longrightarrow k\C \Psi_n, \quad  x \longmapsto x\Psi_n.
\]
We show that $\ker\phi = \mathfrak{J}_n$. One direction follows from the proof of Lemma \ref{idempotent}. It remains to prove the inclusion of the other direction.

By Lemma \ref{factorization}, every non-injective morphism with source $[n]$ either factors through some $s_i$ with $i<n$, or in the cases $\C = \CA, \SA$, factors as $fs_n$ for an injective morphism $f \in \C^+(n-1, -)$. Consequently, $k\C(n, -)/\mathfrak{J}_n$ as a $k$-module is spanned by morphisms $g \in \C^+(n, -)$, and in the cases $\C = \CA, \SA$, morphisms of the form $f s_n$ with $f \in \C^+(n-1, -)$.

Take $v \in \ker \phi$ and write it as
\[
v = \sum_{i=1}^r a_i f_i s_n + \sum_{j=1}^s b_j g_j,
\]
with $f_i \in \C^+(n-1, -)$, $g_j \in \C^+(n, -)$, and $a_i, b_j \in k$. Then
\[
v\Psi_n = \sum_{i=1}^r a_i f_i s_n \Psi_n + \sum_{j=1}^s b_j g_j + \sum_{j=1}^s b_j g_j \cdot (\text{products involving some } d_q s_q),
\]
a linear combination of morphisms. Since degeneracy maps identifies two elements, the only injective morphisms appearing in this expansion are the original $g_j$ with coefficients $b_j$. Thus $v \Psi_n = 0$ forces
\[
\sum_{j=1}^s b_j g_j = 0.
\]
For $\C=\OA,\BA$, there are no terms $f_i s_n$, so we conclude $v=0$.

For $\C =\CA, \SA$, we may therefore assume
\[
v = \sum_{i=1}^r a_i f_i s_n .
\]
Accordingly,
\[
v \Psi_n = \sum_{i=1}^r a_i f_i s_n + \sum_{i=1}^r a_i f_i s_n \cdot (\text{products involving some } d_q s_q).
\]
Each morphism $f_i s_n$ identifies exactly the pair $(1,n)$ in $[n]$. By contrast, any morphism in the second sum identifies a different pair or more than one pair. Therefore, no term in the second sum can equal any $f_i s_n$. Thus the linear combination of $a_i f_i s_n$ is linearly independent from all other terms in the expansion of $v \Psi_n$, and $v \Psi_n = 0$ implies
\[
v = \sum_{i=1}^r a_i f_i s_n = 0.
\]
This finishes the proof.
\end{proof}

The normalized projective module $k\C \Psi_n$ has a very simple filtration by standard modules.

\begin{proposition} \label{filtration of projectives}
If $\C$ is $\OA$ or $\BA$, then $\Delta_n \cong k\C \Psi_n$, and hence is projective. If $\C$ is $\CA$ or $\SA$, then for $n \geqslant 3$, there exists a short exact sequence of $\C$-modules
\[
0 \to \Delta_{n-1} \to k\C \Psi_n \to \Delta_n \to 0.
\]
\end{proposition}

\begin{proof}
If $\C = \OA$ or $\C = \BA$, then $\mathfrak{J}_n = \mathfrak{I}_n$, so by Corollary \ref{ideal description}
\[
k\C \Psi_n \cong k\C(n, -) / \mathfrak{J}_n = k\C(n, -) / \mathfrak{I}_n = \Delta_n,
\]
which implies the projectivity of $\Delta_n$.

If $\C = \CA$ or $\C = \SA$ and $n \geqslant 3$, then $\mathfrak{J}_n$ is a proper subideal of $\mathfrak{I}_n$, and the quotient of the map
\[
k\C \Psi_n \cong k\C(n, -) / \mathfrak{J}_n \longrightarrow \Delta_n = k\C(n, -) / \mathfrak{I}_n
\]
is isomorphic to $\mathfrak{I}_n / \mathfrak{J}_n$. By Proposition \ref{standard modules}, $\mathfrak{I}_n$ has a filtration whose factors are
\[
k\C \otimes_{k\C^-} k\C^-(n, m) \cong \Delta_m \otimes_{kG_m} k\C^-(n, m),
\]
with $m < n$. Similarly, we can construct a filtration for $ \mathfrak{J}_n $ whose factors are
\[
k\C \otimes_{k\C^-} \mathfrak{J}^-(n, m),
\]
with $ m < n $, where $ \mathfrak{J}^-(n, m) $ is the free $k$-module spanned by all surjective morphisms from $ [n] $ to $ [m] $ factoring through some simplicial degeneracy map $s_i$ with $i < n$. By Lemma \ref{factorization}, we have
\[
\mathfrak{J}^-(n, m) = k\C^-(n, m)
\]
when $m < n - 1$. Consequently,
\[
\mathfrak{I}_n / \mathfrak{J}_n \cong k\C \otimes_{k\C^-} (k\C^-(n, n-1) / \mathfrak{J}^-(n, n-1)).
\]

In $ \CA $ and $ \SA $, the set $ \C^-(n, n-1) $ consists of surjective morphisms identifying exactly one adjacent pair (either simplicial or cyclic). Since $\mathfrak{J}^-(n, n-1)$ is spanned by those identifying a simplicial pair, the quotient $ k\C^-(n, n-1) / \mathfrak{J}^-(n, n-1)$ is a free right $kG_{n-1}$-module with a basis consisting of surjective morphisms identifying the cyclic pair $(1, n)$. Therefore, we obtain
\[
\mathfrak{I}_n / \mathfrak{J}_n \cong k\C \otimes_{k\C^-} kG_{n-1} \cong \Delta_{n-1},
\]
by Proposition \ref{standard modules}. The desired short exact sequence then follows.
\end{proof}

For $n \leqslant 2$, one has $\Delta_n = k\C \Psi_n$, which is projective since $s_2: [2] \twoheadrightarrow [1]$ is the same as $s_1$.

\begin{proposition} \label{thin category}
Suppose that $|m - n| \geqslant 2$. Then
\[
\Hom_{k\C} (k\C \Psi_n, \, k\C\Psi_m) = 0.
\]
If $\C$ is $\OA$ or $\BA$, then the conclusion also holds for $m = n+1$.
\end{proposition}

\begin{proof}
By the Yoneda lemma, we have
\[
\Hom_{k\C}(k\C\Psi_n, \, k\C\Psi_m) \cong \Psi_n \, k\C \Psi_m = \Psi_n \, k\C(m, n) \, \Psi_m.
\]
We will show $k\C(m, n)\Psi_m = 0$ for $m \geqslant n + 2$ and $\Psi_n k\C(m, n) = 0$ for $n \geqslant m+2$.

Suppose that $m \geqslant n + 2$. Then any morphism $f \in \C(m, n)$ can be written as $f = g \circ h$, where $h: [m] \twoheadrightarrow [r]$ is surjective, $g: [r] \hookrightarrow [n]$ is injective, and $m \geqslant n+2 \geqslant r$. By Lemma \ref{factorization}, one can write $h = h' \circ s_i$, where $s_i: [m] \twoheadrightarrow [m-1]$ is a simplicial degeneracy map. It follows that
\[
f \Psi_m = gh \Psi_m = g h' s_i \Psi_m = 0
\]
by the proof of Lemma \ref{idempotent}, so $k\C(m, n) \Psi_m = 0$ as desired.

Now suppose $n \geqslant m + 2$. Then any morphism $f \in \C(m, n)$ can be written as $f = g \circ h$, where $h: [m] \twoheadrightarrow [r]$ is surjective, $g: [r] \hookrightarrow [n]$ is injective, and $n \geqslant m+2 \geqslant r$. By Lemma \ref{factorization}, we have $g = d_i \circ g'$ where $d_i: [n-1] \hookrightarrow [n]$ is a face map with $i < n$. It follows that
\[
\Psi_n f = \Psi_n gh = \Psi_n d_i g' h = 0
\]
by the proof of Lemma \ref{idempotent}, so $\Psi_n k\C(m, n) = 0$ as desired.

Finally let $\C$ be $\OA$ or $\BA$. In this case $k\C \Psi_m \cong \Delta_m$ and $k\C \Psi_n \cong \Delta_n$. The second statement then follows from Proposition \ref{standard modules}.
\end{proof}

Now we are ready to establish a Dold-Kan correspondence for $\C$. We define a $k$-linear category $\mathscr{K}$ as follows:
\begin{itemize}
\item objects of $\mathscr{K}$ are positive integers $n \in \mathbb{Z}_+$,
\item for $m,n \in \mathbb{Z}_+$, the morphism set is
\[
\mathscr{K} (m, n) = \Hom_{k\C} (k\C \Psi_n, \, k\C \Psi_m) \cong \Psi_n \, k\C \Psi_m = \Psi_n \, k\C(m, n) \, \Psi_m.
\]
\end{itemize}
Call $\mathscr{K}$ the \emph{normalized category} of $\C$. For a morphism $f \in \C(m, n)$, we denote by
\[
\mathbf{f} = \Psi_n f \Psi_m \in \mathscr{K}(m,n)
\]
its \emph{normalized version}. In particular, we define the normalized face, degeneracy, and boundary operators as follows:
\begin{align*}
\mathbf{d}_{n+1} &= \Psi_{n+1} d_{n+1} \Psi_n, \\
\mathbf{s}_{n+1} &= \Psi_n s_{n+1} \Psi_{n+1} \quad \text{(for $\C = \CA$ or $\SA$)},\\
\delta_n &= \Psi_n d_n s_n \Psi_n \quad \text{(for $\C = \CA$ or $\SA$)}.
\end{align*}

By construction, the identity morphism in $\mathscr{K}(n,n)$ is the idempotent $\Psi_n$. Consequently, for any normalized morphism $\mathbf{f} \in \mathscr{K}(m,n)$, we may freely insert or remove the normalization idempotents $\Psi$ at the source and target without changing the morphism.

\begin{theorem} \label{generalized D-K correspondence}
The category $\C \Mod$ is equivalent to $\mathscr{K} \Mod$, the category of $k$-linear functors from $\mathscr K$ to $k\Mod$. Furthermore, $\mathscr{K}$ satisfies the following properties:

\medskip

\textbf{(I) General properties (for all $\C$):}
\begin{enumerate}
\item $\mathscr K(m,n) = 0$ if $|m-n| > 1$,

\item $\mathscr K(n,n+1)$ is spanned by $\mathbf{d}_{n+1} \boldsymbol{\sigma}$ with $\sigma \in G_n$,

\item $\mathbf{d}_{n+1} \mathbf{d}_n = 0$.
\end{enumerate}

\medskip
\textbf{(II) Additional structure for $\C = \OA,\ \BA$:}
\begin{enumerate}
\setcounter{enumi}{3}
\item $\mathscr K(n,n)$ is spanned by $\boldsymbol{\sigma}$ with $\sigma\in G_n$,

\item $\mathscr{K} (m, n) = 0$ if $m < n$.
\end{enumerate}

\medskip
\textbf{(III) Additional structure for $\C= \CA,\ \SA$:}
\begin{enumerate}
\setcounter{enumi}{5}
\item $\mathscr K(n,n)$ is spanned by $\boldsymbol{\sigma}$ with $\sigma\in G_n$ and $\mathbf{d}_n \boldsymbol{\tau} \mathbf{s}_n$ with $\tau \in G_{n-1}$,

\item $\mathscr K(n+1,n)$ is spanned by $\boldsymbol{\sigma}\mathbf{s}_{n+1}$ with $\sigma\in G_n$,

\item $\mathbf{s}_n \mathbf{s}_{n+1}=0$,

\item $\mathbf{s}_{n+1} \mathbf{d}_{n+1} = \Psi_n$,

\item $\mathbf{d}_{n+1} \mathbf{s}_{n+1} = \delta_{n+1}$,

\item $\delta_n^2 = \delta_n$,

\item $\mathbf{s}_n \delta_n = \mathbf{s}_n$ and $\delta_n \mathbf{d}_n = \mathbf{d}_n$,

\item $\mathbf{d}_n \delta_{n-1} = 0$ and $\delta_{n-1} \mathbf{s}_n = 0$.
\end{enumerate}

\medskip
\textbf{Crossed simplicial structure.} The automorphism groups $G_n$ form a crossed simplicial group: elements of $G_n$ permute the simplicial face and degeneracy maps; in $\mathscr K$ where only the top face and degeneracy survive, their images are stable under this induced action.
\end{theorem}

In a more transparent way, for $\OA$ and $\BA$, $\mathscr{K}$ looks like
\[
\xymatrix{
1 \ar@(ul,ur)[]^{G} \ar[r]^{\mathbf{d}_2} &
2 \ar@(ul,ur)[]^{G} \ar[r]^{\mathbf{d}_3} &
3 \ar@(ul,ur)[]^{G} \ar[r]^{\mathbf{d}_4} &
\cdots
}
\]
where $G$ is either the trivial group or the cyclic group of order 2. For $\CA$ and $\SA$, $\mathscr{K}$ looks like
\[
\xymatrix{
1 \ar@(ul,ur)[]^{G_1} \ar[r]^{\mathbf{d}_2} &
2 \ar@(ul,ur)[]^{G_2} \ar@/^1pc/[r]^{\mathbf{d}_3} &
3 \ar@(ul,ur)[]^{G_3} \ar@/^1pc/[r]^{\mathbf{d}_4} \ar@/^1pc/[l]^{\mathbf{s}_3} &
\cdots \ar@/^1pc/[l]^{\mathbf{s}_4}
}
\]
Note that $\mathbf{s}_2$ does not appear in this diagram since it coincides with $\mathbf{s}_1$, which is killed by $\Psi_2$.

\begin{proof}
Since each representable functor $k\C(n-)$ has a finite filtration by standard modules, and each $\Delta_n$ is a quotient of $k\C \Psi_n$ by Proposition \ref{filtration of projectives}, we deduce that $\{ k\C \Psi_n \mid n \in \mathbb{Z}_+ \}$ is a family of projective generators of $\C \Mod$. The equivalence then follows from the Morita equivalence induced by the functor
\[
V \mapsto \bigoplus_n \Hom(k\C \Psi_n, V).
\]
Now we check the properties of $\mathscr{K}$.

\medskip
(1), (3) and (8). These statements immediately follow from Proposition \ref{thin category}.

\medskip
(2). By Lemma $\ref{factorization}$, any non-injective morphism $f \in \C(n, n+1)$ factors through some $[r]$ with $r < n$, and is thus killed by $\Psi_{n+1}$. Consequently, $\mathscr{K}(n, n+1)$ is spanned by the images of injective morphisms. Note that an injective morphism $f  \in \C^+(n, n+1)$ can be written as $d_i \sigma$ for some $\sigma \in G_n$. Since $\Psi_{n+1}$ kills all simplicial face maps $d_i$ for $i < n+1$ from the left, the only non-zero survivors in the normalized category is the orbit of $d_{n+1}$ under the right action of $G_n$. Thus the space is spanned by $\mathbf{d}_{n+1} \boldsymbol{\sigma}$ with $\sigma \in G_n$.

\medskip
(4) and (6). Any endomorphism $f: [n] \to [n]$ is killed by $\Psi_n$ if it factors through some $s_i: [n] \to [n-1]$ with $i < n$ or a certain $d_i: [n-1] \to [n]$ with $i < n$. By Lemma \ref{factorization}, only automorphisms $\sigma \in G_n$ and endomorphisms $d_n \tau s_n$ with $\tau \in G_{n-1}$ (for $\CA$ and $\SA$) can survive in $\mathscr{K}$.

\medskip
(5). One has $k\C \Psi_n \cong \Delta_n$ by Proposition \ref{filtration of projectives}. The conclusion follows from (6) of Proposition \ref{standard modules}.

\medskip
From now on we assume $\C \in \{\CA,\SA\}$, so that the cyclic degeneracy map $s_n: [n] \twoheadrightarrow [n-1]$ is a morphism in $\C$.

\medskip
(7). This is dual to (2). Any morphism $f: [n+1] \to [n]$ that factors through $[r]$ with $r < n$ is killed by $\Psi_{n+1}$ from the right via Lemma $\ref{factorization}$. The remaining candidates are surjective morphisms in $\C^-(n+1, n)$. In $\CA$ and $\SA$, $G_n$ acts on $\C^-(n+1, n)$ from the left with orbits represented by $s_1, \dots, s_{n+1}$. Since $\Psi_{n+1}$ kills the simplicial degeneracies $s_1, \dots, s_n$ from the right, the only survivor is the left $G_n$-orbit of the cyclic degeneracy $s_{n+1}$. Thus, the space is spanned by $\boldsymbol{\sigma}\mathbf{s}_{n+1}$ with $\sigma \in G_n$.

\medskip
(9). Since $s_{n+1} d_{n+1} = \id_n$ in $\C$, it follows that $\mathbf{s}_{n+1} \mathbf{d}_{n+1}$ is the corresponding identity $\Psi_n$ in $\mathscr{K}$.

\medskip
(10). By definition, one has
\[
\mathbf{d}_{n+1} \mathbf{s}_{n+1} = (\Psi_{n+1} d_{n+1} \Psi_n) (\Psi_n s_{n+1} \Psi_{n+1}) = \Psi_{n+1} d_{n+1} \Psi_n s_{n+1} \Psi_{n+1}.
\]
Recall that $\Psi_n$ is the ordered product of terms $(\id_n - d_i s_i)$ for $i = n-1, \ldots, 1$, and
\[
d_{n+1}(\id_n - d_i s_i) = d_{n+1} - d_{n+1} d_i s_i = d_{n+1} - d_i d_n s_i.
\]
since $d_{n+1} d_i = d_i d_n$. But $d_i$ is killed by $\Psi_{n+1}$ from the left, so we have
\[
\Psi_{n+1} d_{n+1} (\id_n - d_i s_i) = \Psi_{n+1} (d_{n+1} - d_i d_n s_i) = \Psi_{n+1} d_{n+1}
\]
for $i < n$. Consequently,
\[
\Psi_{n+1} d_{n+1} \Psi_n = \Psi_{n+1} d_{n+1}.
\]
This allows the internal normalization to be absorbed, yielding
\[
\mathbf{d}_{n+1}\mathbf{s}_{n+1} = \Psi_{n+1} d_{n+1} s_{n+1} \Psi_{n+1} = \delta_{n+1}.
\]

\medskip
(11)-(13). By (3), (9) and (10),  we have
\begin{align*}
\delta_n^2 & = (\mathbf{d}_n \mathbf{s}_n)(\mathbf{d}_n \mathbf{s}_n) = \mathbf{d}_n (\mathbf{s}_n \mathbf{d}_n) \mathbf{s}_n = \mathbf{d}_n (\Psi_{n-1}) \mathbf{s}_n = \mathbf{d}_n \mathbf{s}_n = \delta_n,\\
\delta_n \mathbf{d}_n & = (\mathbf{d}_n \mathbf{s}_n) \mathbf{d}_n = \mathbf{d}_n (\mathbf{s}_n \mathbf{d}_n) = \mathbf{d}_n \Psi_{n-1} = \mathbf{d}_n,\\
\mathbf{s}_n \delta_n & = \mathbf{s}_n (\mathbf{d}_n \mathbf{s}_n) = (\mathbf{s}_n \mathbf{d}_n) \mathbf{s}_n = \Psi_{n-1} \mathbf{s}_n = \mathbf{s}_n,\\
\mathbf{d}_n \delta_{n-1} & = \mathbf{d_n} \mathbf{d}_{n-1} \mathbf{s}_{n-1} = 0,\\
\delta_{n-1} \mathbf{s}_n & = \mathbf{d}_{n-1} \mathbf{s}_{n-1} \mathbf{s}_n = 0.
\end{align*}

\medskip
Finally, the reader can refer to \cite{FL91} for details on the crossed simplicial structure.
\end{proof}

We can deduce several results from the above equivalence. In the rest of this section let $\C$ be one of the five categories associated to highly homogeneous relations (including $\FA$).

\begin{corollary} \label{finite length 1}
If $k$ is noetherian, then every finitely generated $\C$-module is noetherian. If $k$ is artinian, then every finitely generated $\C$-module is of finite length.
\end{corollary}

\begin{proof}
It suffices to show the conclusion for every $P_n = k\C(n, -)$. Firstly suppose that $\C$ is not $\FA$. By the above equivalence, we need to show the conclusion for $\mathscr{K}(n, -)$. By Proposition \ref{thin category},
\[
\mathscr{K} (n, -) = \mathscr{K} (n, n) \oplus \mathscr{K} (n, n+1) \oplus \mathscr{K} (n, n-1)
\]
as $k$-modules. Since each direct summand is a free $k$-module of finite rank, so is $\mathscr{K} (n, -)$ viewed as a $k$-module. The conclusion follows.

\medskip

Now suppose that $\C = \FA$. It suffices to show the conclusion for standard modules $\Delta_n^{\C}$ since each representable functor has a finite filtration by standard modules. Let $\D = \OA $ and $\Res^\C_\D$ be the restriction functor. We prove the following isomorphism
\begin{equation} \label{eq:res-decomp}
\Res^\C_\D \Delta_n^\C \cong \bigoplus_{\sigma \in G_n} (\Delta_n^\D)_{\sigma},
\end{equation}
where $G_n = S_n$ and $\Delta_n^\D$ is the standard module for $\OA$. This decomposition implies the conclusion since $\Delta_n^\D$ has the desired property.

Note that $\Delta_n^\C$ has a $k$-basis consisting of $\overline{f}$ with $f: [n] \hookrightarrow [m]$ injections, $\Delta_n^\D$ has a basis indexed by order-preserving injections, and every injection $f$ can be uniquely written as $f = f_0 \circ \sigma$ with $f_0$ order-preserving and $\sigma \in G_n$. This observation gives the decomposition \eqref{eq:res-decomp} of the underlying $k$-modules. It remains to check that this decomposition is compatible with the action of $\D$. Since any morphism in $\D$ is a composite of an injective morphism and a surjective morphism, we only need to check that the decomposition is compatible with these special morphisms.

\begin{itemize}
\item \textbf{Injective morphisms.} Let $g: [m] \hookrightarrow [\ell]$ be an order-preserving injection. Then
\[
g \cdot \overline{f} = \overline{g \circ (f_0 \circ \sigma)},
\]
where $g \circ f_0$ is still order-preserving. Thus the action preserves each summand $(\Delta_n^\D)_\sigma$, and the decomposition is compatible with injective morphisms.

\item \textbf{Surjective morphisms.} Let $h: [m] \twoheadrightarrow [\ell]$ be an order-preserving surjection. Then
\[
h \cdot \overline{f} = \overline{h \circ (f_0 \circ \sigma)}.
\]
If $h \circ f_0$ is not injective, it becomes zero in $\Delta_n^{\C}$. Otherwise, $h \circ f_0$ is injective and order-preserving, and $h \cdot \overline{f}$ is therefore contained in $(\Delta_n^\D)_\sigma$.
\end{itemize}

Since in either case each summand $(\Delta_n^\D)_\sigma$ is preserved under the actions, \eqref{eq:res-decomp} is indeed a decomposition of $\D$-modules.
\end{proof}

One immediately deduces the following result.

\begin{corollary} \label{socle}
If $k$ is artinian, then every nonzero $\C$-module has a nonzero socle.
\end{corollary}

\begin{proof}
Note that a nonzero $\C$-module $V$ is a sum of finitely generated $\C$-modules. Since every finitely generated $\C$-module is of finite length, its socle is nonzero. Thus $V$ has a nonzero socle.
\end{proof}

The next corollary tells us that most Hom-spaces between standard modules are trivial.

\begin{corollary} \label{no homo with big gap}
If $\Hom_{k\C} (\Delta_m, \, \Delta_n) \neq 0$, then $m = n$ or $m = n+1$.
\end{corollary}

\begin{proof}
By Proposition \ref{standard modules}, we have $m \geqslant n$. It remains to show $m \leqslant n+1$. The conclusion for $\D = \OA$ follows from Propositions \ref{filtration of projectives} and \ref{thin category}. For the category $\C$, since the restriction functor is faithful, we deduce that
\[
\Hom_{k\D} (\Res^\C_\D \Delta_m^\C, \, \Res^\C_\D \Delta_n^\C) \neq 0.
\]
But as in the proof of Corollary \ref{finite length 1}, one can show that $\Res^\C_\D \Delta_m^\C$ is a direct sum of finitely many copies of $\Delta_m^{\D}$. It follows that
\[
\Hom_{k\D} (\Delta_m^\D, \, \Delta_n^\D) \neq 0.
\]
The conclusion then follows from the $\OA$ case.
\end{proof}

For $\C = \BA$, we have a ``two-fold'' version of the classical Dold-Kan correspondence.

\begin{corollary} \label{DK for BA}
Let $k$ be a commutative ring and $\C = \BA$. Then $\C \Mod$ is equivalent to the category of chain complexes $(V_n, \mathbf{d}_n)_{n \geqslant 1}$ such that:
\begin{itemize}
\item $V_1$ is a $k$-module;
\item $V_n$ is a $kC_2$-module for all $n \geqslant 2$, with $kC_2$-linear differential maps $\mathbf{d}_n: V_n \to V_{n+1}$ satisfying $\mathbf{d}_{n+1} \mathbf{d}_n = 0$.
\end{itemize}

In addition, if the characteristic of $k$ is not $2$, then $\mathscr{K} \Mod$ decomposes as
\[
\mathscr{K}_+ \Mod \;\times\; \mathscr{K}_- \Mod,
\]
where $\mathscr{K}_+$ is isomorphic to the normalized category of $\OA$, and $\mathscr{K}_-$ is isomorphic to the full subcategory of $\mathscr{K}_+$ obtained by removing its initial object.
\end{corollary}

\begin{proof}
From Theorem \ref{generalized D-K correspondence}, the normalized category $\mathscr{K}$ for $\BA$ has objects $n \geqslant 1$ and
\[
\mathscr{K}(n,n) \simeq kC_2 \quad \text{for } n \geqslant 2,
\]
while $\mathscr{K}(1,1)$ is trivial. Each $V_n$ is a $kC_2$-module for $n \geqslant 2$, and the differential $\mathbf{d}_n$ is $kC_2$-linear. This yields the first statement.

In the case that $k$ has characteristic distinct from $2$, each $kC_2$-module $V_n$ decomposes into a plus and a minus eigenspace under the nontrivial element of $C_2$, giving a decomposition of $\mathscr{K}$-modules into two independent complexes. The second statement follows.
\end{proof}

\section{Standard modules}

Standard modules play a vital role for representation theory of $\C$. In this section we deduce more structural information on them. Firstly, we explicitly compute Hom-spaces between them, strengthening Corollary \ref{no homo with big gap}. This computation is completely based on the generalized Reedy structure and mutation graphs, without relying on the normalization described in the previous section.

\begin{theorem} \label{hom spaces}
Let $k$ be a commutative ring. Then for all $m, n \geqslant 1$, one has
\[
\mathrm{Hom}_{k\C}(\Delta_m, \, \Delta_n) \; \cong \;
\begin{cases}
kG_n, & m=n,\\
k^d, & m = n+1,\\
0, & \text{else},
\end{cases}
\]
where $d$ is the number of connected components in the mutation graph $\Gamma_{n, n+1}$. Explicitly, $d = 1$ for $\FA$, $\OA$, or $\CA$; $d = 2$ for $\BA$ or $\SA$ when $m \geqslant 2$, and $d = 1$ when $m = 1$.
\end{theorem}

\begin{proof}
The case $m \leqslant n$ is established in Proposition \ref{standard modules}. Assume $m > n$. We have
\[
\Hom_{k\C}(\Delta_m, \, \Delta_n) = \Hom_{k\C}(P_m/\mathfrak I_m, \, \Delta_n) \cong \{\phi \in \Hom_{k\C}(P_m, \, \Delta_n) \mid \mathfrak{I}_m \subseteq \ker \phi\}.
\]
Yoneda's Lemma gives a natural identification
\[
\Hom_{k\C} (P_m, \, \Delta_n) \cong \Delta_n(m),
\]
the value of $\Delta_n$ on $[m]$. Under this identification, an element $v\in\Delta_n(m)$ defines a morphism $P_m \to \Delta_n$ that vanishes on $\mathfrak I_m$ if and only if $\mathfrak I_m \cdot v = 0$ in $\Delta_n$, or equivalently, $\mathfrak I_m \cdot v \subseteq \mathfrak I_n$ in $P_n$. Thus
\[
\Hom_{k\C}(\Delta_m,\, \Delta_n) \cong \{ v \in \Delta_n(m) \mid \mathfrak I_m\cdot v \subseteq \mathfrak I_n \}.
\]

Since $\Delta_n(m)$ has a natural $k$-basis given by morphisms in $\C^+(n,m)$, we may identify $\Delta_n(m)$ with $k\C^+(n,m)$. Also note that $\mathfrak I_m$ is generated by morphisms in $\C^-(m, m-1)$. Under this identification, we have
\begin{equation}
\Hom_{k\C}(\Delta_m,\, \Delta_n) \cong \{ v \in k\C^+(n,m) \mid g \cdot v \in \mathfrak I_n, \, \forall \, g \in\C^-(m,m-1) \}, \tag{$\ast$}
\end{equation}
where $g\cdot v$ is induced by composition.

\medskip
Let $v\in k\C^+(n,m)$ satisfy the condition in $(\ast)$, and write
\[
v = \sum_{f\in \C^+(n,m)} c_f f, \qquad c_f\in k.
\]
Suppose that $f'$ is a mutation of $f$. Then there exists a morphism $g \in \C^-(m, m-1)$ such that $g \circ f = g \circ f'$ is injective; see Remark \ref{mutation}. Besides, by (3) of Lemma \ref{lifts}, there is no other morphism $f'' \in \C^+(n,m)$ such that $g \circ f'' = g \circ f$. Since $g \cdot v \in \mathfrak I_n$ is a linear combination of non-injective morphisms, the coefficient of $g \circ f$ in $g \cdot v$ must vanish, so $c_f + c_{f'} = 0$, namely $c_f = -c_{f'}$. It follows that the coefficients $c_f$ of vertices $f$ in any connected component of $\Gamma_{n, m}$ are uniquely determined up to a global sign.

Now we prove the conclusion for the last two cases.

\medskip
\noindent
\textbf{Case $m>n+1$.}  This has been shown in Corollary \ref{no homo with big gap}. Here we use the mutation graph to give an alternative proof. Since $m > n+1$, we can find in each connected component of $\Gamma_{n, m}$ a morphism $f: [n] \hookrightarrow [m]$ such that $[m] \setminus f([n])$ contains the pair $(m-1, m)$ (the standard inclusion) or the pair $(1, 2)$ (the reverse inclusion). Let $g \in \C^-(m, m-1)$ be a morphism identifying this pair. Then by Lemma \ref{lifts}, $g \circ f$ is injective (hence not contained in $\mathfrak{I}_n$), and there is no other morphism $f' \in \C^+(n, m)$ with $g \circ f = g \circ f'$. Thus $g \cdot v \in \mathfrak{I}_n$ forces $c_f = 0$, and hence coefficients for all vertices in the same component must be 0. Consequently, we deduce that $\Hom_{k\C} (\Delta_m, \, \Delta_n) = 0$ for $m > n + 1$ by $(\ast)$.

\medskip
\noindent
\textbf{Case $m=n+1$.} We already know that each connected component of $\Gamma_{n, n+1}$ contributes at most one element $v$ (up to scalar) satisfying the condition in $(\ast)$, namely the alternating sum of vertices in the component. Thus $\Hom_{k\C} (\Delta_{n+1}, \, \Delta_n)$ has dimension at most $d$. The conclusion follows after showing that this $v$ indeed satisfies this condition.

Let $C$ be a connected component of $\Gamma_{n, n+1}$. By Corollary \ref{sign}, we can take a mutation-compatible function $\epsilon: C \to \{1, -1\}$ such that $\epsilon(f) = -\epsilon(f')$ if $f'$ is a mutation of $f$. Then up to sign, one has
\[
v = \sum_{f \in C} \epsilon(f) f.
\]
Fix a morphism $h \in \C^+(n+1, n)$. We have
\[
h \cdot v = \sum_{f \in C} \epsilon(f) (h \circ f).
\]
We need to show that the coefficient of $g$ is 0 for every injective morphism $g = h \circ f$ appearing in the above expression of $h \cdot v$. It then follows that $h \cdot v$ is a linear combination of non-injective morphisms, and hence is contained in $\mathfrak{I}_n$ as desired.

Since both $f: [n] \hookrightarrow [n+1]$ and $h \circ f$ are injective, the hypotheses of Lemma \ref{lifts} (3) apply. Thus there are exactly two injective morphisms $f$ and $f'$ in $\C^+(n, n+1)$ such that
\[
g = h \circ f = h \circ f',
\]
and furthermore $f'$ is a mutation of $f$. By mutation compatibility, one has $\epsilon(f) = - \epsilon(f')$, so the coefficient of $g$ is indeed $\epsilon(f) + \epsilon(f') = 0$.
\end{proof}

When $k$ is a field, this elementary computation provides an alternating proof of Dold-Kan correspondence for $\OA$ and $\BA$.

\begin{corollary} \label{OA and BA}
Let $\C$ be $\OA$ or $\BA$ and $k$ a field. Then every $\Delta_n$ is a projective $\C$-module.
\end{corollary}

\begin{proof}
We prove the conclusion by an induction on $n$. Since $\Delta_1 = P_1$ is projective, we assume that $n \geqslant 2$, and suppose that $\Delta_i$ is projective for $i < n$. Now we consider $\Delta_n$.

Note that $\mathfrak{I}_n$ has a filtration
\[
0 \subseteq I^1 \subseteq I^2 \subseteq \ldots \subseteq I^{n-1} = \mathfrak{I}_n
\]
such that $I^i$ is the submodule spanned by all morphisms with source $[n]$ whose image has cardinality at most $i$. Furthermore, each graded quotient $I^i/I^{i-1}$ is a direct sum of copies of $\Delta_i$ with multiplicity $s(n, i) = \binom{n-1}{i-1}$, the number of linear order-preserving surjections from $[n]$ to $[i]$: this is clear for $\OA$, while for $\BA$, note that the group $G_i$ acts freely on $\C^-(n, i)$ from the left. Since each $\Delta_i$ is projective for $i < n$ by the induction hypothesis, we deduce that
\[
\mathfrak{I}_n \cong \bigoplus_{i < n} (\Delta_i)^{s(n, i)}.
\]

Applying $\Hom_{k\C} (-, \Delta_r)$ for a fixed $r < n$ to the short exact sequence
\[
0 \longrightarrow \bigoplus_{i < n} (\Delta_i)^{s(n, i)} \longrightarrow P_n \longrightarrow \Delta_n \longrightarrow 0,
\]
we obtain an associated exact sequence
\[
0 \longrightarrow \Hom(\Delta_n, \, \Delta_r) \longrightarrow \Hom(P_n, \, \Delta_r) \longrightarrow \Hom(\bigoplus_{i < n} (\Delta_i)^{s(n, i)}, \, \Delta_r) \longrightarrow \Ext^1_{k\C} (\Delta_n, \, \Delta_r) \longrightarrow 0.
\]

\medskip
\textbf{Case $\OA$.} The second Hom-space has dimension
\[
\dim_k \Hom_{k\C} (P_n, \, \Delta_r) = \dim_k \Delta_r(n) = \binom{n}{r}.
\]
On the other hand, by Theorem \ref{hom spaces}, for $r = n - 1$, one has
\[
\dim_k \Hom_{k\C} (\Delta_n, \, \Delta_{n-1}) + \dim_k \Hom_{k\C} ((\Delta_{n-1})^{s(n, n-1)}, \, \Delta_{n-1}) = 1 + n-1 = \binom{n}{n-1};
\]
for $1 \leqslant r < n-1$, the first Hom-space is 0, and the third Hom-space has dimension
\[
\dim_k \Hom_{k\C} ((\Delta_r)^{s(n, r)}, \, \Delta_r) + \dim_k \Hom_{k\C} ((\Delta_{r+1})^{s(n, r+1)}, \, \Delta_r) = \binom{n-1}{r-1} + \binom{n-1}{r} = \binom{n}{r},
\]
by Pascal's identity. Thus $\Ext^1_{k\C} (\Delta_n, \, \Delta_r)=0$ for each $r < n$, and hence $\Ext^1_{k\C} (\Delta_n, \, \mathfrak{I}_n) = 0$ as well. Consequently, the original short exact sequence splits, so $\Delta_n$ is  projective.

\medskip

\textbf{Case $\BA$.} Since the number of connected components in $\Gamma_{r, r+1}$ is 1 (for $r = 1$) or 2 (for $r \geqslant 2$), we need to handle these two cases separately. For $r = 1$, dimensions of all Hom-spaces are identical to those in the $\OA$ case; for $r \geqslant 2$, dimensions of all Hom-spaces are twice of those in the $\OA$ case. The conclusion then follows by the same dimension-counting and exactness argument.
\end{proof}

In the rest of this section let $k$ be a field. For a fixed object $[n]$ in $\C$, denote by $\Irr(G_n)$ the set of irreducible representations (up to isomorphism) of $G_n$. For a fixed $\lambda \in \Irr(G_n)$, let $e_{\lambda} \in kG_n$ be the primitive idempotent such that $kG_n e_{\lambda}$ is the projective cover of $\lambda$. Define
\begin{itemize}
\item $P_{n, \lambda} = P_n \otimes_{kG_n} kG_n e_{\lambda}$;

\item $\mathfrak{I}_{n, \lambda} = \mathfrak{I}_n \otimes_{kG_n} kG_n e_{\lambda}$;

\item $\Delta_{n, \lambda} = P_{n, \lambda} / \mathfrak{I}_{n, \lambda} \cong \Delta_n \otimes_{kG_n} kG_n e_{\lambda}$.
\end{itemize}
Here the isomorphism holds since $\Delta_n$ is a right free $kG_n$-module, so the functor $\Delta_n \otimes_{kG_n} -$ is exact. Indeed, for each $l \geqslant 1$, the value $\Delta_n (l)$ of $\Delta_n$ on $[l]$ is isomorphic to $k\C^+(n, l)$, which is a right free $kG_n$-module since $G_n$ acts freely on $\C^+(n, l)$ from the right side. Call $\Delta_{n, \lambda}$ an \textit{indecomposable standard module} as it is an indecomposable direct summand of $\Delta_n$ corresponding the primitive idempotent $e_{\lambda}$.

\begin{proposition} \label{filtration}
Notation as above. Then:
\begin{enumerate}
\item $P_{n, \lambda}$ has a finite filtration
\[
0 = I^0 \subseteq I^1 \subseteq I^2 \subseteq \ldots \subseteq I^n = P_{n, \lambda}
\]
such that $I^n / I^{n-1} \cong \Delta_{n, \lambda}$ and $I^i/I^{i-1}$ is a finite coproduct of some $\Delta_{i, \mu}$ with $i < n$;

\item isomorphism classes of irreducible $\C$-modules are parameterized by the set
\[
\bigsqcup_{n \geqslant 1} \Irr(G_n);
\]

\item $\Delta_{n, \lambda}$ has a simple top isomorphic to the irreducible $\C$-module $L_{n, \lambda}$ parameterized by $(n, \lambda)$.
\end{enumerate}
\end{proposition}

\begin{proof}
Statement (1) follows from the proofs of \cite[Lemma 3.10 and Proposition 3.11]{DLL}, and the last two statements follow from \cite[Theorem 3.13]{DLL} and its proof.
\end{proof}

By Theorem \ref{hom spaces}, $\Hom_{k\C} (\Delta_{n+1}, \, \Delta_n)$ has dimension 1 or 2. Thus one can find certain $\Delta_{n+1, \lambda}$ and $\Delta_{n, \mu}$ such that
\[
\Hom_{k\C} (\Delta_{{n+1}, \lambda}, \, \Delta_{n, \mu}) \neq 0.
\]
The following proposition classifies these special indecomposable standard modules.

\begin{proposition} \label{classify singular}
Let $k$ be a field with characteristic different from 2. If there is a nonzero homomorphism from $\Delta_{n+1, \lambda}$ to $\Delta_{n, \mu}$, then $(\lambda, \mu)$ must be the pair of irreducible representations appearing in the following table.
\begin{center}
\begin{tabular}{c|c|c}
category & group & irreducible representations $(\lambda,\mu)$ \\
\hline
$\FA$
& $S_{n+1},\,S_n$
& $(\sgn_{n+1},\,\sgn_n)$ \\[4pt]

$\OA$
& trivial
& $(\mathrm{triv},\,\mathrm{triv})$ \\[4pt]

$\CA$
& $C_{n+1}=\langle g \rangle, \; C_n = \langle h \rangle$
& $\lambda(g)=(-1)^n,\quad \mu(h)=(-1)^{n-1}$ \\[6pt]

$\BA$
& $C_2 (n \geqslant 2)$
& $(\mathrm{triv},\,\mathrm{triv}),\;(\sgn,\,\sgn)$ \\[6pt]

$\SA$
& $D_{2(n+1)} = \langle \sigma, \tau \rangle, \; D_{2n} = \langle \sigma, \tau \rangle$
& $(\lambda^+,\mu^+), \; (\lambda^-,\mu^-)$ \\[2pt]
& $(n \geqslant 3)$ & $\lambda^\pm(\sigma)=(-1)^n,\;\lambda^\pm(\tau)=\pm1$ \\[2pt]
& & $\mu^\pm(\sigma)=(-1)^{n-1},\;\mu^\pm(\tau)=\pm1$
\end{tabular}
\end{center}
\end{proposition}

\begin{proof}
By the proof of Theorem \ref{hom spaces}, a nonzero homomorphism $\Delta_{n+1,\lambda} \to \Delta_{n,\mu}$ can be represented by a nonzero vector $v \in \Delta_{n,\mu}(n+1)$ annihilated by all morphisms $g \in \C^+(n+1, n)$. Furthermore, each connected component of $\Gamma_{n, n+1}$ contributes a basic vector satisfying this requirement: an alternating sum of vertices appearing in that component, so $v$ must be a linear combination of these basic vectors. Besides, using the induced-module description,
\begin{equation}\label{eq:induced}
\Delta_{n,\mu}(n+1) = k\C(n, n+1)\otimes_{kG_n}\mu,
\end{equation}
we deduce that $v \otimes_{kG_n} \mu$ in \eqref{eq:induced} is nonzero.

\medskip
We now determine all possible pairs $(\lambda,\mu)$ case by case.

\medskip
\noindent
\textbf{Case 1: $\FA$.}
Here $G_n=S_n$.
The mutation graph $\Gamma_{n, n+1}$ is connected, so one can take $v$ to be the alternating sum of all injections from $[n]$ to $[n+1]$. A direct computation shows that the left $S_{n+1}$-action on $v$ is the sign character $\sgn_{n+1}$, while the right $S_n$-action is $\sgn_n$. Thus the tensor product $v \otimes_{kG_n} \mu \neq 0$ if and only if $\mu$ is the sign representation.

\medskip
\noindent
\textbf{Case 2: $\CA$.}
Here $G_n=C_n$ and $G_{n+1}=C_{n+1}$. Again, the mutation graph $\Gamma_{n, n+1}$ is connected, so one can take $v$ to be an alternating sum of all injective morphisms from $[n]$ to $[n+1]$. Let $g$ be a generator of $C_{n+1}$ and $h$ a generator of $C_n$. A direct computation using the cyclic relations gives
\[
g \cdot v = (-1)^n v, \qquad v \cdot h = (-1)^{n-1} v.
\]
Thus $v$ affords a one-dimensional left $kC_{n+1}$-module $\lambda$ and a one-dimensional right $kC_n$-module $\mu'$. Consequently, the tensor product $v \otimes_{kG_n} \mu \neq 0$ if and only if $\mu \cong \mu'$.

\medskip
\noindent
\textbf{Case 3: $\BA$.}
For $n \geqslant 2$, $G_n\cong C_2$ and $G_{n+1}\cong C_2$. The mutation graph $\Gamma_{n, n+1}$ has two connected components: one consisting of linear order-preserving embeddings and the other of linear order-reversing embeddings. Each component contributes a vector, say $v_1$ and $v_2$, and $G_{n+1}$ permutes them. Consequently, the vectors
\[
v_+ = v_1 + v_2, \qquad v_- = v_1 - v_2
\]
span the trivial representation and the sign representation respectively. The same description applies to the right action of $G_n$. Thus $v_\pm \otimes_{kG_n} \mu \neq 0$ if and only if $\mu$ is the trivial character (for $v_+$) or the sign character (for $v_-$).

For $n = 1$, a direct computation shows that the pair must be $(\mathrm{sgn}, \mathrm{sgn})$.

\medskip
\noindent
\textbf{Case 4: $\SA$.}
For $n \geqslant 3$, $G_n = D_{2n}$ and $G_{n+1} = D_{2(n+1)}$. Let $\sigma$ and $\tau$ denote the standard rotation and reflection generators of $D_{2(n+1)}$. The mutation graph $\Gamma_{n, n+1}$ again has two connected components: the cyclic order-preserving component and the cyclic order-reversing component. The reflection $\tau$ interchanges the two mutation components, while the rotation $\sigma$ preserves each component. Each component contributes a vector, say $v_1$ and $v_2$. Then $v_+ = v_1 + v_2$ and $v_- = v_1 - v_2$ span two one-dimensional $kG_{n+1}$-modules.

A direct computation shows that the two characters of $D_{2(n+1)}$ afforded by $v_\pm$ are
\[
\lambda^\pm(\sigma) = (-1)^n, \qquad \lambda^\pm(\tau) = \pm 1.
\]
The same analysis applies to the right action of $D_{2n}$, yielding two characters $\mu^\pm$ with
\[
\mu^\pm(\sigma) = (-1)^{n-1}, \qquad \mu^\pm(\tau) = \pm 1,
\]
so $v_\pm \otimes_{kG_n} \mu \neq 0$ if and only if $\mu \cong \mu^\pm$.

For $n = 1$, a direct computation shows that the pair is $(\mathrm{sgn}, \mathrm{sgn})$; for $n = 2$, the pairs are $(\mathrm{triv}, \mu^+)$ and $(\mathrm{sgn}, \mu^-)$.
\end{proof}

\begin{remark}
The assumption that the characteristic of $k$ is distinct from 2 guarantees that $v_+$ and $v_-$ in the above proof are linearly independent. The reader can see that this assumption is not necessary for $\OA$, $\CA$ or $\FA$.

If the characteristic of $k$ is 2, then the one-dimensional representations of the dihedral groups (or cyclic groups of order 2) coincide. More precisely:
\begin{itemize}
\item for $\SA$, we have $\lambda^+ = \lambda^-$ and $\mu^+ = \mu^-$, so there is only one singular pair $(\lambda, \mu)$;

\item for $\BA$, we have $\mathrm{triv} = \sgn$, so there is only one singular pair $(\mathrm{triv}, \mathrm{triv})$.
\end{itemize}

\end{remark}

In the rest of this paper, we call irreducible representations $\lambda$ appearing in the above table \emph{singular} irreducible representations of $G_n$. The corresponding $\Delta_{n, \lambda}$ are called \emph{singular} indecomposable standard $\C$-modules.

To classify irreducible $\C$-modules, given $\lambda \in \Irr(G_n)$, we define
\[
\overline{\Delta}_{n, \lambda} = \Delta_n \otimes_{kG_n} \lambda,
\]
which is a quotient of $\Delta_{n, \lambda}$. If $kG_n$ is semisimple, then $\overline{\Delta}_{n, \lambda} = \Delta_{n, \lambda}$, but in general they are different.

\begin{proposition} \label{regular standard is simple}
Let $k$ be a field. If $\lambda \in \Irr(G_n)$ is regular, then $\overline{\Delta}_{n, \lambda}$ is irreducible.
\end{proposition}

\begin{proof}
Suppose, for contradiction, that there exists a $\lambda \in \Irr(G_n)$ such that $\overline{\Delta}_{n,\lambda}$ is not irreducible. By Corollary \ref{socle}, $\overline{\Delta}_{n,\lambda}$ has a nonzero socle. Take an irreducible $\C$-module $L_{m,\mu}$ in this socle. Then $L_{m, \mu}$ is a proper submodule of $\overline{\Delta}_{n, \lambda}$. Besides, we get a nonzero homomorphism
\[
f: \Delta_{m,\mu} \longrightarrow \overline{\Delta}_{n,\lambda}
\]
whose image is precisely $L_{m, \mu}$.

Note that the functor $\Delta_n \otimes_{kG_n} -$ is exact. Also note that $\lambda$ is the socle of $kG_n e_{\lambda}$ since $kG_n$ is a symmetric algebra. Therefore, applying $\Delta_n \otimes_{kG_n} -$ to the inclusion $\lambda \to kG_ne_{\lambda}$ we get an injective $\C$-module homomorphism
\[
g: \overline{\Delta}_{n, \lambda} = \Delta_n \otimes_{kG_n} \lambda \longrightarrow \Delta_n \otimes_{kG_n} kG_ne_{\lambda} = \Delta_{n, \lambda}.
\]
Consequently, we obtain a nonzero $\C$-module homomorphism
\[
g \circ f: \Delta_{m, \mu} \longrightarrow \Delta_{n, \lambda}.
\]
By Theorem \ref{hom spaces}, this can occur if and only if $m = n$ or $m = n+1$.

If $m = n$, we obtain a nonzero $kG_n$-module homomorphism
\[
f_n: \Delta_{n, \mu} (n) \cong kG_n e_{\mu} \longrightarrow \overline{\Delta}_{n, \lambda} (n) \cong \lambda.
\]
This happens if and only if $\mu \cong \lambda$. Accordingly, $f_n$ is surjective, so is $f$ since $\overline{\Delta}_{n, \lambda}$ is generated by $\lambda$, contradicting the assumption that $L_{m, \mu}$ is a proper submodule of $\overline{\Delta}_{n, \lambda}$. Thus $m = n + 1$. But by Proposition \ref{classify singular}, this implies that $\lambda$ is singular, contradicting the hypothesis. Therefore $\overline{\Delta}_{n,\lambda}$ must be irreducible.
\end{proof}

Now we consider the singular case.

\begin{lemma} \label{at most 2}
Let $k$ be a field. If $\lambda \in \Irr(G_n)$ is singular, then $\overline{\Delta}_{n, \lambda}$ has length at most 2.
\end{lemma}

\begin{proof}
Let $\D$ be the simplex category, viewed as a subcategory of $\C$. As in the proof of Corollary \ref{finite length 1}, we have
\[
\Res^{\C}_{\D} \Delta_n \cong \bigoplus_{\sigma \in G_n} (\Delta_n^{\D})_{\sigma} \cong \Delta_n^{\D} \otimes_k kG_n.
\]
It follows that
\[
\Res^{\C}_{\D} \overline{\Delta}_{n, \lambda} = \Res_{\D}^{\C} (\Delta_n \otimes_{kG_n} \lambda) \cong  (\Delta_n^{\D} \otimes_k kG_n) \otimes_{kG_n} \lambda \cong \Delta_n^{\D} \otimes_k \lambda \cong \Delta_n^{\D},
\]
since the singular $\lambda$ is one-dimensional as shown in Proposition \ref{classify singular}.

Note that $\Delta_n^{\D}$ is the normalized projective $\D$-module. By the Dold-Kan correspondence, it has a simple top and a simple socle, and no further composition factors. Hence its length is $2$. Since the restriction functor is exact and cannot send a nonzero module to zero (as $\D$ is a wide subcategory of $\C$), it follows that $\overline{\Delta}_{n,\lambda}$ has length at most $2$ as a $\C$-module.
\end{proof}

\begin{lemma} \label{existence of morphism}
Let $k$ be a field and $\lambda \in \Irr(G_n)$ be a singular irreducible representation. Then there exists a nonzero non-surjective homomorphism
\[
\overline{\phi}: \overline{\Delta}_{n+1, \mu} \longrightarrow \overline{\Delta}_{n, \lambda}
\]
with $\mu \in \Irr(G_{n+1})$ singular.
\end{lemma}

\begin{proof}
Let $\{ \lambda_i \mid i \in [r] \}$ be the composition factors of $kG_ne_{\lambda}$ where $e_{\lambda} \in kG_n$ is a primitive idempotent such that $kG_n e_{\lambda}$ is the projective cover of $\lambda$. Then $\Delta_{n, \lambda}$ has a filtration whose factors are exactly $\overline{\Delta}_{n, \lambda_i}$. Since $\lambda$ is singular, we can find a unique singular $\mu \in \Irr(G_{n+1})$ together with a nonzero homomorphism $\phi: \Delta_{n+1, \mu} \to \Delta_{n, \lambda}$ by Theorem \ref{hom spaces}. Therefore, there exists a nonzero homomorphism from $\Delta_{n+1, \mu}$ to $\overline{\Delta}_{n, \lambda_i}$ for a certain $i \in [r]$. Note that $\lambda_i$ must be singular. Otherwise, $\overline{\Delta}_{n, \lambda_i}$ is irreducible by Proposition \ref{regular standard is simple}. But this is impossible since the top of $\Delta_{n+1, \mu}$ is not isomorphic to $\overline{\Delta}_{n, \lambda_i}$.

By Lemma \ref{at most 2}, the length of $\overline{\Delta}_{n, \lambda_i}$ is at most 2. Since the above map from $\Delta_{n+1, \mu}$ to $\overline{\Delta}_{n, \lambda_i}$ is clearly not surjective, its image is contained in the socle of $\overline{\Delta}_{n, \lambda_i}$. Consequently, one gets a nonzero homomorphism
\[
\overline{\phi}: \overline{\Delta}_{n+1, \mu} \longrightarrow \overline{\Delta}_{n, \lambda_i},
\]
which is the desired homomorphism for $\C = \OA, \CA, \FA$ since $\lambda_i \cong \lambda$ by Proposition \ref{classify singular}: there is a unique singular irreducible representation in $\Irr(G_n)$.

For $\C = \BA$ or $\SA$, $\lambda_i$ may not be isomorphic to $\lambda$, but the other singular irreducible $\lambda'$ twisted by a sign. In this case, we repeat the above argument for $\lambda'$ to get a nonzero homomorphism
\[
\overline{\phi}': \overline{\Delta}_{n+1, \mu'} \longrightarrow \overline{\Delta}_{n, \lambda'_i},
\]
where $\lambda_i' \in \Irr(G_n)$ is singular and $\mu'$ is not isomorphic to $\mu$. Then $\lambda_i'$ is isomorphic to $\lambda$ and we get the desired map. Indeed, if this is not the case, then $\lambda_i'$ is isomorphic to $\lambda'$ as well. But then we get two nonzero homomorphisms $\overline{\phi}$ and $\overline{\phi}'$ to $\overline{\Delta}_{n, \lambda'}$, whose images land in distinct simple submodules of the socle. Consequently, the length of $\overline{\Delta}_{n, \lambda'}$ is at least 3, contradicting Lemma \ref{at most 2}.
\end{proof}

Combining these two lemmas, we get:

\begin{proposition} \label{length 2}
Let $k$ be a field. If $\lambda \in \Irr(G_n)$ is singular, then $\overline{\Delta}_{n, \lambda}$ has length 2.
\end{proposition}

\begin{proof}
Lemma \ref{at most 2} tells us that $\overline{\Delta}_{n, \lambda}$ has length at most 2, and Lemma \ref{existence of morphism} asserts that its length is at least 2. The conclusion follows.
\end{proof}

These two propositions bring a complete classification of irreducible $\C$-modules.

\begin{theorem} \label{classification of irreducibles}
Let $k$ be a field. Then the isomorphism classes of irreducible $\C$-modules are parameterized by elements in the set
\[
\{ (n, \lambda) \mid n \in \mathbb{Z}_+, \, \lambda \in \Irr(G_n) \}.
\]

More precisely, we have the following cases:
\begin{enumerate}
\item if $\lambda$ is regular, then the corresponding irreducible $\C$-module is isomorphic to $\overline{\Delta}_{n, \lambda}$;

\item if $\lambda$ is singular, then the corresponding irreducible $\C$-module is the top of $\overline{\Delta}_{n, \lambda}$.
\end{enumerate}
\end{theorem}

From this result we can deduce an explicit description of the Grothendieck group.

\begin{corollary}
Let $k$ be a field. Then the Grothendieck group $K_0(\C)$ of the category of finitely generated $\C$-modules has the following canonical decomposition:
\[
K_0(\C) \;\cong\; \bigoplus_{n \geqslant 1} R(G_n),
\]
where $R(G_n)$ is the representation ring of $G_n$. In addition, the following set forms a basis of $K_0(\C)$:
\[
\{ [\Delta_{n, \lambda}] \mid n \in \mathbb{Z}_+, \, \lambda \in \Irr(G_n) \}.
\]
\end{corollary}

\begin{proof}
By Corollary \ref{finite length 1}, the category of finitely generated $\C$-modules has the Jordan-Hölder property. Thus, $K_0(\C)$ is the free abelian group generated by the isomorphism classes of simple modules. It immediately follows from Theorem \ref{classification of irreducibles} that
\[
K_0(\C) \cong \bigoplus_{n \geqslant 1} R(G_n).
\]

Next we show that the set $\{ [\Delta_{n, \lambda}] \mid n \in \mathbb{Z}_+, \, \lambda \in \Irr(G_n) \}$ is a basis as well. By Proposition \ref{regular standard is simple} and \ref{length 2}, if $\lambda \in \Irr(G_n)$ is regular, then $[ \Delta_{n,\lambda}] = [L_{n,\lambda}]$; otherwise,
\[
[\Delta_{n,\lambda}] = [L_{n,\lambda}] + [L_{n+1,\mu}]
\]
by the proof of Proposition \ref{length 2}. Therefore, the transition matrix between $\{[L_{n, \lambda}]\}$ and $\{[\Delta_{n, \lambda}]\}$ is upper unitriangular. Since an upper unitriangular matrix (even of infinite size, as long as the degrees are bounded below) is invertible over $\mathbb{Z}$, the classes $\{ [\Delta_{n, \lambda}] \}$ form an alternative basis for $K_0(\C)$.
\end{proof}

\section{Further results for $\CA$, $\SA$ and $\FA$}

In the case that $k$ is a field, the representation theory of $\OA$ and $\BA$ is very transparent. Thus in this section we only consider representations of the other three categories.

\subsection{$\CA$ and $\SA$} In this subsection let $\C$ be $\CA$ or $\SA$. First we compute extensions between standard modules. Note that $\Delta_n$ is projective if $n \leqslant 2$.

\begin{proposition} \label{ext groups}
Let $k$ be a field and $\C$ be $\CA$ or $\SA$. Then for all $m \geqslant 3$, one has
\[
\Ext^1_{k\C}(\Delta_m,\, \Delta_n) \; \cong \;
\begin{cases}
k^d & m = n+1,\\
k^d & m = n+2,\\
0 & \text{else},
\end{cases}
\]
where $d$ is the number of connected components in the mutation graph $\Gamma_{n, n+1}$.
\end{proposition}

\begin{proof}
By Proposition \ref{standard modules}, the extension group vanishes if $m \leqslant n$, so we assume that $m > n$. Applying $\Hom_{k\C} (-, \, \Delta_n)$ to the short exact sequence
\[
0 \to \Delta_{m-1} \to k\C \Psi_m \to \Delta_m \to 0
\]
given by Proposition \ref{filtration of projectives}, we obtain an associated exact sequence
\[
0 \to \Hom_{k\C} (\Delta_m, \, \Delta_n) \to \Hom_{k\C} (k\C \Psi_m, \, \Delta_n) \to \Hom_{k\C} (\Delta_{m-1}, \, \Delta_n) \to \Ext_{k\C}^1 (\Delta_m, \, \Delta_n) \to 0.
\]

\medskip
\noindent
\textbf{Case $m > n+2$.} By Theorem \ref{hom spaces}, the third term vanishes, so does the extension group.

\medskip
\noindent
\textbf{Case $m = n+2$.} In this case we have $\Hom_{k\C} (k\C \Psi_{n+2}, \, \Delta_n) = 0$. To see it, one applies the exact functor $\Hom_{k\C} (k\C \Psi_{n+2}, \, -)$ to the short exact sequence
\[
0 \to \Delta_{n-1} \to k\C \Psi_n \to \Delta_n \to 0
\]
given by Proposition \ref{filtration of projectives} and note that $\Hom_{k\C} (k\C \Psi_{n+2}, \, k\C \Psi_n) = 0$ by Proposition \ref{thin category}. Thus,
\[
\Ext_{k\C}^1 (\Delta_{n+2}, \, \Delta_n) \cong \Hom_{k\C} (\Delta_{n+1}, \, \Delta_n) \cong k^d
\]
by Theorem \ref{hom spaces}.

\medskip
\noindent
\textbf{Case $m = n+1$.} In this case, by Theorem \ref{hom spaces}
\[
\Hom_{k\C} (\Delta_{n+1}, \, \Delta_n) \cong k^d, \quad \Hom_{k\C} (\Delta_n, \, \Delta_n) \cong kG_n,
\]
and
\[
\Hom_{k\C} (k\C \Psi_{n+1}, \, \Delta_n) \cong \Psi_{n+1} \Delta_n = \Psi_{n+1} \Delta_n(n+1)
\]
which is a vector space spanned by $\Psi_{n+1} d_{n+1} \sigma$ with $\sigma \in G_n$ since $\Psi_{n+1}$ kills all $d_i$ with $i \leqslant n$ by Lemma \ref{idempotent}. As in the proof of Corollary \ref{ideal description}, one can show that
\[
kG_n \to (\Psi_{n+1} d_{n+1}) kG_n, \quad \sigma \mapsto \Psi_{n+1} d_{n+1} \sigma
\]
is an isomorphism of right $kG_n$-module. Thus $\Psi_{n+1} \Delta_n \cong kG_n$. By comparing dimensions, one deduces that $\Ext_{k\C}^1 (\Delta_{n+1}, \, \Delta_n) \cong k^d$.
\end{proof}

Since $d$ is either 1 or $2$, this computation tells us that most indecomposable standard modules are actually projective. More precisely, fix $n \geqslant 1$, and let
\[
\Irr(G_n) = \{\lambda_1, \ldots, \lambda_r\}, \quad \Irr(G_{n+1}) = \{\mu_1, \ldots, \mu_s\}.
\]
Then we have
\[
k^d \cong \Ext_{k\C}^1 (\Delta_{n+1}, \, \Delta_n) = \bigoplus_{i=1}^r \Ext_{k\C}^1 ((\Delta_{n+1, \mu_i})^{a_i}, \, \Delta_n)
\]
Since $d = 1$ or $2$, there are at most two irreducible representations $\mu_i$ of $kG_{n+1}$ such that
\[
\Ext_{k\C}^1 (\Delta_{n+1, \mu_i}, \, \Delta_n) \neq 0.
\]
For all other $\mu_i$, one has
\[
\Ext_{k\C}^1 (\Delta_{n+1, \mu_i}, \, \Delta_n) = 0.
\]
It follows from Proposition \ref{filtration of projectives} that these $\Delta_{n+1, \mu_i}$ are projective $\C$-modules.

The following result classifies those indecomposable standard modules which are not projective.

\begin{proposition} \label{ext for CA and SA}
Let $\lambda \in \Irr(G_n)$ and $\mu \in \Irr(G_{n+1})$. Then $\Ext_{k\C}^1 (\Delta_{n+1, \mu}, \, \Delta_{n, \lambda}) \neq 0$ if and only if $(\mu, \lambda)$ is a pair of singular irreducible representations appearing in Proposition \ref{classify singular}.
\end{proposition}

\begin{proof}
We give a detailed proof for $\C = \CA$. Since
\[
\Ext_{k\C}^1 (\Delta_{n+1}, \, \Delta_n) = \bigoplus_{i=1}^r \bigoplus_{j=1}^s \Ext_{k\C}^1 ((\Delta_{n+1, \mu_j})^{a_j}, \, (\Delta_{n, \lambda_i})^{b_i})
\]
has dimension 1, it suffices to show the following conclusion: if $(\mu, \lambda)$ is a pair of singular irreducible representations, then
\[
\Ext_{k\C}^1 (\Delta_{n+1, \mu}, \, \Delta_n) \neq 0 \neq \Ext_{k\C}^1 (\Delta_{n+1}, \, \Delta_{n, \lambda}).
\]

Assume $\mu \in \Irr(G_{n+1})$ and $\lambda \in \Irr(G_n)$ are singular. Applying $\Hom_{k\C} (-, \, \Delta_{n, \lambda})$ to the sequence
\[
0 \to \mathfrak{I}_{n+1} \to P_{n+1} \to \Delta_{n+1} \to 0,
\]
we obtain another exact sequence
\begin{align*}
0 \to \Hom_{k\C}(\Delta_{n+1}, \, \Delta_{n, \lambda}) \longrightarrow \Hom_{k\C} (P_{n+1}, \, \Delta_{n, \lambda})\\
\to \Hom_{k\C} (\mathfrak{I}_{n+1}, \, \Delta_{n, \lambda}) \to \Ext_{k\C}^1 (\Delta_{n+1}, \, \Delta_{n, \lambda}) \to 0.
\end{align*}
Now we compute the dimension of each term.

The first term has dimension 1 by Theorem \ref{hom spaces}. The second one is isomorphic to
\[
\Delta_{n, \lambda} (n + 1) \cong k\C^+(n, n+1) \otimes_{kG_n} kG_n e_{\lambda},
\]
which has dimension $(n+1) \dim_k (kG_n e_{\lambda})$ since $k\C^+(n, n+1)$ is a right free $kG_n$-module of rank $n+1$. For the third term, we note that $\mathfrak{I}_{n+1}$ has a short exact sequence
\[
0 \to V \to \mathfrak{I}_{n+1} \to k\C \otimes_{k\C^-} k\C^-(n+1, n) \to 0,
\]
where $V$ has a filtration whose factors are of the form $\Delta_m$ with $m < n$. Furthermore,
\[
k\C \otimes_{k\C^-} k\C^-(n+1, n) \cong k\C \otimes_{k\C^-} (kG_n \otimes_{kG_n} k\C^-(n+1, n)) \cong (k\C \otimes_{k\C^-} kG_n)^{n+1} \cong (\Delta_n)^{n+1}.
\]
Thus by Proposition \ref{standard modules},
\[
\Hom_{k\C} (\mathfrak{I}_{n+1}, \, \Delta_{n, \lambda}) \cong \Hom_{k\C} (k\C \otimes_{k\C^-} k\C^-(n+1, n), \, \Delta_{n, \lambda}) \cong \Hom_{k\C} ((\Delta_n)^{n+1}, \, \Delta_{n, \lambda})
\]
has dimension $(n+1) \dim_k (kG_n e_{\lambda})$. Consequently, $\Ext_{k\C}^1 (\Delta_{n+1}, \, \Delta_{n, \lambda})$ has dimension 1, so it is nonzero as claimed.

\medskip
Let $e_{\mu} \in kG_{n+1}$ be a primitive idempotent such that $kG_{n+1} e_\mu$ is the projective cover of $\mu$. Applying $\Hom_{k\C} (-, \, \Delta_n)$ to the short exact sequence
\[
0 \to \mathfrak{I}_{n+1}e_{\mu} \to k\C e_{\mu} \to \Delta_{n+1, \mu} \to 0,
\]
we obtain
\begin{align*}
0 \to \Hom_{k\C}(\Delta_{n+1, \mu}, \, \Delta_n) \longrightarrow \Hom_{k\C} (k\C e_{\mu}, \, \Delta_n)\\
\to \Hom_{k\C} (\mathfrak{I}_{n+1}e_{\mu}, \, \Delta_n) \to \Ext_{k\C}^1 (\Delta_{n+1, \mu}, \, \Delta_n) \to 0.
\end{align*}
Let us compute dimensions again.

The first term has dimension 1 by Theorem \ref{hom spaces}. The second term is naturally isomorphic to
\[
e_\mu \Delta_n(n+1) \cong \Hom_{kG_{n+1}}(kG_{n+1}e_\mu, \, \Delta_n(n+1)),
\]
which has dimension $n \dim_k (kG_{n+1} e_{\mu})$ since $\Delta_n(n+1) \cong k\C^+(n,n+1)$ is a free left $kG_{n+1}$-module of rank $n$. Finally, $\mathfrak{I}_{n+1}e_\mu$ admits a filtration whose only contribution to $\Hom(-, \, \Delta_n)$ comes from
\[
k\C \otimes_{k\C^-} k\C^-(n+1,n)e_{\mu} \cong k\C \otimes_{k\C^-} (kG_n \otimes_k kG_{n+1} e_{\mu}) \cong (k\C \otimes_{k\C^-} kG_n) \otimes_k kG_{n+1} e_{\mu} \cong \Delta_n \otimes_k kG_{n+1} e_{\mu},
\]
where the first isomorphism holds as $k\C^-(n+1,n)$ is a free $(kG_n, kG_{n+1})$-bimodule of rank $1$. Thus
\[
\Hom_{k\C} (\mathfrak{I}_{n+1} e_{\mu}, \, \Delta_n) \cong \Hom_{k\C} (\Delta_n, \, \Delta_n \otimes_k kG_{n+1} e_{\mu})
\]
has dimension $n \dim_k (kG_{n+1} e_{\mu})$, so $\Ext_{k\C}^1(\Delta_{n+1,\mu}, \, \Delta_n)$ has dimension 1, which is nonzero as claimed.

\medskip
For the category $\SA$, the argument proceeds similarly with a refined look at the bimodule structure of the morphisms. For $n \geqslant 3$, $G_{n+1} = D_{2(n+1)}$ and $G_n = D_{2n}$ are dihedral, so $k\C^-(n+1, n)$ is not a free bimodule. Instead, since the action of $G_n \times G_{n+1}$ on $\C^-(n+1, n)$ is transitive and the stabilizer is a subgroup of order 2, one has
\[
k\C^-(n+1, n) \cong kG_n \otimes_{kC_2} kG_{n+1},
\]
as bimodules, where $C_2$ is the common subgroup of order 2 generated by the reflection. Under this identification, applying the singular idempotent $e_{\mu}$ yields
\[
k\C^-(n+1, n)e_{\mu} \cong kG_n \otimes_{kC_2} kG_{n+1} e_{\mu},
\]
which has dimension $n \dim_k (kG_{n+1} e_{\mu})$ since $kG_n$ is a free right $kC_2$-module of rank $n$. Consequently, the dimension of $\Hom_{k\C}(\mathfrak{I}_{n+1}e_\mu, \, \Delta_n)$ remains $n  \dim_k (kG_{n+1} e_{\mu})$, preserving the exact dimension balance required to show that $\Ext^1_{k\C} (\Delta_{n+1, \mu}, \, \Delta_{n, \lambda})$ is one-dimensional.
\end{proof}

As an immediate consequence, we have:

\begin{corollary} \label{projective cover}
Let $k$ be a field and $\lambda \in \Irr(G_n)$. One has:
\begin{enumerate}
\item if $\lambda$ is regular, then $\Delta_{n, \lambda}$ is projective;

\item if $n = 1$, then $\Delta_{n, \lambda}$ is projective of length 2;

\item if $n \geqslant 2$ and $\lambda$ is singular, then the projective cover $P^c_{n, \lambda}$ of $\Delta_{n, \lambda}$ has a filtration with factors $\Delta_{n, \lambda}$ and $\Delta_{n-1, \mu}$, where $(\mu, \lambda)$ is a pair appearing in Proposition \ref{classify singular}.
\end{enumerate}
\end{corollary}

\begin{proof}
Since there is a surjective homomorphism $k\C \Psi_n \twoheadrightarrow \Delta_{n, \lambda}$, it follows that $P^c_{n, \lambda}$ is a direct summand of $k\C \Psi_n$. Consequently, it has a filtration with factors $\Delta_{n, \lambda}$ and some $\Delta_{n-1, \mu}$'s by Proposition \ref{filtration of projectives}. Furthermore, since $\Delta_{n, \lambda}$ has a simple top, we deduce that $P_{n, \lambda}^c$ is indecomposable.

(1) By Proposition \ref{filtration}, $P_{n, \lambda}$ has a filtration by $\Delta_{n, \lambda}$ and some $\Delta_{i, \mu}$ with $i < n$. But Proposition \ref{ext for CA and SA} tells us that there is no nontrivial extension of $\Delta_{n, \lambda}$ by other indecomposable standard modules. Consequently, $\Delta_{n, \lambda}$ is a direct summand of $P_{n, \lambda}$, and hence is projective.

(2) This is clear since $\Delta_{1, \lambda}  \cong k\C(1, -)$.

(3) By Proposition \ref{ext for CA and SA}, there is only a unique singular $\mu \in \Irr(G_{n-1})$ with multiplicity one such that $\Delta_{n-1, \mu}$ can occur in the filtration.
\end{proof}

If $k$ has characteristic 0, we obtain the following result, a refined version of Theorem \ref{generalized D-K correspondence}.

\begin{theorem} \label{DK for CA and SA}
Let $k$ be a field of characteristic 0, and $\lambda \in \Irr(G_n)$. One has:
\begin{enumerate}
\item if $\lambda$ is regular, then $\Delta_{n, \lambda}$ is projective and irreducible;

\item if $\lambda$ is singular and $n = 1$, then $\Delta_{n, \lambda}$ is projective of length 2;

\item if $\lambda$ is singular and $n \geqslant 2$, then $P^c_{n, \lambda}$ is of length 4, and contains a submodule $\Delta_{n-1, \mu}$ where $(\lambda, \mu)$ is a pair of singular representations appearing in Proposition \ref{classify singular}.
\end{enumerate}

\medskip

For $\C = \CA$, the category $\C \Mod$ is equivalent to the representation category of a $k$-linear category $\mathscr{K}$, which is a direct sum of two parts:
\begin{itemize}
\item $\mathscr{K}^{\mathrm{reg}}$ with vertices $(n, \lambda)$ and no arrows, where $\lambda \in \Irr(G_n)$ is regular;

\item $\mathscr{K}^{\mathrm{sing}}$ defined by the following quiver
\[
\xymatrix{
(1, \mu) \ar[r]^{\mathbf{d}_2} &
(2, \mu)  \ar@/^1pc/[r]^{\mathbf{d}_3} &
(3, \mu) \ar@/^1pc/[r]^{\mathbf{d}_4} \ar@/^1pc/[l]^{\mathbf{s}_3} &
\cdots \ar@/^1pc/[l]^{\mathbf{s}_4}
}
\]
modulo the relations:
\[
\mathbf{d}_{n+1} \mathbf{d}_n = 0, \quad \mathbf{s}_n \mathbf{s}_{n+1} = 0, \quad \mathbf{s}_n \mathbf{d}_n = 0,
\]
where $\mu \in \Irr(G_n)$ is singular.
\end{itemize}

\medskip

For $\C = \SA$, the category $\C \Mod$ is equivalent to the representation category of a $k$-linear category $\mathscr{K}$, which is a direct sum of two parts:
\begin{itemize}
\item $\mathscr{K}^{\mathrm{reg}}$ with vertices $(n, \lambda)$ and no arrows, where $\lambda \in \Irr(G_n)$ is regular;

\item $\mathscr{K}^{\mathrm{sing}}$ defined by the following quiver
\[
\xymatrix{
(1, \mu^+) \ar[r]^{\mathbf{d}_2} &
(2, \mu^+)  \ar@/^1pc/[r]^{\mathbf{d}_3} &
(3, \mu^+) \ar@/^1pc/[r]^{\mathbf{d}_4} \ar@/^1pc/[l]^{\mathbf{s}_3} &
\cdots, \ar@/^1pc/[l]^{\mathbf{s}_4} &
(2, \mu^-)  \ar@/^1pc/[r]^{\mathbf{d}_3} &
(3, \mu^-) \ar@/^1pc/[r]^{\mathbf{d}_4} \ar@/^1pc/[l]^{\mathbf{s}_3} &
\cdots, \ar@/^1pc/[l]^{\mathbf{s}_4}
}
\]
modulo the relations:
\[
\mathbf{d}_{n+1} \mathbf{d}_n = 0, \quad \mathbf{s}_n \mathbf{s}_{n+1} = 0, \quad \mathbf{s}_n \mathbf{d}_n = 0,
\]
where $\mu^+, \mu^- \in \Irr(G_n)$ are singular.
\end{itemize}
\end{theorem}

\begin{proof}
The equivalence is given by the standard Morita equivalence with respect to the family of projective generators consisting of projective covers $P^c_{n, \lambda}$ of $\Delta_{n, \lambda}$ with $n \in \mathbb{Z}_+$ and $\lambda \in \Irr(G_n)$. All statements concerning projectivity, irreducibility, and Loewy length have already been established in Propositions \ref{regular standard is simple}, \ref{length 2}, and Corollary \ref{projective cover}. Furthermore, if $\lambda \in \Irr(G_n)$ is regular and $\mu \in \Irr(G_m)$ is singular, then one has
\[
\Hom_{k\C} (\Delta_{n, \lambda}, \, P^c_{m, \mu}) = 0 = \Hom_{k\C} (P^c_{m, \mu}, \, \Delta_{n, \lambda})
\]
by Proposition \ref{classify singular} and Corollary \ref{projective cover}. It therefore remains to describe the singular block $\mathscr K^{\mathrm{sing}}$ and to determine its relations. We treat the case $\C=\CA$; the case $\SA$ is entirely analogous.

Fix a singular representation $\lambda \in \Irr(G_n)$ with $n \geqslant 2$. By Proposition \ref{length 2}, Theorem \ref{classification of irreducibles}, and Corollary \ref{projective cover}, $P^c_{n,\lambda}$ has length $4$ with composition factors
\[
L_{n,\lambda} \text{ (top)}, \quad
L_{n-1,\lambda}, \quad
L_{n+1,\lambda}, \quad
L_{n,\lambda} \text{ (socle)}.
\]
Thus one has
\[
\dim_k \Hom_{k\C}(P^c_{n,\lambda}, \, P^c_{n-1,\lambda}) = 1,
\qquad
\dim_k \Hom_{k\C}(P^c_{n-1,\lambda}, \, P^c_{n,\lambda}) = 1,
\]
and
\[
\dim_k \End_{k\C}(P^c_{n,\lambda}) = 2.
\]

Let $\mathbf{s}_n: P^c_{n,\lambda} \to P^c_{n-1,\lambda}$ be a nonzero lift of the canonical map $\Delta_{n,\lambda} \to \Delta_{n-1,\lambda}$. Such a lift exists and is unique up to scalar since $P^c_{n,\lambda}$ is projective and the above Hom-space is one-dimensional. Define $\mathbf{d}_n: P^c_{n-1, \lambda} \to P^c_{n,\lambda}$ to be the composite
\[
P^c_{n-1,\lambda} \twoheadrightarrow \Delta_{n-1,\lambda} \hookrightarrow P^c_{n,\lambda}.
\]
The composite $\mathbf{d}_n \mathbf{s}_n$ induces the nonzero endomorphism of $P^c_{n,\lambda}$, which is the composite
\[
P_{n, \lambda}^c \twoheadrightarrow \Delta_{n,\lambda} \to \Delta_{n-1,\lambda} \hookrightarrow P^c_{n,\lambda}.
\]
Thus $\mathbf{d}_n \mathbf{s}_n$ is a nonzero, non-invertible endomorphism of $P^c_{n,\lambda}$. On the other hand, the composite
\[
\mathbf{s}_n \mathbf{d}_n: P^c_{n-1,\lambda} \twoheadrightarrow \Delta_{n-1,\lambda} \hookrightarrow P^c_{n,\lambda} \to P^c_{n-1, \lambda}
\]
is 0 since the last map is the lift of the natural map
\[
\Delta_{n, \lambda} = P^c_{n, \lambda} / \Delta_{n-1, \lambda} \longrightarrow \Delta_{n-1, \lambda}.
\]

By Proposition~\ref{thin category}, there are no nonzero morphisms $P^c_{m,\lambda} \to P^c_{m',\lambda}$ when $|m-m'| \geqslant 2$. In particular, we have
\[
\mathbf d_{n+1}\mathbf d_n = 0,
\qquad
\mathbf s_n\mathbf s_{n+1} = 0.
\]

Finally, since all Hom-spaces between indecomposable projective objects in the singular block are at most one-dimensional (except for $\End(P_{n,\lambda})$, which is two-dimensional), it follows that the morphisms $\mathbf{s}_\bullet$ and $\mathbf{d}_\bullet$ generate all morphisms in the singular block. The relations among them are precisely those listed above. This identifies the singular block $\mathscr{K}^{\mathrm{sing}}$ with the stated bound quiver and completes the proof.
\end{proof}

We can use this equivalence to classify indecomposable injective $\C$-modules.

\begin{corollary}
Let $k$ be a field of characteristic 0, and let $I_{n, \lambda}$ be the injective hull of the irreducible $\C$-module $L_{n, \lambda}$. One has:
\begin{enumerate}
\item if $\lambda$ is regular, then $I_{n, \lambda} \cong \Delta_{n, \lambda} \cong L_{n, \lambda}$;

\item if $\lambda$ is singular, then under the equivalence in Theorem \ref{DK for CA and SA}, $I_{n, \lambda}$ is the $\C$-module corresponding to $D\Hom_{\mathscr{K}} (-, \, (n, \lambda))$ where $D = \Hom_k(-, k)$.
\end{enumerate}
\end{corollary}

\begin{proof}
If $\lambda$ is regular, then the conclusion follows from the decomposition $\mathscr{K} = \mathscr{K}^{\mathrm{reg}} \oplus \mathscr{K}^{\mathrm{sing}}$ and the fact that $\mathscr{K}^{\mathrm{reg}}$ is semisimple. If $\lambda$ is singular, then the conclusion follows from the fact that $\Hom_{\mathscr{K}} (-, \, (n, \lambda))$ is a finite dimensional right projective $\mathscr{K}$-module.
\end{proof}

Consequently, the category of finitely generated $\C$-modules has enough injective objects.

\subsection{$\FA$}

In this subsection let $\C = \FA$ and $k$ be a field of characteristic 0. Note that the category $\C^-$ of finite sets and surjections is locally finite in the sense that for each object, there are only finitely many morphisms in $\C^-$ starting at it. Consequently, every representable functor $k\C^-(n, -)$ is of finite dimensional. Thus one can apply the algorithm in \cite[Section 4]{LiHereditary} to construct the ordinary quiver of the category algebra $k\C^-$, whose vertices are pairs $(n, \lambda)$ with $n \in \mathbb{Z}_+$ and $\lambda \in \Irr(S_n)$. Note that singular irreducible representations are sign representations by Proposition \ref{classify singular}.

\begin{lemma}\label{lem:sgn-isolated-FAminus}
There is no arrow starting at the vertex $(n,\mathrm{sgn})$ in the ordinary quiver of $k\C^{-}$.
\end{lemma}

\begin{proof}
We apply the quiver construction algorithm described in \cite[Section 4]{LiHereditary}. In $\C^{-}$, the only unfactorizable morphisms with source $[n]$ are surjections from $[n]$ to $[n-1]$. Up to the action of automorphism groups, all such surjections are equivalent, so we can choose the standard surjection $\alpha$ given by
\[
\alpha(i) = i \ (1\leqslant i \leqslant n-1),\qquad \alpha(n) = \alpha(n-1) = n-1.
\]

\smallskip

Let $G = S_n$, $H = S_{n-1}$, and $G_0 \leqslant G$ (resp.\ $H_0 \leqslant H$) be the stabilizer of $\alpha$ under right (resp., left) composition. By a direct computation, one has $G_0\cong S_2$, generated by the transposition exchanging $n-1$ and $n$, and $H_0$ is the trivial group. Similarly, let $G_1 \leqslant G$ (resp.\ $H_1 \leqslant H$) be the subgroup preserving the fibers of $\alpha$, namely
\[
G_1 = \{ g \in G \mid \alpha g \in H \alpha \}, \quad H_1 = \{h \in H \mid h \alpha \in \alpha G\}.
\]
One checks that
\[
G_1 \cong S_{n-2} \times S_2, \qquad H_1 \cong S_{n-2},
\]
and that there is a canonical identification
\[
G_1/G_0 \cong H_1/H_0 \cong S_{n-2}.
\]

According to the algorithm in \cite[Section 4]{LiHereditary}, there exists an arrow from $(n, \lambda)$ to $(n-1, \mu)$ in the ordinary quiver if and only if $\lambda \downarrow_{G_1}^G$ and $\mu \downarrow_{H_1}^H$ share common simple summands appearing in the induced module
\[
k \uparrow^{G_1}_{G_0} = kG_1 \otimes_{kG_0} k \cong kH_1 \otimes_{kH_0} k = k \uparrow^{H_1}_{H_0}.
\]
However, restricting to $G_1\cong S_{n-2} \times S_2$, we have
\[
\mathrm{sgn}_n \! \downarrow_{G_1}^G \;\cong\; \mathrm{sgn}_{n-2} \boxtimes \mathrm{sgn}_2,
\]
which cannot share a common simple summand with $k \uparrow^{G_1}_{G_0}$ since
\[
\Hom_{kG_1} (k \uparrow^{G_1}_{G_0}, \, \mathrm{sgn}_n \! \downarrow_{G_1}^G) \cong \Hom_{kG_0} (k, \, \mathrm{sgn}_n \! \downarrow_{G_1}^G \downarrow^{G_1}_{G_0}) \cong \Hom_{kG_0} (k, \, \mathrm{sgn}_2) = 0.
\]
Thus there is no arrow in the ordinary quiver of $k\C^{-}$ starting at $(n,\mathrm{sgn})$.
\end{proof}

Consequently, different from $\CA$ or $\SA$, singular indecomposable standard modules for $\FA$ are projective.

\begin{proposition}\label{disjoint}
Let $e_{\mathrm{sgn}} \in kG_n$ be the idempotent corresponding to the sign representation. Then $\Delta_{n, \mathrm{sgn}} \cong k\C e_{\mathrm{sgn}}$. Furthermore, if $\lambda \in \Irr(G_m)$ is regular, then
\[
\Hom_{k\C}(P^c_{m,\lambda}, \, \Delta_{n, \mathrm{sgn}}) = 0.
\]
\end{proposition}

\begin{proof}
By Lemma~\ref{lem:sgn-isolated-FAminus}, there is no arrow starting at the vertex $(n, \mathrm{sgn})$ in the ordinary quiver of $k\C^-$, so $kG_n e_{\mathrm{sgn}} = k\C^- e_{\mathrm{sgn}}$ is a projective $k\C^{-}$-module concentrated on the object $[n]$. Consequently, one has
\[
\Delta_{n, \mathrm{sgn}} \cong k\C \otimes_{k\C^-} kG_n e_{\mathrm{sgn}} = k\C \otimes_{k\C^-} k\C^- e_{\mathrm{sgn}},
\]
which is projective. The second statement is obvious since the top of $P^c_{m, \lambda}$ is $L_{m, \lambda} \cong \Delta_{m, \lambda}$, but composition factors of $\Delta_{n, \mathrm{sgn}}$ are $L_{n, \mathrm{sgn}}$ and $L_{n+1, \mathrm{sgn}}$ by the proof of Proposition \ref{length 2}.
\end{proof}

In summary, we obtain a refined version of Theorem \ref{generalized D-K correspondence} for $\FA$.

\begin{theorem} \label{DK for FA}
Let $k$ be a field of characteristic $0$. Then $\C \Mod$ is equivalent to the representation category of a $k$-linear category $\mathscr{K}$ describe as follows:
\begin{enumerate}
\item objects are pairs $(n, \lambda)$ with $n \geqslant 1$ and $\lambda \in \Irr(G_n)$;

\item the endomorphism algebra of each object is one-dimensional;

\item the full subcategory $\mathscr{K}^{\mathrm{reg}}$ consisting of regular objects is an inverse category;

\item the full subcategory $\mathscr{K}^{\mathrm{sing}}$ consisting of singular objects is the $k$-linear category appearing in the classical Dold--Kan correspondence;

\item $\mathscr{K}$ is triangular: there is no nonzero morphism from $\mathscr{K}^{\mathrm{sing}}$ to $\mathscr{K}^{\mathrm{reg}}$.
\end{enumerate}
\end{theorem}

\begin{proof}
The equivalence is given by the Morita equivalence with respect to the family of projective generators
\[
\{ P^c_{n, \lambda} \mid n \in \mathbb{Z}_+,\ \lambda \in \Irr(G_n) \},
\]
where $P^c_{n, \lambda}$ is the projective cover of $\Delta_{n, \lambda}$. Let $\mathscr K$ be the corresponding endomorphism category. It suffices to verify properties \emph{(2)-(5)} in the statement of the theorem.

\medskip
\noindent\textbf{(2) Endomorphism algebras.} If $\lambda \in \Irr(G_n)$ is regular, then by Proposition \ref{filtration} the projective module $P_{n,\lambda} = k\C(n, -) \otimes_{kG_n} \lambda$ admits a filtration whose factors consist of exactly one copy of $\Delta_{n,\lambda}$ and several $\Delta_{m,\mu}$ with $m<n$, so does $P^c_{n, \lambda}$ since it is isomorphic to a direct summand of $P_{n, \lambda}$. Note that $\Delta_{n, \lambda}$ is simple by Proposition \ref{regular standard is simple}, and every composition factor of $\Delta_{m, \mu}$ is either itself (when $\mu$ is regular) or singular (when $\mu$ is the sign representation). It follows that $\Delta_{m, \mu}$ cannot have a composition factor isomorphic to $\Delta_{n, \lambda}$. Consequently, $\Delta_{n, \lambda}$ occurs only once in the composition series of $P^c_{n, \lambda}$, and hence
\[
\End_{k\C}(P^c_{n,\lambda}) \cong k.
\]

If $\lambda = \mathrm{sgn}_n$ is singular, the conclusion follows from Propositions~\ref{length 2} and \ref{disjoint} since $P^c_{n, \mathrm{sgn}} \cong \Delta_{n, \mathrm{sgn}}$, whose composition factors are one copy of $L_{n, \mathrm{sgn}}$ and one copy of $L_{n+1, \mathrm{sgn}}$. Thus in all cases the endomorphism algebra of each object of $\mathscr K$ is one-dimensional.

\medskip
\noindent\textbf{(3) Regular part.}
Let $\lambda \in \Irr(G_n)$ and $\mu \in \Irr(G_m)$ be regular with $m>n$. We claim that
\[
\Hom_{k\C}(P^c_{m,\mu}, \, P^c_{n,\lambda}) = 0,
\]
namely $P^c_{n,\lambda}$ has no composition factor isomorphic to $L_{m, \mu} \cong \Delta_{m, \mu}$. But this follows from the same argument as in the proof of (2). Together with the one-dimensionality of endomorphism algebras, this shows that $\mathscr K^{\mathrm{reg}}$ is an inverse category.

\medskip
\noindent\textbf{(4) Singular part.}
If $\lambda = \mathrm{sgn} \in \Irr(G_n)$, then the projective module $P^c_{n, \mathrm{sgn}} \cong \Delta_{n, \mathrm{sgn}}$ has exactly two composition factors, namely $L_{n,\mathrm{sgn}}$ and $L_{n+1,\mathrm{sgn}}$. This structure reproduces the quiver and Loewy structure underlying the $k$-linear category appearing in the classical Dold--Kan correspondence.

\medskip
\noindent\textbf{(5) Triangular decomposition.} This follows from the second statement of Proposition \ref{disjoint}.
\end{proof}

\begin{remark}
The contrast between the direct sum decompositions for $\CA$ and $\SA$ described in Theorem \ref{DK for CA and SA} and the triangular decomposition for $\FA$ in Theorem~\ref{DK for FA} reflects a genuine difference in the homological behavior of these categories.

For $\CA$ and $\SA$, every regular indecomposable standard module is both simple and projective. As a consequence, there are no nonzero homomorphisms between regular and singular objects in either direction, which forces a direct sum decomposition
\[
\mathscr{K} = \mathscr{K}^{\mathrm{reg}} \oplus \mathscr{K}^{\mathrm{sing}}
\]
as well as the semisimplicity of $\mathscr{K}^{\mathrm{reg}}$.

In contrast, for $\FA$, the large number of ways in which morphisms in $\FA$ factor and interact with automorphism groups gives rise to a genuinely non-semisimple homological behavior. Although regular indecomposable standard modules remain irreducible, most of them fail to be projective because of the existence of abundant nontrivial extensions. Actually, by a direct computation
\[
\dim_k \Ext^1_{k\C}(\Delta_n, \, \Delta_{n-1}) = \frac{(n-1)n!}{2} + 1 - n!,
\]
which grows rapidly with $n$. This phenomenon prevents a direct sum decomposition and the semisimplicity of $\mathscr{K}^{\mathrm{reg}}$.
\end{remark}

In contrast to $\CA$ and $\SA$, the category of finitely generated $\C$-modules does not have enough injective objects. We end this section by giving a brief explanation of this fact.

Fix $n \geqslant 1$ and view $kG_n$ as a $\C^+$-module concentrated on the object $[n]$. It has an injective hull $I_n$ in $\C^+ \Mod$, which is the dual space of the finite dimensional space spanned by all injections ending at $[n]$. Since $\C^+$ is a wide subcategory of $\C$, we have a restriction functor from $\C \Mod$ to $\C^+ \Mod$ as well as a right adjoint
\[
Q: \C^+ \Mod \longrightarrow \C \Mod,
\]
called the \emph{coinduction functor}. It is well known that $Q$ preserves injectives, so we obtain an injective $\C$-module $QI_n$.

By Corollary \ref{socle}, the socle of $QI_n$ is nonzero. Actually, the socle is a finitely generated $\C$-module. To see this, we consider the Hom-space
\[
\Hom_{k\C} (\Delta_m, \, QI_n) \cong \Hom_{k\C^+} (k\C^+(m, -), \, I_n) \cong I_n(m) \cong k\C^+(m, n)
\]
since the restriction of $\Delta_m$ is isomorphic to $k\C^+(m, -)$ and the value of $I_n$ on the object $[m]$ has a natural basis consisting of all injections from $[m]$ to $[n]$. Since $k\C^+(m, n) = 0$ for $m > n$, it follows that the socle of $QI_n$ is a finitely generated $\C$-module.

We claim that $QI_n$ is the injective hull of its socle. Otherwise, we can write $QI_n$ as the direct sum of the injective hull of the socle of $QI_n$ and an injective complement $K$. But then the socle of $K$ is zero, contradicting Corollary \ref{socle}.

Now we show that $QI_n$ is not finitely generated. Note that for each $r \geqslant 0$, the function
\[
\dim: \mathbb{Z}_+ \longrightarrow \mathbb{N}, \quad t \mapsto \dim_k \Delta_r(t)
\]
is of polynomial growth, so is $P_r = k\C(r, -)$. Consequently, each finitely generated $\C$-module is of polynomial growth as well. Thus we only need to show that $QI_n$ is not of polynomial growth. This is indeed the case since for $t \geqslant n$,
\begin{align*}
\dim_k QI_n(t) & = \dim_k \Hom_{k\C^+} (k\C(t, -), \, I_n)\\
 & = \dim_k \Hom_{k\C^+} (\bigoplus_{i=1}^t k\C^+(i, -)^{S(t, i)}, \, I_n)\\
 & = \sum_{i=1}^t S(t, i) \dim_k I_n(i)\\
 & = \sum_{i=1}^n S(t, i) n!/(n-i)!
\end{align*}
is of exponential growth, where $S(t, i)$ is the Stirling number of the second kind.

\section{Uniformly continuous representations of endomorphism monoids}

In this section, we establish a categorical equivalence between uniformly continuous modules over topological transformation monoids $\mathscr{E}$ and sheaves on the corresponding categories $\C$ of finite subsets and finite restrictions of elements in $\mathscr{E}$, and explore several concrete examples and applications.

\subsection{Continuous and uniformly continuous representations}

Let $\Omega$ be an infinite set, and let $\mathscr{E} \leqslant \Omega^{\Omega}$ be a monoid of self-maps equipped with the \emph{permutation topology}. Specifically, we equip $\Omega$ with the discrete topology, $\Omega^{\Omega}$ with the product topology, and $\mathscr{E}$ with the subspace topology. Then any element $f \in \mathscr{E}$ has a basis of open neighborhoods described below:
\[
\mathscr{N}_f = \{ \mathscr{N}_{f, A} \mid A \subseteq \Omega, \, |A| < \infty  \}, \quad \mathscr{N}_{f, A} = \{g \in \mathscr{E} \mid g|_{A} = f |_{A}\}
\]
With respect to this topology the composition in $\mathscr{E}$ is continuous, so it becomes a topological monoid.

Let $k$ be a commutative ring. We say that a $k\mathscr{E}$-module $V$ is \emph{continuous} if the action of $\mathscr{E}$ on $V$ is continuous with respect to the permutation topology on $\mathscr{E}$ and the discrete topology on $V$. Denote by $k\mathscr{E} \Mod^{\con}$ the category of continuous $k\mathscr{E}$-modules. It is an abelian category since submodules and quotient modules of continuous $k\mathscr{E}$-modules are still continuous.

The following result gives a characterization of continuous $k\mathscr{E}$-modules.

\begin{lemma} \label{finite support}
A $k\mathscr E$-module $V$ is continuous if and only if for every $v\in V$ and $h \in \mathscr{E}$, there exists a finite subset $A_{h, v} \subseteq \Omega$ such that for all $h' \in \mathscr E$,
\[
h|_{A_{h, v}} = h'|_{A_{h, v}} \;\Longrightarrow\; h \cdot v = h'\cdot v.
\]
\end{lemma}

\begin{proof}
\textbf{The only if direction.} Fix $v \in V$ and consider the action map
\[
\alpha_v : \mathscr E \to V, \qquad g \mapsto g \cdot v.
\]
Since $V$ is equipped with the discrete topology and the action map is continuous, the set
\[
U_h = \{ h' \in \mathscr E \mid h' \cdot v = h \cdot v \}
\]
is an open neighborhood of $h$. By the above description of basic open neighborhoods, we can find a finite subset $A_{h, v} \subseteq \Omega$ such that $\mathscr{N}_{h, A_{h, v}} \subseteq U_h$. That is, if $h'|_{A_{h, v}} = h|_{A_{h, v}}$, then $h' \cdot v = h \cdot v$.

\medskip

\textbf{The if direction.}
Assume that for every $v \in V$ and $h \in \mathscr{E}$, there exists a finite subset $A_{h,v} \subseteq \Omega$ such that
\[
h|_{A_{h,v}} = h'|_{A_{h,v}} \;\Longrightarrow\; h \cdot v = h' \cdot v.
\]
We want to show that the global action map
\[
\alpha : \mathscr{E} \times V \to V, \qquad (g, v) \mapsto g \cdot v,
\]
is continuous. Since $V$ is given the discrete topology, the map $\alpha: \mathscr{E} \times V \to V$ is continuous if and only if for every $v \in V$, the action map
\[
\alpha_v: \mathscr{E} \to V, \quad g \mapsto g \cdot v
\]
is continuous, namely the preimage of every singleton $\{w\} \subseteq V$ is open in $\mathscr{E}$. Fix $v \in V$ and let $f \in \mathscr{E}$ such that $f \cdot v = w$. We need to find an open neighborhood $U$ of $f$ in $\mathscr{E}$ such that $\alpha_v(U) \subseteq \{w\}$. By hypothesis, there exists a finite subset $A_{f, v} \subseteq \Omega$ such that:
\[
g|_{A_{f, v}} = f|_{A_{f, v}} \;\Longrightarrow\; g \cdot v = f \cdot v.
\]
Taking $U = \mathscr{N}_{f, A_{f, v}}$, we see that $U$ is an open neighborhood of $f$ and $g \cdot v = w$ for all $g \in U$. Thus $\alpha_v^{-1}(w)$ is open, completing the proof.
\end{proof}

\begin{remark}
This characterization can be viewed as a topological refinement of the classical notion of continuous representations of totally disconnected groups, adapted to the setting of topological monoids. It might be well known to experts, but has not, to the author’s knowledge, been stated in precisely this form in the literature. Related ideas appear in several strands of the literature, including \cite{MMP, PS, Steinberg2023}.
\end{remark}

Note that in the above lemma, the finite subset $A_{h, v}$ depends on both $v$ and $h$. In general, we cannot find a common $A_v$ working for all elements in $\mathscr{E}$. This stronger requirement is precisely the uniform continuity in topology, so we make the following definition.

\begin{definition}
A $k\mathscr{E}$-module $V$ is \emph{uniformly continuous} if for every $v \in V$, there exists a finite subset $A_v \subseteq \Omega$ such that
\[
h|_{A_v} = h'|_{A_v} \;\Longrightarrow\; h \cdot v = h' \cdot v, \, \forall \, h, h' \in \mathscr{E}.
\]
Denote the category of uniformly continuous $k\mathscr{E}$-modules by $k\mathscr{E} \Mod^{\mathrm{uc}}$, which is also abelian.
\end{definition}

This terminology is justified by the natural uniform structure on $\mathscr{E}$ induced by pointwise convergence: basic entourages are $\mathcal{W}_A = \{ (h,h') \mid h|_A = h'|_A \}$ for finite $A \subseteq \Omega$. Then a module is uniformly continuous precisely when all action maps $h \mapsto h \cdot v$ are uniformly continuous with respect to this uniformity. Clearly, every uniformly continuous module is continuous, but the converse statement is not true. However, if $\mathscr{E}$ is a permutation group equipped with the pointwise convergence topology, then continuity is equivalent to uniform continuity. For monoids, the lack of inverses breaks translation-invariance, making the distinction significant.

\begin{remark} \label{remind}
We remind the reader that the regular representation $k\mathscr{E}$ is not a continuous $k\mathscr{E}$-module, and hence not uniformly continuous. The action map
\[
\rho: \mathscr{E} \times \mathscr{E} \longrightarrow \mathscr{E}
\]
is continuous with respect to the permutation topology on all three components, but is not continuous when we equip the acting monoid $\mathscr{E}$ with the permutation topology and the other two $\mathscr{E}$ with the discrete topology. Indeed, for the identity $\id \in \mathscr{E}$ and a given $f \in \mathscr{E}$, we cannot find a finite subset $A_f \subseteq \Omega$ such that when $g \in \mathscr{E}$ has the same restriction on $A$ as $f$, then $f \circ \id = g \circ \id$. Thus by Lemma \ref{finite support}, the action is not continuous.

The requirement of uniform continuity leads us to consider the submonoid of $\mathscr{E}$ consisting of maps whose image is finite. Explicitly, define
\[
\mathscr{E}_0 = \{ f \in \mathscr{E} \mid |f(\Omega)| < \infty \}.
\]
Then $k\mathscr{E}_0$ is a uniformly continuous $k\mathscr{E}$-module. To see this, take $v \in \mathscr{E}_0$ and write it as
\[
v = a_1 f_1 + \ldots + a_nf_n, \, a_i \in k, \, f_i \in \mathscr{E}_0.
\]
Set $A_v$ be the union of $f_i (\Omega)$, which is a finite set. One checks that if $g, h \in \mathscr{E}$ have the same restriction to $A_v$, then $g \cdot v = h \cdot v$ where $\cdot$ is induced by the composition of maps. Thus the action of $\mathscr{E}$ on $v$ is uniformly continuous, so $k\mathscr{E}_0$ is a uniformly continuous $k\mathscr{E}$-module.
\end{remark}

In the category $k\mathscr{E} \Mod^{\uc}$, the interaction between the algebraic structure and the topology is exceptionally rigid. Every homomorphism $\phi: V \to W$ inside it is continuous. Furthermore, $\phi$ respects the uniform structure of the action: if $v \in V$ is supported on a finite subset $A \subseteq \Omega$ (meaning $h|_A = h'|_A \implies h \cdot v = h' \cdot v$), then $\phi(v)$ is also supported on $A$ because
\[
h|_A = h'|_A \;\implies\; h \cdot \phi(v) = \phi(h \cdot v) = \phi(h' \cdot v) = h' \cdot \phi(v).
\]

\subsection{Realization and its right adjoint}

Given a transformation monoid $\mathscr{E} \leqslant \Omega^{\Omega}$, we construct a category $\C$ as follows. Objects in $\C$ are finite subsets of $\Omega$, and morphisms are restrictions of elements of $\mathscr{E}$. Call $\C$ the category of finite subset associated to $\mathscr{E}$. Explicitly, $f: A \to B$ is a morphism in $\C$ if there is an element $\tilde{f} \in \mathscr{E}$ such that $\tilde{f}|_A = f$. For examples:
\begin{itemize}
\item if $\mathscr{E} = \Omega^{\Omega}$, then $\C$ is the category of finite subsets of $\Omega$ and maps;

\item if $\Omega$ is equipped with a linear order $\leqslant$ and $\mathscr{E} = \End(\Omega, \leqslant)$, then $\C$ is the category of finite subsets of $\Omega$ and order-preserving maps.
\end{itemize}
Note that $\C$ contains every inclusion morphism $\iota_{A, B}: A \hookrightarrow B$ if $A$ is a finite subset of $B$.

We seek to relate representations of $\C$ to uniformly continuous representations of $\mathscr{E}$.

\begin{lemma}
Let $\C$ be the category of finite subset associated to $\mathscr{E}$, and $V$ a $\C$-module. Then
\[
\mathbb V \;=\; \varinjlim_{A \in \Ob(\C)} V(A),
\]
admits a natural structure of a uniformly continuous $k\mathscr E$-module, where the colimit is taken over the directed system given by inclusion morphisms.
\end{lemma}

\begin{proof}
By definition of the colimit, we have
\[
\mathbb V
=
\Big( \bigsqcup_{A \in \Ob(\C)} V(A) \Big) \Big/ \sim,
\]
where $A$ ranges over all finite subsets of $\Omega$, and the equivalence relation is generated by
\[
v \sim \iota_{A,B} \cdot v \quad \text{for all inclusions } \iota_{A,B}: A \hookrightarrow B.
\]
Denote by $[v]$ the equivalence class of $v \in V(A)$. We define the action of $\mathscr{E}$ on $\mathbb{V}$ as follows: for $f \in \mathscr E$ and $v \in V(A)$, set
\[
f \cdot [v] = [ f|_A \cdot v ].
\]

\medskip
We check that this is a well-defined monoid action. Suppose $[v] = [v']$ in $\mathbb V$, with $v \in V(A)$ and $v' \in V(B)$. Then there exists a finite subset $C$ such that $A, B \subseteq C$ and
\[
\iota_{A,C} \cdot v = \iota_{B,C} \cdot v' \in V(C).
\]
Using functoriality of $V$, we have
\[
f \cdot [v] = [\, f|_A \cdot v \,] = [\, (\iota_{f(A),f(C)}) \cdot (f|_A \cdot v) \,] = [\, (f|_C) \cdot (\iota_{A,C} \cdot v) \,],
\]
and similarly
\[
f \cdot [v'] = \big[\, (f|_C)  \cdot (\iota_{B,C} \cdot v') \,\big].
\]
Since $\iota_{A,C} \cdot v = \iota_{B,C} \cdot v'$, these two classes are equal, so the action is well-defined. Furthermore, for $f,g \in \mathscr E$ and $v \in V(A)$, one has
\[
(fg) \cdot [v]
= \big[\, (fg)|_A \cdot v \,\big]
= \big[\, (f|_{g(A)}) \cdot (g|_A \cdot v) \,\big]
= f \cdot \big(g \cdot [v]\big),
\]
by functoriality of $V$. Thus it is indeed a monoid action.

\medskip
We complete the proof by showing the uniform continuity of this action. Equip $\mathbb V$ with the discrete topology. Let $[v] \in \mathbb V$ be represented by $v \in V(A)$ for some finite subset $A \subseteq \Omega$.
For any $f, f' \in \mathscr{E}$, if $f|_A = f'|_A$, then by definition of the action,
\[
f \cdot [v] = [\, f|_A \cdot v \,] = [\, f'|_A \cdot v \,] = f' \cdot [v].
\]
Hence the single finite subset $A_v = A$ works for all $f, f' \in \mathscr{E}$, showing that the action map
\[
\alpha_{[v]} : \mathscr{E} \longrightarrow \mathbb V, \quad f \mapsto f \cdot [v],
\]
is uniformly continuous with respect to the uniformity of pointwise convergence on $\mathscr{E}$. Since $[v]$ was arbitrary, the action of $\mathscr{E}$ on $\mathbb V$ is uniformly continuous.
\end{proof}

The above construction by taking colimits gives a functor, called \emph{the realization functor}
\[
\mathcal{R} : \C \Mod \longrightarrow k\mathscr E\Mod^{\uc}, \qquad V \longmapsto \varinjlim_{A \in \Ob(\C)} V(A).
\]
It is exact since the indexing category (finite subsets with inclusions) is filtered, and filtered colimits of $k$-modules are exact. Furthermore, it has a right adjoint satisfying some special properties.

\begin{proposition} \label{adjunction}
The realization functor
\[
\mathcal{R} : \C\Mod \longrightarrow k\mathscr{E}\Mod^{\uc}
\]
has a right adjoint
\[
\mathcal{L}: k\mathscr{E}\Mod^{\uc} \longrightarrow \C\Mod,
\]
such that $\mathcal{R} \circ \mathcal{L}$ is isomorphic to the identity functor on $k\mathscr{E}\Mod^{\uc}$.
\end{proposition}

\begin{proof}
We construct $\mathcal{L}$ explicitly and verify the adjunction as well as the natural isomorphism.

\medskip
\textbf{Definition of $\mathcal{L}$ on objects.} Given a uniformly continuous $k\mathscr{E}$-module $M$, we construct a $\C$-module $V = \mathcal{L}(M)$ as follows. For each finite subset $A \subseteq \Omega$, define
\[
V(A) = \{ v \in M \; | \; \tilde{f} \cdot v = \tilde{g} \cdot v \text{ for all } \tilde{f}, \tilde{g} \in \mathscr E \text{ with } \tilde{f}|_A = \tilde{g}|_A \}.
\]
For a morphism $f: A \to B$ in $\C$, choose any extension $\tilde{f} \in \mathscr E$ (which exists since every morphism in $\C$ is the restriction of an element in $\mathscr{E}$) and set
\[
f \cdot v = \tilde{f} \cdot v.
\]
This construction is well-defined. Firstly, the action is independent of the choice of extensions by definition of $V(A)$. Secondly, for $\tilde{g}, \tilde{h} \in \mathscr E$ satisfying $\tilde{g}|_B = \tilde{h}|_B$, one has
\[
(\tilde{g} \tilde{f}) |_A = (\tilde{h} \tilde{f})|_A,
\]
which implies
\[
\tilde{g} \cdot (f \cdot v) = \tilde{g} \cdot(\tilde{f} \cdot v) = (\tilde{g} \tilde{f}) \cdot v = (\tilde{h} \tilde{f}) \cdot v  = \tilde{h} \cdot (\tilde{f} \cdot v) = \tilde{h} \cdot (f \cdot v).
\]
Therefore, $f \cdot v \in V(B)$ by the definition of $V(B)$. To check the functoriality of $V$, let $h: B \to C$ be another morphism in $\C$ and choose extensions $\tilde{h} \in \mathscr E$. Then $\tilde{h} \tilde{f}$ extends $hf$, so we have
\[
h \cdot (f \cdot v) = \tilde{h} \cdot (\tilde{f} \cdot v) = (\tilde{h} \tilde{f}) \cdot v = (hf) \cdot v.
\]

\medskip
\textbf{Definition of $\mathcal{L}$ on morphisms.} Given a homomorphism $\phi: M \to N$ in $k\mathscr{E} \Mod^{\uc}$, define
\[
\mathcal{L}(\phi)_A : V(A) \longrightarrow W(A), \quad v \mapsto \phi(v),
\]
where $W = \mathcal{L}(N)$. This is well-defined:
\[
\tilde{f}|_A = \tilde{g}|_A \implies \tilde{f} \cdot \phi(v) = \phi(\tilde{f} \cdot v) = \phi(\tilde{g} \cdot v) = \tilde{g} \cdot \phi(v),
\]
so $\phi(v) \in W(A)$. Moreover, for any morphism $f: A \to B$ in $\C$ and $v \in V(A)$,
\[
f \cdot \phi(v) = \tilde{f} \cdot \phi(v) = \phi(\tilde{f} \cdot v) = \phi(f \cdot v),
\]
showing that $\mathcal{L}(\phi)$ is a homomorphism of $\C$-modules. Functoriality is immediate.

\medskip
\textbf{$\mathcal{R} \circ \mathcal{L} \cong \id$.} By uniform continuity, each $v \in M$ has a finite support $A_v \subseteq \Omega$ such that $\tilde{f} \cdot v = \tilde{g} \cdot v$ whenever $\tilde{f}|_{A_v} = \tilde{g}|_{A_v}$. Thus $v \in V(A)$ for any finite $A \supseteq A_v$, and the colimit
\[
\mathcal{R} \circ \mathcal{L}(M) = \varinjlim_{A \subseteq \Omega \text{ finite}} V(A)
\]
maps canonically to $M$ via $\phi: [v] \longmapsto v$. This map is surjective since every $v$ lies in some $V(A)$. To see that $\phi$ is well-defined, suppose $[v] = [v']$ with $v \in V(A)$ and $v' \in V(B)$. Then there exists a finite subset $C \supseteq A,B$ such that
\[
\iota_{A,C} \cdot v = \iota_{B,C} \cdot v'.
\]
Since $\iota_{A,C}$ and $\iota_{B,C}$ are restrictions of the identity map $\id_{\Omega}$, by the definition of $V(A)$ and $V(B)$, we have
\[
\iota_{A,C} \cdot v = v, \quad \iota_{B, C} \cdot v' = v'.
\]
Thus $v = v'$ in $M$, so $\phi$ is well-defined. Clearly, $\phi$ is injective. Moreover, $\phi$ is $\mathscr E$-equivariant:
\[
\phi(\tilde{f} \cdot [v]) = \phi([\tilde{f}|_A \cdot v]) = \tilde{f}|_A \cdot v = \tilde{f} \cdot v = \tilde{f} \cdot \phi([v]),
\]
so $\phi$ is an isomorphism of $k\mathscr E$-modules.

\medskip
\textbf{ Adjunction.}
For $V \in \C\Mod$ and $M \in k\mathscr{E}\Mod^{\uc}$, define
\[
\Hom_{k\mathscr{E}}(\mathcal{R}(V), \, M) \longleftrightarrow \Hom_{\C\Mod}(V, \, \mathcal{L}(M))
\]
as follows. Given $\phi: \mathcal{R}(V) \to M$, set
\[
\psi_A: V(A) \longrightarrow \mathcal{L}(M)(A), \quad v \mapsto \phi([v]).
\]
For $\tilde{f}|_A = \tilde{g}|_A$, one has
\[
\tilde{f} \cdot \psi_A(v) = \tilde{f} \cdot \phi([v]) = \phi(\tilde{f} \cdot [v]) = \phi([\tilde{f}|_A \cdot v]) = \phi([\tilde{g}|_A \cdot v]) = \tilde{g} \cdot \phi([v]) = \tilde{g} \cdot \psi_A(v),
\]
so $\psi_A(v) \in \mathcal{L}(M)(A)$. Furthermore, for $f: A \to B$ in $\C$ and $v \in V(A)$, one has
\[
\psi_B(f \cdot v) = \phi([f \cdot v]) = \phi( \tilde{f} \cdot [v]) = \tilde{f} \cdot \phi([v]) = f \cdot \psi_A(v),
\]
so $\psi$ is a natural transformation.

Conversely, given $\psi: V \to \mathcal{L}(M)$, set
\[
\phi: \mathcal{R}(V) \longrightarrow M, \quad [v] \mapsto \psi_A(v), \quad v \in V(A),
\]
which is independent of the choice of representative. Specifically, if $v \in V(A)$ and $v' \in V(B)$ represent the same class $[v]$, there exists a finite $C \supseteq A, B$ such that $\iota_{A,C}(v) = \iota_{B,C}(v')$. By the naturality of $\psi$, we have
\[
\iota_{A,C} \cdot \psi_A(v) = \psi_C(\iota_{A,C} \cdot v) = \psi_C(\iota_{B,C} \cdot v') = \iota_{B,C} \cdot \psi_{B}(v').
\]
Thus $\psi_A(v) = \psi_B(v')$ since $\iota_{A,C}$ and $\iota_{B,C}$ are restrictions of the identity $\id_{\Omega}$. In addition, $\phi$ is $\mathscr{E}$-equivariant since for $\tilde{f} \in \mathscr{E}$ and $[v] \in \mathcal{R}(V)$, one has
\[
\phi(\tilde{f} \cdot [v]) = \phi([\tilde{f}|_A \cdot v]) = \psi_{\tilde{f}(A)}(\tilde{f}|_A \cdot v) = \tilde{f} \cdot \psi_A(v) = \tilde{f} \cdot \phi([v]).
\]

It is routine to check that these maps are inverse to each other. This finishes the proof.
\end{proof}

\subsection{Equivalences}

The above proposition allows us to construct an explicit equivalence. For this purpose, we introduce the following definition.

\begin{definition} \label{def torsion}
Given a $\C$-module $V$ and $A \in \Ob(\C)$, an element $v \in V(A)$ is \emph{torsion} if there is an inclusion morphism $\iota_{A, B}: A \hookrightarrow B$ such that $\iota_{A, B} \cdot v = 0$. We call $V$ a \emph{torsion module} if for all $A \in \Ob(\C)$, each $v \in V(A)$ is torsion.
\end{definition}

The following lemma is well known in torsion theory.

\begin{lemma} \label{torsion}
Let $V$ be a $\C$-module. Then torsion elements in $V$ form a submodule.
\end{lemma}

\begin{proof}
Suppose that $v, w \in V(A)$ are torsion, then clearly so are their scalar multiples. It remains to show that $v + w$ and $f \cdot v$ are torsion for any $f: A \to C$ in $\C$.

Choose inclusions $\iota_{A, A'}: A \hookrightarrow A'$ and $\iota_{A, B'}: A \hookrightarrow B'$ such that
\[
\iota_{A, A'} \cdot v = 0 = \iota_{A, B'} \cdot w.
 \]
Set $D = A' \cup B'$. Then
\[
\iota_{A, D} \cdot v = 0 = \iota_{A, D} \cdot w,
\]
so $v + w$ is torsion as well.

Let $\tilde{f} \in \mathscr{E}$ be an extension of $f$, $D' = \tilde{f}(A') \cup C$, and $g: A' \to D'$ be the restriction $\tilde{f}|_{A'}$. We have $g \circ \iota_{A, A'} = \iota_{C, D'} \circ f$ since they both are the restriction of $\tilde{f}$ to $A$ . Therefore,
\[
\iota_{C, D'} \cdot (f \cdot v) = g \cdot (\iota_{A, A'} \cdot v) = g \cdot 0 = 0,
\]
so $f \cdot v$ is torsion in $V(C)$.
\end{proof}

By this lemma, torsion elements in $V$ form a submodule of $V$. It is routine to check that the full subcategory $\C \Mod^{\tor}$ of torsion modules is a localizing subcategory of $\C \Mod$.

\begin{proposition} \label{equivalence 1}
One has an equivalence of categories
\[
k\mathscr{E} \Mod^{\mathrm{uc}} \simeq \C \Mod / \C \Mod^{\tor}.
\]
\end{proposition}

\begin{proof}
By the classical result of Gabriel on localization of abelian categories \cite[Proposition III.2.5]{Gabriel} and Proposition \ref{adjunction}, we have the following equivalence
\[
k\mathscr{E} \Mod^{\mathrm{uc}} \simeq \C \Mod / \ker \mathcal{R},
\]
where $\ker \mathcal{R}$ is the full subcategory consisting of $\C$-modules $V$ such that $\mathcal{R} (V) = 0$. But $\mathcal{R}(V) = 0$ if and only if for every $A \in \Ob(\C)$ and every $v \in V(A)$, one has $[v] = 0$. By the definition of the colimit equivalence relation, this occurs if and only if for every $A \in \Ob(\C)$ and every $v \in V(A)$, there is an inclusion $\iota_{A, B}: A \hookrightarrow B$ such that $\iota_{A, B} \cdot v = 0$, namely $V$ is a torsion module.
\end{proof}

To get a sheaf-theoretic interpretation of the above equivalence, for an inclusion $\iota_{A, B}: A \hookrightarrow B$ in $\C$, we define a left ideal
\[
S_{A, B} = \{ f \circ \iota_{A, B} \mid f \in \C(B, -) \},
\]
which is a principal sieve on the opposite category $\C^{\op}$. Define a rule $J$ assigning to each $A \in \Ob(\C)$ a set of sieves
\[
J(A) = \{ S \subseteq \C(A, -) \mid S \text{ is a sieve containing some } S_{A, B} \},
\]
in other words, $J(A)$ is the set of all sieves each of which contains an inclusion $\iota_{A, B}: A \hookrightarrow B$ for a certain finite subset $B$.

\begin{lemma} \label{Grothendieck topology}
The rule $J$ is a Grothendieck topology on $\C^{\op}$.
\end{lemma}

\begin{proof}
We verify the three standard axioms for Grothendieck topologies. For any $A \in \Ob(\C)$, the maximal sieve $S_{A,A}$ contains $\id_A$, and hence is a covering sieve. Thus the maximality axiom holds.

To check the stability axiom, take $S \in J(A)$ and an arbitrary morphism $f: A \to C$ in $\C$. We need to show that the pullback sieve
\[
f^*(S) = \{ h \in \C(C, -) \mid h \circ f \in S \}
 \]
is in $J(C)$, namely $f^*(S)$ contains an inclusion with source $C$. By the definition of $J(A)$, one can find an inclusion $\iota_{A, B} \in S$. Let $D$ be the finite subset $B \cup C$. Note that the morphism $f: A \to C$ extends to a map $\tilde{f} \in \mathscr{E}$, whose restriction gives a morphism $B \to D$. Using the same argument as in the proof of Lemma \ref{torsion}, one can show $\iota_{C, D} \in f^*(S)$.

Now we turn to the transitivity axiom. Suppose $S \in J(A)$ and $R$ is a sieve on $A$ such that for every $f: A \to \bullet$ in $S$, one has $f^*(R) \in J(\bullet)$. Since $S \in J(A)$, there is an inclusion $\iota_{A, B} \in S$. By hypothesis, $\iota_{A, B}^*(R) \in J(B)$, implying the existence of an inclusion $\iota_{B, D} \in \iota_{A, B}^*(R)$. It follows that
\[
\iota_{A, D}  = \iota_{B, D} \circ \iota_{A, B} \in R.
\]
Thus $R \in J(A)$ as desired.
\end{proof}

Let $(\C^{\op}, J)$ be the Grothendieck site, $\underline{k}$ the constant structure sheaf, and $\Sh(\C^{\op}, \, J, \, \underline{k})$ the category of sheaves of $k$-modules. We have:

\begin{theorem} \label{equivalence 2}
One has the following equivalences of categories
\[
k\mathscr{E} \Mod^{\mathrm{uc}} \simeq \C \Mod / \C \Mod^{\tor} \simeq \Sh(\C^{\op}, \, J, \, \underline{k}).
\]
\end{theorem}

\begin{proof}
The first equivalence is proved in Proposition \ref{equivalence 1}, so we consider the second one. By \cite[Corollary 1.2]{DLL}, it suffices to show that $J$-torsion elements defined in that paper are the same as torsion elements in Definition \ref{def torsion}. To see this, given a $\C$-module $V$ and an object $A$ in $\C$, by \cite[Definition 3.1]{DLL}, an element $v \in V(A)$ is $J$-torsion if and only if there is a covering sieve $S \in J(A)$ such that $f \cdot v = 0$ for every $f \in S$. But $S$ contains an inclusion $\iota_{A, B}$, so $v$ is torsion in the sense of Definition \ref{def torsion}. Conversely, if $v$ is torsion, then one can find an inclusion $\iota_{A, B}$ such that $\iota_{A, B} \cdot v = 0$. Consequently, every morphism in the covering sieve $S_{A, B} \in J(A)$ sends $v$ to 0, so $v$ is $J$-torsion.
\end{proof}

\subsection{Application I} \label{Artin's theorem}

We consider several applications of this theorem.

In the first application, we take $\mathscr{E} = G \leqslant \Sym(\Omega)$ to be a permutation group on $\Omega$. The associated category $\C$ has finite subsets of $\Omega$ as objects, and morphisms given by restrictions of elements of $G$. Equivalently, $\C$ can be identified with the category of finite substructures and embeddings of $(\Omega, \mathcal{R}_G)$, where $\mathcal{R}_G$ is the orbit relational structure induced by the action of $G$ on $\Omega$; see \cite{Ca90} for details. Every morphism in $\C$ is injective, as it arises from a restriction of a permutation in $G$. Conversely, since $(\Omega, \mathcal{R}_G)$ is homogeneous, every embedding of finite substructures extends to an element of $G$.

Moreover, the Grothendieck topology $J$ on $\C^{\op}$ is the atomic topology, i.e. every nonempty sieve is covering. To see this, take $A \in \Ob(\C)$ and let $S \subseteq \C(A,-)$ be a nonempty sieve. Choose an embedding $f: A \hookrightarrow B$ in $S$. By homogeneity, there exists an extension $\tilde{f} \in G$ of $f$. Let $g: B \to \tilde{f}^{-1}(B)$ be the restriction of $\tilde{f}^{-1}$ to $B$. Then $g \circ f \in S$ is the inclusion of $A$ into $\tilde{f}^{-1}(B)$, showing that every nonempty sieve contains an inclusion morphism and hence is covering.

We now explain how Artin’s reconstruction theorem for Hausdorff topological groups is recovered as a special case of the present equivalence theorem. Although our framework is formulated for permutation groups equipped with the permutation topology, every Hausdorff topological group admits such a realization. Indeed, let $\Omega$ be the set of the coset spaces $G/H$ over all open subgroups $H \leqslant G$. The natural left action of $G$ on $\Omega$ induces a topological embedding of groups
\[
G \hookrightarrow \Sym(\Omega),
\]
where $\Sym(\Omega)$ is endowed with the permutation topology. Thus, without loss of generality, one may assume that $G$ is a permutation group equipped with the permutation topology.

If the action of $G$ on $\Omega$ satisfies the strong amalgamation property—namely, if the pointwise stabilizer of every finite subset $A$ has no fixed point outside $A$—then $\C^{\op}$ is isomorphic to the orbit category $\mathscr{G}$ of $G$ with respect to the cofinal system of stabilizer subgroups of finite subsets; see \cite[Theorem~1.4]{Li25}. In this case, Theorem~\ref{equivalence 2} yields an equivalence
\[
kG \Mod^{\mathrm{uc}} \simeq \Sh(\mathscr{G}, \, J_{\mathrm{at}}, \, \underline{k}),
\]
which is precisely the $k$-linear form of Artin’s theorem.

If, on the other hand, $\C$ fails to satisfy the strong amalgamation property, then there exist non-isomorphic finite subsets $A$ and $B$ with identical pointwise stabilizer subgroups, $G_A = G_B$. In this situation, the opposite category $\C^{\op}$ is no longer isomorphic to the orbit category $\mathscr{G}$: the former distinguishes $A$ and $B$ as separate objects, while the latter identifies them. Nevertheless, this distinction disappears upon sheafification. Although $\C^{\op}$ and $\mathscr{G}$ may differ as sites, they give rise to equivalent sheaf categories. Conceptually, the atomic topology forces finite subsets with the same stabilizer subgroup to become locally indistinguishable, so that the resulting topos depends only on the stabilizer data rather than on the individual finite subsets. More formally, one may apply the Comparison Lemma to obtain an equivalence
\[
\Sh(\C^{\op}, \, J_{\mathrm{at}}, \, \underline{k}) \simeq \Sh(\mathscr{G}, \, J_{\mathrm{at}}, \, \underline{k}),
\]
and combining this with Theorem~\ref{equivalence 2} again recovers Artin’s result, since uniformly continuous modules coincide with continuous modules for topological groups.

\subsection{Application II} \label{Application II}
As the second application, we have the following result.

\begin{corollary} \label{equivalence 4}
Let $\mathscr{E} \leqslant \Omega^{\Omega}$ and $\mathscr{F}$ a dense submonoid of $\mathscr{E}$. Then
\[
k\mathscr{F} \Mod^{\uc} \simeq k\mathscr{E} \Mod^{\uc}.
\]
\end{corollary}

\begin{proof}
Recall that a neighborhood basis of the permutation topology on $\mathscr{E}$ is given by the sets
\[
\mathscr{N}_{f,A} = \{ g \in \mathscr{E} \mid g|_A = f|_A \},
\]
where $f \in \mathscr{E}$ and $A \subseteq \Omega$ is finite. Density of $\mathscr{F}$ in $\mathscr{E}$ is therefore equivalent to the condition that
\[
\mathscr{F} \cap \mathscr{N}_{f,A} \neq \varnothing
\quad
\text{for all } f \in \mathscr{E},\ A \subseteq \Omega \text{ finite}.
\]

We show that $\mathscr{E}$ and $\mathscr{F}$ have the same associated category $\C$ of finite subsets. For any finite subset $A \subseteq \Omega$, we must verify that
\[
\{ f|_A \mid f \in \mathscr{F} \}
=
\{ g|_A \mid g \in \mathscr{E} \}.
\]
The inclusion is trivial. Conversely, given $f \in \mathscr{E}$, density of $\mathscr{F}$ ensures the existence of $g \in \mathscr{F} \cap \mathscr{N}_{f,A}$, so that $g|_A = f|_A$. Hence the two sets coincide.

The conclusion now follows from Theorem \ref{equivalence 2}.
\end{proof}

\begin{remark}
By the proof, we see that a submonoid $\mathscr{F} \subseteq \mathscr{E}$ induces the same associated category $\C$ of finite subsets as $\mathscr{E}$ if and only if $\mathscr{F}$ is dense in $\mathscr{E}$ with respect to the pointwise convergence topology, equivalently, if every finite restriction of an element of $\mathscr{E}$ is realized by an element of $\mathscr{F}$.
\end{remark}

For topological groups, continuous representations are automatically uniformly continuous. We can use the following result to find more monoids having this property.

\begin{lemma}\label{continuity and uniform continuity}
Let $\mathscr{E} \leqslant \Omega^{\Omega}$ and $\mathscr{F}$ a dense submonoid of $\mathscr{E}$. Then a continuous representation $V$ of $\mathscr{E}$ is uniformly continuous if and only if its restriction to $\mathscr{F}$ is uniformly continuous.
\end{lemma}

\begin{proof}
The ``only if'' direction is immediate: restriction of a uniformly continuous map remains uniformly continuous. For the ``if'' direction, assume that the restriction of $V$ to $\mathscr{F}$ is uniformly continuous. Then for each $v \in V$, there exists a finite set $A_v \subseteq \Omega$ (the \emph{support} of $v$) such that
\[
g|_{A_v} = h|_{A_v} \implies g \cdot v = h \cdot v \quad \text{for all } g, h \in \mathscr{F}.
\]

Let $f, g \in \mathscr{E}$ satisfy $f|_{A_v} = g|_{A_v}$. Since $\mathscr{F}$ is dense in $\mathscr{E}$, there are nets $(f_i)_i$, $(g_i)_i$ in $\mathscr{F}$ such that
\[
f_i \to f, \qquad g_i \to g
\]
and eventually satisfy
\[
f_i|_{A_v} = f|_{A_v} = g|_{A_v} = g_i|_{A_v}.
\]
Uniform continuity then gives $f_i \cdot v = g_i \cdot v$ eventually. Since the action of $\mathscr{E}$ on $V$ is continuous, and $V$ is a discrete space, we conclude
\[
f \cdot v = \lim_i f_i \cdot v = \lim_i g_i \cdot v = g \cdot v.
\]
Thus, for every $v \in V$ there exists a finite set $A_v \subseteq \Omega$ such that
\[
f|_{A_v} = g|_{A_v} \implies f \cdot v = g \cdot v \qquad \forall \, f,g \in \mathscr{E},
\]
so $V$ is a uniformly continuous representation of $\mathscr{E}$.
\end{proof}

We give some counterintuitive examples to illustrate the application of this result.

\begin{example} \label{surprising examples}
Let $\mathscr{E} = \Omega^{\Omega}$ be the full transformation monoid, and let $\mathscr{E}_i$ be the submonoid of injective maps. Note that finite restrictions of elements of $\mathscr{E}_i$ are injective, and every injective map between finite subsets extends to a permutation in $G = \Sym(\Omega)$. It follows that $G$ is dense in $\mathscr{E}_i$ with respect to the permutation topology, so $\mathscr{E}_i$ and $G$ have the same uniformly continuous representation theory. Moreover, since continuous representations of $G$ are automatically uniformly continuous, Lemma~\ref{continuity and uniform continuity} implies that every continuous representation of $\mathscr{E}_i$ is uniformly continuous.

Similarly, let $\mathscr{E} = \End(\mathbb{Q}, \leqslant)$ be the monoid of order-preserving maps. Then the submonoid $\mathscr{E}_i$ of injective order-preserving maps has the same uniformly continuous representation theory as $\Aut(\mathbb{Q},\leqslant)$, which has been studied in \cite{DLLX}. The same conclusion holds for other highly homogeneous relations on $\mathbb{Q}$, namely the dense cyclic order, dense linear betweenness, and dense separation relations.

On the other hand, the submonoid $\mathscr{E}_s$ of surjective maps is also dense in $\mathscr{E} = \Omega^{\Omega}$. Indeed, given any map $f : A \to B$ between finite subsets, there exists a surjective map $\tilde{f} \in \mathscr{E}_s$ such that $\tilde{f}|_A = f$. Analogous density statements hold for the corresponding monoids of maps preserving a dense linear order, cyclic order, betweenness relation, or separation relation on $\Omega = \mathbb{Q}$.
\end{example}

To consider uniform continuity of continuous representations of $\mathscr{E}$, we introduce the following concept.

\begin{definition}
We say that $\mathscr{E} \leqslant \Omega^{\Omega}$ satisfies the \emph{finite retraction property} if for every finite subset $A \subseteq \Omega$, there exists an element $h \in \mathscr{E}$ such that
\[
h|_A = \id_A \quad \text{and} \quad \mathrm{im}(h) \subseteq A.
\]
\end{definition}

For instance, $\mathscr{E} = \Omega^{\Omega}$ clearly satisfies the finite retraction property.

\begin{proposition}
If $\mathscr{E} \leqslant \Omega^{\Omega}$ satisfies the finite retraction property, then every continuous representation of $\mathscr{E}$ is uniformly continuous.
\end{proposition}

\begin{proof}
Fix $v \in V$. By continuity of the action map $f \mapsto f \cdot v$ at the identity, there exists a finite subset $A_v \subseteq \Omega$ such that
\[
f|_{A_v} = \id|_{A_v} \;\Longrightarrow\; f \cdot v = v.
\]
Let $f,g \in \mathscr{E}$ satisfy $f|_{A_v} = g|_{A_v}$. By the finite retraction property, there exists $h \in \mathscr{E}$ such that $h|_{A_v} = \id_{A_v}$ and $\mathrm{im}(h) \subseteq A_v$. Thus $f \circ h = g \circ h$ and $h \cdot v = v$ by the choice of $A_v$. Hence
\[
f \cdot v = f \cdot (h \cdot v) = (f \circ h)\cdot v = (g \circ h)\cdot v = g \cdot (h \cdot v) = g \cdot v,
\]
as desired.
\end{proof}

By this result, we know that every continuous representation of $\mathscr{E} = \Omega^{\Omega}$ is uniformly continuous. However, this is not true for $\mathscr{E}_s$, as shown by the following example.

\begin{example}
Let $\Omega = \mathbb{N}$ and let $\mathscr{E}_s$ denote the monoid of all surjective self-maps of $\Omega$. Let $k$ be a field and let
\[
V = \bigoplus_{n \in \mathbb{N}} k e_n
\]
be the direct sum, endowed with the discrete topology. Define an action of $\mathscr{E}_s$ on $V$ by
\[
f \cdot e_n = e_{\,f(n + f(0))}, \qquad f \in \mathscr{E}_s.
\]

We show that $V$ is a continuous representation of $\mathscr{E}_s$. Fix $f \in \mathscr{E}_s$, $n \in \mathbb{N}$, and set
\[
F_{f,n} = \{0,\, n + f(0)\}.
\]
If $g \in \mathscr{E}_s$ satisfies $g|_{F_{f,n}} = f|_{F_{f,n}}$, then $g(0) = f(0)$ and hence
\[
g(n + g(0)) = g(n + f(0)) = f(n + f(0)).
\]
Thus $g \cdot e_n = f \cdot e_n$. For a finite linear combination
\[
v = \sum_{i=1}^s a_i e_{n_i} \in V,
\]
the union of $F_{f,n_i}$ provides a finite support of $v$, so the action is continuous on all of $V$.

However, $V$ is not a uniformly continuous representation of $\mathscr{E}_s$. To see this, for any finite subset $A \subseteq \Omega$, choose $f \in \mathscr{E}_s$ such that $x = n + f(0) \notin A$. Since $A \cup \{x\}$ is finite and $\Omega$ is infinite, we can find $g \in \mathscr{E}_s$ such that $g|_A = f|_A$ and $g(x) \neq f(x)$. Therefore,
\[
f \cdot e_n = e_{f(x)} \neq e_{g(x)} = g \cdot e_n,
\]
showing that no finite set $A$ can serve as a uniform support of $e_n$. Hence the action is not uniformly continuous.
\end{example}

\begin{remark}
The order-preserving transformation monoid $\mathscr{E} = \End(\mathbb{Q}, \leqslant)$ does not satisfy the finite retraction property. Indeed, for any finite subset $A \subseteq \mathbb{Q}$ with at least two points, there is no endomorphism $h \in \mathscr{E} = \End(\mathbb{Q}, \leqslant)$ such that $h|_A = \id_A$ and $\mathrm{im}(h) \subseteq A$, because any order-preserving map from a dense linear order into a finite set must be locally constant, and thus cannot fix a set with more than one point. Consequently, $\mathscr{E}$ admits continuous representations that are not uniformly continuous.

Let $V = \bigoplus_{q \in \mathbb{Q}} k e_q$ be the discrete direct sum. One can define a continuous action of $\mathscr{E}$ on $V$ by $f \cdot e_n = e_{f(n + f(0))}$. The action is continuous because the value of $f \cdot e_n$ is determined by the values of $f$ on the finite set $\{0, n+f(0)\}$. However, it is not uniformly continuous: for any finite $A \subseteq \mathbb{Q}$, we can choose $f \in \mathscr{E}$ such that $n+f(0)$ lies outside the convex hull of $A$. By perturbing $f$ at this point while preserving the order, we can change the result of the action without changing the restriction $f|_A$.
\end{remark}

\subsection{Application III}

For the third application, we let $\mathscr{E} \leqslant \Omega^{\Omega}$ and assume that the associated category $\C$ of finite subsets satisfies the following condition
\begin{enumerate}
\item[($\sharp$)] every inclusion $\iota_{A, B}: A \hookrightarrow B$ admits a retraction given by a surjective map $p: B \twoheadrightarrow A$.
\end{enumerate}
It is easy to see that the five categories $\OA, \BA, \CA, \SA, \FA$ all satisfy this condition.

\begin{corollary} \label{equivalence 3}
Let $\mathscr{E} \leqslant \Omega^{\Omega}$, and suppose that the associated category $\C$ of finite subsets satisfies condition $(\sharp)$. Then the realization functor $\mathcal{R}$ gives an equivalence
\[
\C \Mod \simeq k\mathscr{E} \Mod^{\mathrm{uc}}
\]
with $\mathcal{L}$ as its inverse.
\end{corollary}

\begin{proof}
Let $A$ be a finite subset and let $S \in J(A)$ be a covering sieve on $\C^{\op}$, which is the same as a left ideal on $\C$. By the description of the Grothendieck topology $J$, the sieve $S$ contains an inclusion $\iota_{A,B} : A \hookrightarrow B$. By condition $(\sharp)$, there is a surjective map $p : B \twoheadrightarrow A$ such that
\[
p \circ \iota_{A,B} = \id_A.
\]
It follows that $\id_A \in S$, hence $S$ is the maximal sieve. Therefore $J$ is the trivial Grothendieck topology, and
\[
\Sh(\C^{\op}, \, J, \, \underline{k}) = \C \Mod.
\]
The conclusion now follows from Theorem~\ref{equivalence 2}.
\end{proof}

Let $(\Omega, \mathcal{R}_G)$ be a highly homogeneous relational structure. Without loss of generality, we may assume $\Omega = \mathbb{Q}$, on which the five highly homogeneous relations have been explicitly described in \cite[Section 13]{BMMN}. Let $\mathscr{E}$ be the monoid of all self-maps of $\mathbb{Q}$ that preserve the highly homogeneous relation on $\Omega$. Then the equivalence described in Corollary~\ref{equivalence 3} provides a convenient framework to study uniformly continuous representations of $\mathscr{E}$.

For a fixed $n \in \mathbb{Z}_+$, we define $\mathscr{M}_n$ to be the set of all relation-preserving maps from $[n]$ to $\mathbb{Q}$, and $\mathscr{N}_n$ the set of all relation-preserving non-injective maps from $[n]$ to $\mathbb{Q}$. The monoid $\mathscr{E}$ has a natural action on them. In addition, since every $f$ in $\mathscr{M}_n$ has a finite image, the same argument used for $\mathscr{E}_0$ in Remark \ref{remind} shows that $k\mathscr{M}_n$ is a uniformly continuous $k\mathscr{E}$-module, so is $k\mathscr{N}_n$ as a submodule.

\begin{proposition}
Let $k$ be a commutative ring and fix $n \in \mathbb{Z}_+$. Let $G_n$ be the automorphism group of the finite subset $[n]$. Then under the equivalence in Corollary \ref{equivalence 3}, we have the following correspondences:
\[
k\mathscr{M}_n \, \longleftrightarrow \, k\C(n, -), \quad k\mathscr{M}_n / k\mathscr{N}_n \, \longleftrightarrow \, \Delta_n.
\]
If furthermore $k$ is a field and $\lambda \in \Irr(G_n)$, then
\[
(k\mathscr{M}_n / k\mathscr{N}_n) \otimes_{kG_n} kG_n e_{\lambda} \, \longleftrightarrow \, \Delta_{n, \lambda}, \quad (k\mathscr{M}_n / k\mathscr{N}_n) \otimes_{kG_n} \lambda \, \longleftrightarrow \, \overline{\Delta}_{n, \lambda}.
\]
\end{proposition}

\begin{proof}
These results follow from the explicit construction in the proof of Proposition \ref{adjunction}.
\end{proof}

One can then use the above result to obtain an explicit construction of irreducible uniformly continuous $k\mathscr{E}$-modules when $k$ is a field.

\begin{corollary}
Let $k$ be a field. The isomorphism classes of irreducible uniformly continuous $k\mathscr{E}$-modules are parameterized by the set of pairs
\[
\{ (n, \lambda) \mid n \in \mathbb{Z}_+, \, \lambda \in \Irr(G_n) \}.
\]
Specifically, let $L_{n, \lambda} = (k\mathscr{M}_n / k\mathscr{N}_n) \otimes_{kG_n} \lambda$. Then irreducible $k\mathscr{E}$-modules are described as follows:
\begin{enumerate}
\item If $\lambda$ is regular, the corresponding irreducible $k\mathscr{E}$-module is isomorphic to $L_{n, \lambda}$.

\item If $\lambda$ is singular, the corresponding irreducible $k\mathscr{E}$-module is isomorphic to the top of $L_{n, \lambda}$.
\end{enumerate}
\end{corollary}

\begin{proof}
The conclusion follows directly from the above proposition and Theorem \ref{classification of irreducibles}.
\end{proof}

\subsection{Application IV}

We end this paper by giving another application beyond the discrete setting. Let $\mathbb{F}_q$ be a finite field with $q$ elements and let
\[
\Omega = \varinjlim_n \mathbb{F}_q^n
\]
be the countably infinite dimensional vector space over $\mathbb{F}_q$. Let $G = \mathrm{GL}(\Omega)$ be the group of all $\mathbb{F}_q$-linear automorphisms of $\Omega$, and let
\[
\mathscr{E} = \End_{\mathbb{F}_q}(\Omega)
\]
be the monoid of all $\mathbb{F}_q$-linear endomorphisms of $\Omega$. Finally, let $\mathscr{C} = \VA_q$ denote the category whose objects are finite dimensional $\mathbb{F}_q$-vector spaces and whose morphisms are $\mathbb{F}_q$-linear maps.

\begin{corollary}
Let $k$ be a commutative ring, and $\mathscr{E}_s$ (resp. $\mathscr{E}_i$) be the submonoid of surjective (resp., injective) elements of $\mathscr{E}$. There are equivalences of categories
\begin{align*}
\VA_q \Mod & \; \simeq \; k\mathscr{E} \Mod^{\mathrm{uc}} \;\simeq\; k\mathscr{E}_s \Mod^{\mathrm{uc}},\\
k\mathscr{E}_i \Mod^{\mathrm{uc}} & \;\simeq\; kG \Mod^{\mathrm{uc}} \;\simeq\; \Sh(\VI_q, \, J_{at}, \, \underline{k}),
\end{align*}
where $\VI_q$ is the wide subcategory of $\VA_q$ with injective morphisms.
\end{corollary}

\begin{proof}
The proofs of Proposition \ref{adjunction}, Theorem \ref{equivalence 2}, and Corollary \ref{equivalence 3} apply in this setting with suitable modifications. The orbit relations on $\Omega$ induced by the action of $G = \mathrm{GL}(\Omega)$ encode linear dependence of tuples. Consequently, the relevant finite subsets are precisely the finite dimensional $\mathbb{F}_q$-subspaces of $\Omega$. It follows that the category constructed from restrictions of elements of $\mathscr{E}$ is precisely $\VA_q$. Moreover, if $f,g \in \mathscr{E}$ coincide on a finite subset $A \subseteq \Omega$, then they coincide on the $\mathbb{F}_q$-linear span $\langle A \rangle$, since $f$ and $g$ are linear maps. Thus uniformly continuous $k\mathscr{E}$-modules can be defined via finite dimensional subspaces. Putting these observations together, we obtain the first equivalence in the first line.

To obtain the second equivalence of the first line, we note that $\mathscr{E}_s$ is dense in $\mathscr{E}$. Indeed, given a linear map $f: A \to B$ between finite dimensional subspaces of $\Omega$, we can extend $f$ to a surjective linear endomorphism of $\Omega$ by choosing complements and extending bases. Hence $\mathscr{E}$ and $\mathscr{E}_s$ have the same uniformly continuous representation theory.

As we explained before, $G$ is a dense submonoid of $\mathscr{E}_i$. The first equivalence of the second line then follows. The second one is proved in \cite{DLLX}.
\end{proof}

When $k$ is a field of characteristic 0, it has been shown in \cite{Kuhn} that
\[
\VA_q \Mod \simeq \prod_{n \geqslant 0} k\mathrm{GL}_n(\mathbb{F}_q) \Mod,
\]
and hence is semisimple.

These equivalences yield explicit classifications of irreducible uniformly continuous representations of these monoids.

\begin{corollary}
Let $k$ be a field of characteristic $0$. Then irreducible uniformly continuous representations of $k\mathscr{E}$, $k\mathscr{E}_s$, $k\mathscr{E}_i$, and $kG$ are classified by irreducible representations of the finite groups $\mathrm{GL}_n(\mathbb{F}_q)$ $(n \geqslant 0)$. More precisely,
\[
\Irr (k \mathscr{E} \Mod^{\mathrm{uc}}) \; \cong \; \Irr (k \mathscr{E}_s \Mod^{\mathrm{uc}}) \; \cong \; \Irr (k \mathscr{E}_i \Mod^{\mathrm{uc}}) \; \cong \; \Irr (kG \Mod^{\mathrm{uc}}) \, \cong \, \coprod_{n \geqslant 0} \Irr_k (\mathrm{GL}_n (\mathbb{F}_q)).
\]
\end{corollary}

\begin{proof}
By the previous two corollaries, it suffices to show
\[
\Irr(\VA_q \Mod) \, \cong \, \coprod_{n \geqslant 0} \Irr_k (\mathrm{GL}_n (\mathbb{F}_q)) \, \cong \, \Irr (kG \Mod^{\mathrm{uc}}).
\]
For the first isomorphism, we note that irreducible $\VA_q$-modules over a field of characteristic $0$ are classified and correspond to irreducible representations of the groups $\mathrm{GL}_n(\mathbb{F}_q)$ for $n \geqslant 0$; see \cite{Kuhn}. The second isomorphism is proved in \cite{DLLX}.
\end{proof}

\end{document}